\newif\ifpreprint
\preprinttrue % enable this if using the standard latex `article` class
\ifdefined\ispreprint
  \preprinttrue
\fi
\ifdefined\isjournal
  \preprintfalse
\fi
\RequirePackage{fix-cm}
\ifpreprint
  \documentclass{article}
  \usepackage[colorlinks=true,linkcolor=black,citecolor=black]{hyperref}
  \usepackage{mathrsfs}
  \usepackage{mathtools}
  \usepackage{authblk}
  \usepackage{xcolor}
  \usepackage{bbm}
  \usepackage{amsmath,amssymb,amsfonts,amsthm,bm}
  \usepackage{dsfont}
  \usepackage[a4paper]{geometry}
\else
  \documentclass{m2an}  
  \usepackage[colorlinks=true,linkcolor=black,citecolor=black]{hyperref}
  \usepackage{amsmath,amssymb}
  \usepackage{dsfont}
  \usepackage{mathrsfs}
  \usepackage{mathtools}
\fi

\usepackage[numbers,sort&compress]{natbib}
\usepackage[latin1]{inputenc}

\allowdisplaybreaks

\newcommand{\Z}{\mathbb{Z}} % integers
\newcommand{\R}{\mathbb{R}} % reals
\newcommand{\N}{\mathbb{N}} % natural numbers {1, 2, ...}
\newcommand{\setu}{\mathfrak{u}}
\newcommand{\setv}{\mathfrak{v}}
\newcommand{\setw}{\mathfrak{w}}
\newcommand{\setz}{\mathfrak{z}}
\newcommand{\setU}{\mathfrak{U}}

\newcommand{\II}{\mathcal{I}} % infinite-dimensional integral
\newcommand{\HH}{\mathcal{H}} % infinite-variate RKHS
\newcommand{\KK}{\mathcal{K}} % reproducing kernel of infinite-variate RKHS
\newcommand{\FF}{\mathcal{F}} % infinite-variate function
\newcommand{\GG}{\mathcal{G}} % infinite-variate function
\newcommand{\YY}{\mathcal{Y}} % domain of infinite-variate RKHS

\DeclareMathOperator{\wal}{wal}

\newcommand{\rd}{\,\mathrm{d}} % differential symbol for use in integrals
\newcommand{\rmd}{\mathrm{d}} % differential symbol without thin space in front

\newcommand{\bszero}{\boldsymbol{0}} % vector of zeros
\newcommand{\bsb}{\boldsymbol{b}}    % vector b
\newcommand{\bsell}{{\boldsymbol{\ell}}}    % vector \ell
\newcommand{\bsq}{{\boldsymbol{q}}}    % vector q
\newcommand{\bst}{\boldsymbol{t}}    % vector t
\newcommand{\bsh}{\boldsymbol{h}}    % vector h
\newcommand{\bsk}{{\boldsymbol{k}}}    % vector k
\newcommand{\bsx}{{\boldsymbol{x}}}    % vector x
\newcommand{\bsy}{{\boldsymbol{y}}}    % vector y
\newcommand{\bsv}{{\boldsymbol{v}}}    % vector v

\newcommand{\bsDelta}{\boldsymbol{\Delta}}    % vector \Delta
\newcommand{\bsalpha}{{\boldsymbol{\alpha}}}    % vector \alpha
\newcommand{\bsgamma}{\boldsymbol{\gamma}}    % vector \alpha
\newcommand{\bsnu}{\boldsymbol{\nu}}    % vector \nu
\newcommand{\bsomega}{{\boldsymbol{\omega}}}    % vector \omega
\newcommand{\bstau}{\boldsymbol{\tau}}			%vector \tau
\newcommand{\bsT}{\boldsymbol{T}}			%vector T
\newcommand{\E}{\mathbb{E}}

\newcommand{\calY}{\mathcal{Y}}

\newcommand{\calN}{\mathcal{N}} % normal distribution

\newcommand{\x}{\chi} % formal variable for Laurent expansions

\newcommand{\e}{\mathrm{e}}
\newcommand{\wor}{{\mathrm{wor}}}

\newcommand{\cost}{\mathrm{cost}}

\DeclareSymbolFont{bbold}{U}{bbold}{m}{n}
\DeclareSymbolFontAlphabet{\mathbbold}{bbold}

\newcommand{\supp}{\mathop{\mathrm{supp}}}
\newcommand{\esssup}{\mathop{\mathrm{ess\,sup}}}
\newcommand{\essinf}{\mathop\mathrm{ess\,inf}} 

\DeclareMathOperator{\tr}{tr} % truncation for polynomials for polynomial lattice rules

\newcommand{\MDFEM}{{\mathrm{MDFEM}}}
\newcommand{\QMCFEM}{{\mathrm{QMCFEM}}}

\newcommand{\dd}{{d'}}
\newcommand{\ddelta}{{\delta'}}

\theoremstyle{plain}
  \newtheorem{theorem}{Theorem}
  \newtheorem{proposition}{Proposition}
  \newtheorem{lemma}{Lemma}
  
\theoremstyle{definition}
  \newtheorem{definition}{Definition}
  
\theoremstyle{remark}
  \newtheorem{remark}{Remark}

\newcommand{\RefDef}[1]{Definition~\textup{\ref{#1}}}
\newcommand{\RefSec}[1]{Section~\textup{\ref{#1}}}
\newcommand{\RefApp}[1]{Appendix~\textup{\ref{#1}}}
\newcommand{\RefThm}[1]{Theorem~\textup{\ref{#1}}}
\newcommand{\RefThmTwo}[2]{Theorems~\textup{\ref{#1}} and~\textup{\ref{#2}}}

\newcommand{\RefProp}[1]{Proposition~\textup{\ref{#1}}}
\newcommand{\RefPropTwo}[2]{Propositions~\textup{\ref{#1}} and~\textup{\ref{#2}}}
\newcommand{\RefPropRange}[2]{Propositions~\textup{\ref{#1}}--\textup{\ref{#2}}}
\newcommand{\RefLem}[1]{Lemma~\textup{\ref{#1}}}

\newcommand{\RefRem}[1]{Remark~\textup{\ref{#1}}}
\newcommand{\RefRemTwo}[2]{Remarks~\textup{\ref{#1}} and~\textup{\ref{#2}}}

\newcommand{\theAbstract}{
We introduce the \emph{multivariate decomposition finite element method} (MDFEM) for elliptic PDEs with lognormal diffusion coefficients, that is, when the diffusion coefficient has the form $a=\exp(Z)$ where $Z$ is a Gaussian random field defined by an infinite series expansion $Z(\bsy) = \sum_{j \ge 1} y_j \, \phi_j$ with $y_j \sim \calN(0,1)$ and a given sequence of functions $\{\phi_j\}_{j \ge 1}$. We use the MDFEM to approximate the expected value of a linear functional of the solution of the PDE which is an infinite-dimensional integral over the parameter space. The proposed algorithm uses the \emph{multivariate decomposition method} (MDM) to compute the infinite-dimensional integral by a decomposition into finite-dimensional integrals, which we resolve using \emph{quasi-Monte Carlo} (QMC) methods, and for which we use the \emph{finite element method} (FEM) to solve different instances of the PDE.

We develop higher-order quasi-Monte Carlo rules for integration over the finite-di\-men\-si\-onal Euclidean space with respect to the Gaussian distribution by use of a truncation strategy. By linear transformations of interlaced polynomial lattice rules from the unit cube to a multivariate box of the Euclidean space we achieve higher-order convergence rates for functions belonging to a class of \emph{anchored Gaussian Sobolev spaces} while taking into account the truncation error. These cubature rules are then used in the MDFEM algorithm.

Under appropriate conditions, the MDFEM achieves higher-order convergence rates in terms of error versus cost, i.e., to achieve an accuracy of $O(\epsilon)$ the computational cost is $O(\epsilon^{-1/\lambda-\dd/\lambda}) = O(\epsilon^{-(p^* + \dd/\tau)/(1-p^*)})$ where $\epsilon^{-1/\lambda}$ and $\epsilon^{-\dd/\lambda}$ are respectively the cost of the quasi-Monte Carlo cubature and the finite element approximations, with $\dd = d \, (1+\ddelta)$ for some $\ddelta \ge 0$ and $d$ the physical dimension, and $0 < p^* \le (2 + \dd/\tau)^{-1}$ is a parameter representing the sparsity of $\{\phi_j\}_{j \ge 1}$.
}

\begin{document}

\title{MDFEM: Multivariate decomposition finite element method for elliptic PDEs with lognormal diffusion coefficients \\ using higher-order QMC and FEM}

\ifpreprint
  \author[1]{Dong T.\,P. Nguyen}
  \author[2]{Dirk Nuyens}
  \affil[1]{dong.nguyen@hcmut.edu.vn, Faculty of Computer Science and Engineering, \authorcr Ho Chi Minh City University of Technology, VNU-HCM, Vietnam}
  \affil[2]{dirk.nuyens@cs.kuleuven.be, Department of Computer Science, \authorcr KU Leuven, Celestijnenlaan 200A  box 2402, B-3001 Leuven, Belgium}
\else
  \author{Dong T.\,P. Nguyen}\address{dong.nguyen@hcmut.edu.vn, Faculty of Computer Science and Engineering, Ho Chi Minh City University of Technology, VNU-HCM, Vietnam} 
  \author{Dirk Nuyens}\address{dirk.nuyens@cs.kuleuven.be, Department of Computer Science, KU Leuven, Celestijnenlaan 200A  box 2402, B-3001 Leuven, Belgium}
  
  \begin{abstract}
  \theAbstract
  \end{abstract}
\fi

\date{June 7, 2021}

\ifpreprint
\else
\fi

\maketitle

\ifpreprint
  \begin{abstract}
  \theAbstract
  \end{abstract}
\fi

\section{Introduction}\label{sec:introduction}

In this paper we are concerned with the application of higher-order quasi-Monte Carlo (QMC) rules and multivariate decomposition methods (MDM) to elliptic PDEs with random diffusion coefficients. We focus on the lognormal diffusion coefficient, the logarithm of which is a \emph{Gaussian random field}. The goal is to compute the expected value of some functional of the solution. 
This method was motivated by the need for new techniques for elliptic PDEs with smooth lognormal diffusion coefficients where the classical QMC approaches can only achieve first-order convergence, see, e.g., \cite{GKNSSS15,GKNSS18,HS19,HS19ML}.

The MDFEM was already analysed in the case of a uniform diffusion coefficient in~\cite{NN21}, but in this case higher-order QMC rules are readily available for integration over the unit cube.
For lognormal diffusions we cope with a more challenging problem since the expectation is taken with respect to the Gaussian distribution over an unbounded domain. Consequently, existing higher-order QMC algorithms are not directly applicable. To solve this problem we propose using a truncation method recently developed in~\cite{DILP18}, see also~\cite{NN17}. By exploiting the fast decay of the Gaussian distribution toward infinity, the Euclidean domain is truncated and the resulting integral is transformed to the unit cube using a linear transformation where suitable higher-order QMC rules can be applied. The proposed algorithm allows us to achieve higher-order convergence for sufficiently smooth integrands. 

Let $D \subset \R^d$ be a bounded polygonal domain in $\R^d$, with typically $d=1,2$ or $3$, with boundary $\partial D$. We consider the following elliptic Dirichlet problem
\begin{align}\label{eq:PDE}
  -\nabla \cdot (a(\bsx, \bsy) \, \nabla u(\bsx, \bsy))
  &=
  f(\bsx)
  ,
  &&\text{for $\bsx$ in $D$,}
  \\
  u(\bsx, \bsy)
  &=
  0
  ,
  &&\text{for $\bsx$ on $\partial D$}
  \notag
  .
\end{align}
Here, the gradient operator $\nabla$ is taken with respect to $\bsx$ and $a:  D \times \Omega^{\N} \to \R$ for some $\Omega \subseteq \R$. 

We consider the case when $\bsy = \{y_j\}_{j \ge 1}$ is a sequence of parameters distributed on $\R^\N$ according to the product Gaussian measure $\mu = \bigotimes_{j \ge 1} \calN(0,1)$, and the diffusion coefficient takes the form
\begin{align}\label{eq:diffusion-a}
  a(\bsx, \bsy)
  &:=
  \exp\left(Z(\bsx, \bsy)\right)
  ,
\end{align}
with 
\begin{align}\label{eq:randomfield-Z}
  Z(\bsx, \bsy)
  &:=
  \sum_{j \ge 1} y_j \, \phi_j(\bsx)
  ,
  \qquad y_j \in \Omega = \R
  , \qquad
  y_j \sim \calN(0,1)
  ,
\end{align}
where $\{\phi_j\}_{j \ge 1}$  is a suitable system of real-valued, bounded, and measurable functions. 

Let us denote the natural numbers by $\N := \{1,2,\ldots\}$, and $\N_0 := \{0, 1, 2, \ldots\}$. 
For any $s \in \N$ we use the shorthand notation $\{1:s\}$ to denote the set of indices $\{1,2,\ldots,s\}$.
Let $G$ be a linear and bounded functional of  the solution $u$. We are interested in computing the expected value of $G(u)$ with respect to the probability distribution $\mu$, i.e., 
\begin{align}\label{eq:QoI}
  \E[G(u)]
  =
  \II(G(u))
  &:=
  \int_{\R^\N} G(u(\cdot, \bsy)) \rd\mu(\bsy)
  \\ \notag
  &:=
  \lim_{s \to \infty} 
  \int_{\R^s} G(u(\cdot,y_1, y_2, \ldots, y_s, 0, 0, \ldots))\rd\mu(\bsy_{\{1:s\}})
,
\end{align} 
where for $\setu \subseteq \N$
\begin{align*}
  \rmd\mu(\bsy_\setu)
  &:=
  \prod_{j\in\setu} \rho(y_j) \rd y_j
  =
  \rho_\setu(\bsy_\setu) \rd\bsy_\setu
  ,
  &
  \rho(y)
  &:=
  \frac{\exp(-y^2/2)}{\sqrt{2 \pi}}
  ,
\end{align*}
with $\rho_\setu(\bsy_\setu) := \prod_{j\in\setu} \rho(y_j)$.
We note that, depending on what is most natural, we write $\rmd\mu(\bsy_\setu)$ or $\rho_\setu(\bsy_\setu) \rd\bsy_\setu$ with the understanding that the product probability measure is always on all the variables of the integral.

The \emph{weak formulation} of problem \eqref{eq:PDE} is to find for a given $\bsy \in \Omega^\N$ the solution $u(\cdot, \bsy) \in V := H_0^1(D)$ such that
\begin{align}\label{eq:PDE-weak-form}
  \int_D a(\bsx, \bsy) \, \nabla u(\bsx, \bsy) \cdot \nabla v(\bsx) \rd\bsx 
  &=
  \int_D f(\bsx) \, v(\bsx) \rd\bsx
  ,
  \qquad
  \forall v \in V
  .
\end{align}
The space $V = H_0^1(D)$ is equipped with the norm $\|v\|_V := \|\nabla v\|_{L^2(D)}$.

Under some assumptions on the system $\{\phi_j\}_{j \ge 1}$ we have existence, uniqueness and an \emph{a priori} estimate of the solution of the weak formulation by means of the Lax--Milgram lemma.
For this we need to show a lower bound and an upper bound on $a(\bsx,\bsy)$ for $\bsx \in D$, but since we have a lognormal field, it might be that $a(\bsx,\bsy)$ is not bounded for certain values of $\bsy \in \R^\N$.
However, under the following assumptions, which are standard, see, e.g., \cite{BCDM17, BCDS17, HS19, Kaz18}, we can claim lower and upper bounds to hold for $\mu$~almost every~$\bsy$.
Let us first define, for some given space $X$ and $p \in [1,\infty)$, the space $L_{p,\rho}(\R^\N;X)$ of all strongly measurable mappings $v: \R^\N \to X$ such that
\begin{align*}
  \|v\|_{L_{p,\rho}(\R^\N;X)}
  &:=
  \left( \int_{\R^\N} \|v\|_X^p  \rd\mu(\bsy) \right)^{1/p}
  =
  \left( \int_{\R^\N} \|v\|_X^p  \, \prod_{j\ge1} \rho(y_j) \rd\bsy \right)^{1/p}
  < \infty
  .
\end{align*}
The following result is implied by~\cite[Section~2]{BCDM17} and \cite[Theorem~2, Proposition~3 and Corollary~6]{HS19}.

\begin{proposition}\label{prop:Lax-Milgram}
  If there exists a positive sequence $\{b_j\}_{j\ge 1}$, with $0 < b_j \le 1$ for all $j$, such that
  \begin{align}\label{eq:condition:kappa}
    \kappa
    &:=
    \left\| \sum_{j \ge 1} \frac{|\phi_j|}{b_j} \right\|_{L^\infty(D)}
    =
    \sup_{\bsx \in D} \sum_{j\ge1} \frac{|\phi_j(\bsx)|}{b_j}
    <
    \infty
    ,
  \end{align}
  and 
  \begin{align}\label{eq:condition:pstar-summability}
    \{b_j\}_{j\ge 1} \in \ell^{p^*}(\N) \text{ for some } p^* \in (0,\infty)
    ,
  \end{align}
  then it holds that the Gaussian field $Z$ and the lognormal field $a = \exp(Z)$ are elements of $L_{p,\rho}(\R^\N;L^\infty(D))$ for every $p \in [1, \infty)$.
  Therefore, for $\mu$~almost every $\bsy \in \R^\N$ we have
  \begin{align}
    \label{eq:amax} 
    a_{\max}(\bsy)
    &:=
    \esssup_{\bsx \in D} |a(\bsx, \bsy)|
    =
    \|a(\cdot, \bsy)\|_{L^\infty(D)}
    <
    \infty
    ,
    \\
    \label{eq:amin}
    a_{\min}(\bsy)
    &:=
    \essinf_{\bsx \in D} a(\bsx, \bsy)
    \ge
    \exp(-\|Z(\cdot,\bsy)\|_{L^\infty(D)})
    >
    0
    ,
  \end{align}
  and, by the Lax--Milgram lemma, the solution $u(\cdot,\bsy)$ then exists and is unique and satisfies
  \begin{align}\label{eq:LM}
    \|u(\cdot, \bsy)\|_V
    &\le
    \frac{1}{a_{\min}(\bsy)} \, \|f\|_{V^*}
    ,
  \end{align}
  and additionally, $a^{-1}_{\min}$ is an element of $L_{p,\rho}(\R^\N)$ for every $p \in [1, \infty)$, such that
  \begin{align}\label{eq:LM-Bochner}
    \|u\|_{L_{p,\rho}(\R^\N; V)}
    \le
    \left\|\frac1{a_{\min}}\right\|_{L_{p,\rho}(\R^\N)} \|f\|_{V^*}
    <
    \infty
    ,
  \end{align}
  for any $f \in V^*$ and $p \in [1, \infty)$.
\end{proposition}

\RefProp{prop:Lax-Milgram} shows that for $\mu$~almost every $\bsy \in \R^\N$ there exists a unique solution to~\eqref{eq:PDE-weak-form} where $\mu$ is the product Gaussian measure, and this stems from the measurability of $Z$ and $a = \exp(Z)$ in $L_{p,\rho}(\R^\N,L^\infty(D))$, in particular for $p=1$.
\label{pp:mu-almost-everywhere}To make the ``for $\mu$~almost every~$\bsy$'' more tangible we note that if we define the set $\Xi := \{ \bsy \in \R^\N : \sup_{j \ge 1} |y_j| \, b_j < \infty \} \subseteq \R^\N$, then for any $\bsy \in \Xi$
\begin{align*}
  \| Z(\cdot, \bsy) \|_{L^\infty(D)}
  &=
  \sup_{\bsx \in D} \left| \sum_{j\ge1} y_j \, \phi_j(\bsx) \right|
  \\
  &\le
  \sup_{\bsx \in D} \sum_{j\ge1} |y_j| \, \frac{b_j}{b_j} \, |\phi_j(\bsx)|
  \le
  \left( \sup_{j \ge 1} |y_j| \, b_j \right) \left( \sup_{\bsx \in D} \sum_{j\ge1} \frac{|\phi_j(\bsx)|}{b_j} \right)
  <
  \infty
  ,
\end{align*}
under condition~\eqref{eq:condition:kappa}. From here we can show the wanted results in \RefProp{prop:Lax-Milgram}: $a_{\max}(\bsy) < \infty$ and $a_{\min}(\bsy) > 0$ for all $\bsy \in \Xi$.
It can be shown, see, e.g., \cite{BCDM17}, that the set $\Xi$ has full measure $\mu(\Xi) = 1$ for $\{b_j\}_{j\ge1} \in \ell^{p^*}$, hence ``for $\mu$ almost every $\bsy$''.

To approximate the solution of~\eqref{eq:PDE-weak-form} the standard approach is to truncate the infinite series~\eqref{eq:randomfield-Z} to $s$ terms.
Let us denote the random field truncated to the first $s$ terms by $Z_s$ and the so-obtained truncated lognormal field by $a_s$, i.e.,
\begin{align}\label{eq:truncate-s}
  a_s(\bsx, \bsy)
  &:=
  a(\bsx, \bsy_{\{1:s\}})
  =
  \exp(Z(\bsx, \bsy_{\{1:s\}}))
  ,
  &
  Z_s(\bsx, \bsy)
  &:=
  Z(\bsx, \bsy_{\{1:s\}})
  =
  \sum_{j = 1}^s y_j \, \phi_j(\bsx)
  ,
\end{align}
and by $u_s$ the solution to~\eqref{eq:PDE-weak-form} with $a = a_s$, then it is shown, under the assumptions of \RefProp{prop:Lax-Milgram}, in \cite{BCDM17,HS19} that $Z_s \to Z$ in $L^\infty(D)$, $a_s \to a$ in $L^\infty(D)$ and $u_s \to u$ in $V$ when $s \to \infty$ for almost every $\bsy \in \R^\N$.
We note that all statements of \RefProp{prop:Lax-Milgram} also hold for the truncation to $s$ terms, see, e.g., \cite{HS19}.

\smallskip

To compute $\mathbb{E}[G(u)]$, see~\eqref{eq:QoI}, we cope with three computational challenges: the infinite number of variables of the integrand, the unboundedness of the integration domain, and the integrand involves solutions of a PDE.
Let us discuss how to solve these problems in more detail.

First, to approximate the infinite-dimensional integral we will make use of the \emph{multivariate decomposition method} which originated from the \emph{changing dimension method}, see, e.g.,~\cite{KSWW10, PW11, KNPSW17}. The goal is to decompose the infinite-dimensional problem into multiple finite-dimensional ones.
Let us illustrate the MDM for calculating the integral
\begin{align*}
  \II(\FF)
  =
  \int_{\R^\N} \FF(\bsy) \rd\mu(\bsy)
  &=
  \lim_{s \to \infty} \int_{\R^s} \FF(y_1,\ldots,y_s,0,0,\ldots) \rd\mu(y_1,\ldots,y_s)
  .
\end{align*}
We use the \emph{anchored decomposition} \cite{KSWW10b}, which, for any finite $s$, states
\begin{align}\label{eq:anchored-decomposition-F}
  \FF(y_1,\ldots,y_s,0,0,\ldots)
  &=
  \sum_{\setu \subseteq \{1:s\}} F_\setu(\bsy_\setu)
  ,
  &
  \text{with } \quad
  F_\setu(\bsy_\setu)
  &=
  \sum_{\setv \subseteq \setu} (-1)^{|\setu|-|\setv|} \FF(\bsy_\setv)
  ,
\end{align}
where $\bsy_\setu$ takes the values $y_j$ for $j \in \setu$ and $0$ for those $j \notin \setu$.
(Conditions and properties for $\FF$ and the $F_\setu$ will be provided later, here we just want to illustrate the concept.)
We note that $\FF(\bsy_\setv)$ is often written as $\FF([\bsy_\setv; \bszero]_\setv)$ or $\FF((\bsy_\setv;\bszero))$ or $\FF((\bsy_\setv,\bszero_{-\setv}))$ or $\FF(\bsy_\setv,\bszero_{-\setv})$, but we use our convention to not clutter further expressions more than necessary.
Thus, to obtain the projections $F_\setu$ we combine $\setv$-truncated evaluations of the function $\FF$ in points where all components outside of $\setv \subseteq \setu$ are set to zero.
Taking the limit, we write
\begin{align}\label{eq:MDM-decomposition-F}
  \FF(y_1,y_2,\ldots)
  &=
  \sum_{|\setu| < \infty} F_\setu(\bsy_\setu)
  :=
  \lim_{s \to \infty} \sum_{\setu \subseteq \{1:s\}} F_\setu(\bsy_\setu)
  .
\end{align}
Instead of just truncating to the first $s$ dimensions, this decomposition allows us to be more selective and approximate the function $\FF$, or its integral $\II(\FF)$ in our case, up to a certain accuracy by considering a set of important subsets $\setU_\epsilon \subset \N^\N$, called the \emph{active set}, as follows
\begin{align*}
  \FF_{\setU_\epsilon}(\bsy)
  &=
  \sum_{\setu \in \setU_\epsilon} F_\setu(\bsy_\setu)
  .
\end{align*}
This form is interesting when a bounded linear operation on $\FF$, which would be costly in the number of ``active variables'', could be approximated on $\FF_{\setU_\epsilon}$ by wrapping the linear operation inside the sum where hopefully only sets $\setu$ of small cardinality appear.
Note that a plain truncation to the first $s$ dimensions could be represented like this as well, but then the active set would contain all subsets of $\{1:s\}$, including the set $\{1:s\}$ itself, and it would be more efficient to just use $\FF(y_1,\ldots,y_s,0,\ldots)$ in this case.
The strength of the MDM is exactly that the active set can be chosen as to satisfy
\begin{align*}
  \left| \II(\FF) - \sum_{\setu \in \setU_\epsilon} Q_{\setu,n_\setu}(F_\setu(\bsy_\setu)) \right|
  &\le
  \epsilon
  ,
\end{align*}
for a certain error request $\epsilon > 0$ and appropriate cubature rules $Q_{\setu,n_\setu}$.
By assessing the relative contribution of specific sets of variables, for a given desired error, the MDM will decide which $\setu$ to include in the \emph{active set} to approximate the infinite MDM sum. In \RefProp{prop:cardinality-active-set} we will show that the sets in the active set have relatively low cardinalities, provided certain conditions on $\{\phi_j\}_{j \ge 1}$ are satisfied, and hence only relatively low-dimensional problems remain, which can be solved at small cost.
As a comparison, for a certain error request $\epsilon$, the truncation strategy might need a truncation dimension of $s=1907$, since $y_{1907}$ still has large enough influence, thus needing a $1907$-dimensional cubature rule, while the analysis of the MDM approach might show it is sufficient to include the sets $\{1,1907\}$ and $\{2,1907\}$, along with a lot of other small sets, but never surpassing more than $7$-dimensional subproblems, see \cite[Table~1]{GKNW18} for examples.
This is particularly useful when error bounds grow exponentially with the number of dimensions.

Next, to compute integrals over the Euclidean space, we exploit the fast decay of the Gaussian distribution to truncate the unbounded domain to bounded boxes. We then use a linear transformation to map the truncated integral into the unit cube. Finally, we apply existing higher-order quasi-Monte Carlo rules, in particular we apply \emph{interlaced polynomial lattice rules}, see, e.g., \cite{God15}, to approximate the truncated integral. Interlaced polynomial lattice rules have been used before in the PDE context to achieve higher-order convergence but with uniform diffusion in \cite{DKLNS14}.
In contrast, most existing QMC methods for integration with respect to the normal density, and hence most existing QMC based methods for estimating expected values in the PDE context with lognormal diffusion, map the integral to the unit cube by using the inverse of the Gaussian cumulative distribution function and then use \emph{randomly shifted lattice rules} to approximate this integral, see, e.g., \cite{KSWWa10,NK14}, and \cite{GKNSSS15,GKNSS18,HS19,HS19ML,KN16}.
The aim of all these methods is to obtain dimension-independent convergence by making use of weighted function spaces.
However, since using the inverse of the Gaussian cumulative distribution function might damage the smoothness of the integrand, they make use of first order QMC rules, namely, randomly shifted lattice rules, and limit themselves to first order convergence.
The proposed QMC method in this paper avoids damaging the smoothness of the integrand by using a truncation and mapping strategy.
As a result, it allows us to achieve higher-order convergence rates for sufficiently smooth integrands on the Euclidean space, since they retain their smoothness on the unit cube.
We do not include weights in our function spaces and therefore the constants depend exponentially on the number of dimensions.
This is by design, since by using the MDM we only need to tackle relatively low-dimensional integrals instead of approximating the truncated high-dimensional integral directly. The importance of the subproblems will determine which ones to include and how well to approximate them.

Lastly, for each variable $\bsy$ sampled by the QMC method the original stochastic PDE becomes a deterministic one. To solve each such problem we use the finite element method (FEM).
We call the combination of the \emph{multivariate decomposition method} (MDM) with the \emph{finite element method} (FEM) the \emph{multivariate decomposition finite element method} or MDFEM in short.

\smallskip

For the further discussion we need some notation and properties of the anchored decomposition.
In general our multi-indices and variables are infinite-dimensional, e.g., $\bsomega \in \N^\N$ and $\bsy \in \R^\N$. By a subscript $\setu \subseteq \N$ we mean to only retain those components in $\setu$ and set the other components to zero.
We extend this notation to sets and for a set $\mathcal{A}$ we define
\begin{align*}
  \mathcal{A}^\N_\setu
  &:=
  \{ \bsv \in (\mathcal{A} \cup \{0\})^\N : v_j \in \mathcal{A} \text{ when } j \in \setu \text{ and } v_j = 0 \text{ when } j \notin \setu \}
  .
\end{align*}
In the case that we want the components of a multi-index or vector to range over all values of a set $\mathcal{A}$ for the components in $\setu$, and be zero otherwise, we write $\bsy \in \mathcal{A}^\N_\setu$, and in this case $\bsy_\setu = \bsy$.
When we only need the components $y_j$ for $j \in \setu$ we then write, with a slight abuse of notation, $\bsy_\setu \in \mathcal{A}^{|\setu|}$ or even $\bsy_\setu \in \mathcal{A}^\N_\setu$ instead of $\bsy \in \mathcal{A}^\N_\setu$.
If there is possible confusion then we will indicate the meaning explicitly or use a more explicit notation.
Note that, contrary to typical usage, we occasionally also write $\bsomega_\setu \in \N_0^{|\setu|}$, with $\omega_j \in \N_0 = \{0,1,2,\ldots\}$ for $j \in \setu$, % and zero otherwise, 
which means that we are working with the indices in $\setu$, but also allow $\omega_j = 0$ for $j \in \setu$ in this case.

We write $\partial^{\bsomega_\setu}_{\bsy_\setu} :=  \partial^{|\bsomega_\setu|} / \prod_{j \in \setu} \partial^{\omega_j}_{y_j}$ with $|\bsomega_\setu| := \sum_{j \in \setu} \omega_j$ and by $(\partial^{\bsomega_\setu}_{\bsy_\setu} F)(\bsy)$ we mean the value of such a partial derivative at~$\bsy$.
Whenever it is clear w.r.t.\ which variables the derivatives are taken then we just write $F^{(\bsomega_\setu)}$ or $F^{(\bsomega)}$.

\begin{lemma}\label{lem:anchored-decomposition}
	Given the anchored decomposition of a function $\FF$ (with anchor at zero) by~\eqref{eq:anchored-decomposition-F} and~\eqref{eq:MDM-decomposition-F}, then $F_\setu$ depends only on the variables listed in $\setu$ and satisfies
	\begin{align}\label{eq:prop1}
	    F_\setu(\bsy_\setu) = 0
		\quad \text{when } \exists j \in \setu : y_j = 0
		.
	\end{align}
	Furthermore, if $\FF$ has continuous partial derivatives up to $\partial^{\bsomega_\setu}_{\bsy_\setu}$ for some $\bsomega_\setu \in \N_0^{|\setu|}$
	then
	\begin{align}\label{eq:prop2}
		(\partial^{\bsomega_\setu}_{\bsy_\setu} F_\setu)(\bsy_\setu) 
		&=
		(\partial^{\bsomega_\setu}_{\bsy_\setu} \FF(\cdot_\setu))(\bsy_\setu)
		\quad \text{when } \forall j \in \setu : \omega_j \ge 1
		,
	\end{align}
	and
	\begin{align}\label{eq:prop3}
	  (\partial^{\bsomega_\setu}_{\bsy_\setu} F_\setu)(\bsy_\setu)
	  &=
	  0
	  \quad \text{when } \exists j \in \setu : y_j = 0 \text{ and } \omega_j = 0
	  .
	\end{align}
\end{lemma}
\begin{proof}
    See \RefApp{app:proof-anchored-props}.
\end{proof}

In the following the function $\FF$ will actually be $G(u)$ and the projections $F_\setu$ will be $G(u_\setu)$.
We remark that the properties of \RefLem{lem:anchored-decomposition} will hold for the decompositions of both $u$ and $G(u)$.
When we describe function spaces for the functions $F_\setu$ we will identify them with $|\setu|$-variate functions on $\R^{|\setu|}$.

\smallskip

The structure of this paper is as follows. \RefSec{sec:MDFEM} presents the main steps of the MDFEM.
\RefSec{sec:HOQMC} introduces higher-order QMC rules for multivariate integration over the Euclidean space with respect to the Gaussian distribution based on a truncation strategy. A novel \emph{anchored Gaussian Sobolev function space} is introduced and QMC rules are developed for that specific space.
\RefSec{sec:parametric-regularity} discusses the parametric regularity of the solution of the PDE and shows that $u_\setu$ and $G(u_\setu)$ live in a Bochner space based on the anchored Gaussian Sobolev space, and provides bounds on their norms.
\RefSec{sec:FE-approximation-error} states our assumptions on the higher-order convergence of the FE approximations.
\RefSec{sec:mainresult} presents the main contribution of this paper where the cost model, the construction and the complexity of the MDFEM algorithm are presented. Finally, a comparison with two existing methods, the QMCFEM \cite{HS19} and MLQMCFEM \cite{HS19ML}, shows the benefit of the MDFEM.

\section{Applying the MDM to PDEs}\label{sec:MDFEM}

We now explain the ingredients of the MDFEM.
Our main building block will be $u(\cdot, \bsy_\setv)$.
By $u(\cdot, \bsy_\setv)$ we mean the weak solution of~\eqref{eq:PDE} with $\bsy = \bsy_\setv$, and we call this the \emph{$\setv$-truncated solution}.
That is, for given $\bsy_\setv \in \R^\N_\setv = \{ \bsy \in \R^\N : y_j = 0 \text{ for } j \notin \setv \}$, with $|\setv| < \infty$, the $\setv$-truncated solution $u(\cdot, \bsy_\setv) \in V$ is the solution of the problem
\begin{align}\label{eq:PDE-weak-form-v}
  \int_D a(\bsx, \bsy_\setv) \, \nabla u(\bsx, \bsy_\setv) \cdot \nabla v(\bsx) \rd\bsx 
  &=
  \int_D f(\bsx) \, v(\bsx) \rd\bsx
  ,
  \qquad
  \forall v \in V
  ,
\end{align}
where
\begin{align}\label{eq:a-Z-truncated}
  a(\cdot, \bsy_\setv)
  &=
  \exp(Z(\cdot, \bsy_\setv))
  ,
  &
  Z(\cdot, \bsy_\setv)
  &=
  \sum_{j \in \setv} y_j \, \phi_j
  .
\end{align}
This is not unlike the commonly used truncation strategy where one solves the problem~\eqref{eq:PDE} by truncating to the first $s$ dimensions, i.e., solving the problem for one (relatively large) set $\setv = \{1:s\}$, see~\eqref{eq:truncate-s}.
In contrast, the MDFEM will solve the problem for multiple (relatively small) sets $\setv$ and combine those results to obtain the $\setu$-projected solution $u_\setu$ (which combines all $\setv$-truncated solutions $u(\cdot,\cdot_\setv)$ for all $\setv \subseteq \setu$), see~\eqref{eq:decomposition-u} below, to approximate the expected value of a functional $G$ of the solution of~\eqref{eq:PDE-weak-form} up to a given accuracy.
A quantitative statement about the number of PDEs which need to be solved and the sizes of the sets $\setu$ to achieve a certain error is given later in \RefProp{prop:cardinality-active-set}.

Using the truncated solutions we obtain a \emph{multivariate decomposition} of the full solution $u(\cdot,\bsy)$ by means of the \emph{anchored decomposition}, cf.~\eqref{eq:anchored-decomposition-F} and~\eqref{eq:MDM-decomposition-F},
\begin{align}\label{eq:decomposition-u}
  u(\cdot, \bsy)
  &=
  \sum_{|\setu| < \infty} u_\setu(\cdot, \bsy_\setu)
  =
  \lim_{s \to \infty} \sum_{\setu \subseteq \{1:s\}} u_\setu(\cdot, \bsy_\setu)
  ,
  &
  u_\setu(\cdot, \bsy_\setu)
  &:=
  \sum_{\setv \subseteq \setu}(-1)^{|\setu|-|\setv|} \, u(\cdot, \bsy_\setv)
  ,
\end{align}
with $u(\cdot, \bsy_\setv)$ being $\setv$-truncated solutions.
Note that the full solution consists of $\setu$-projected solutions which in their turn consist of $\setv$-truncated solutions.
The infinite sum in~\eqref{eq:decomposition-u} will be truncated to the active set $\setU_\epsilon$.
Compared to the more standard solution method of truncating the problem to the first $s$ dimensions, see~\eqref{eq:truncate-s}, we remind the reader that, cf.~\eqref{eq:anchored-decomposition-F},
\begin{align*}
  u_s(\cdot, y_1, \ldots, y_s, 0, \ldots)
  &=
  \sum_{\setu \subseteq \{1:s\}} u_\setu(\cdot, \bsy_\setu)
\end{align*}
and hence the limit in~\eqref{eq:decomposition-u} can be read as $u = \lim_{s\to\infty} u_s$.
We will show in \RefSec{sec:parametric-regularity} that the functions $u_\setu$ and $G(u_\setu)$ belong to the spaces $H_{\alpha,0,\rho,|\setu|}(\R^{|\setu|}; V)$ and $H_{\alpha,0,\rho,|\setu|}(\R^{|\setu|})$, see \RefLem{lem:norm-uu-Guu}, where the spaces $H_{\alpha,0,\rho,|\setu|}(\R^{|\setu|})$ are reproducing kernel Hilbert spaces which we introduce in \RefSec{sec:HOQMC} and we formally identify $\R^\N_\setu$ with $\R^{|\setu|}$ after a relabelling of the components.
In \RefRem{rem:infinite-variate-RKHS} we explain that they form an orthogonal decomposition of an infinite-variate reproducing kernel Hilbert space.
We will denote integration for $F_\setu \in H_{\alpha,0,\rho,|\setu|}(\R^{|\setu|})$ by
\begin{align}\label{eq:def:Iu}
  I_\setu(F_\setu)
  &:=
  \int_{\R^{|\setu|}} F_\setu(\bsy_\setu) \rd\mu(\bsy_\setu)
  .
\end{align}
Convergence of the MDM decomposition~\eqref{eq:decomposition-u} and equivalence of the infinite-dimensional integral with
\begin{align}\label{eq:MDM-QoI}
  \II(G(u))
  =
  \int_{\R^\N} G(u(\cdot,\bsy)) \rd\mu(\bsy)
  &=
  \sum_{|\setu| < \infty}
  I_\setu(G(u_\setu))
  ,
\end{align}
is then guaranteed under the conditions from \cite{GMR14} in the setting of applying the MDM to our PDE, see \RefRemTwo{rem:infinite-variate-RKHS}{rem:FF-in-HH} for further details, and the conditions we ask in \RefSec{sec:mainresult}.
Therefore, the convergence statements from \cite{BCDM17,HS19} for solving the $s$-truncated problem as in~\eqref{eq:truncate-s} will also hold in our case.
Note that the active set $\setU_\epsilon$, which will be defined in~\eqref{eq:active-set}, grows for $\epsilon \to 0$ and will include all subsets of $\{1:s_\epsilon\}$ for some $s_\epsilon$ for which $s_\epsilon \to \infty$ as $\epsilon \to 0$.
(Indeed, for any finite set $\setu$ there is an $\epsilon$ such that $\setu \in \setU_\epsilon$.)

The above discussion was using the exact weak solution for each $\setv$-truncated problem which is typically not available and a numerical approximation will be computed using the FE method of which the approximation error will have to be taken into account in the error analysis. Let us define a family of finite-dimensional subspaces $V^h \subset V$, where $h > 0$ is the \emph{mesh diameter}, i.e., the largest element diameter over all elements of the FE mesh, and such that $V^h \subset V^{h'} \subset V$ for $h' < h$.
The finite element approximation of the weak formulation of the $\setv$-truncated problem~\eqref{eq:PDE-weak-form-v} for a given $\bsy_\setv$ is to find $u^h(\cdot, \bsy_\setv) \in V^h$ such that the following equation holds
\begin{align}\label{eq:PDE-weak-form-v-h}
  \int_D a(\bsx, \bsy_\setv) \, \nabla u^h(\bsx, \bsy_\setv) \cdot \nabla v^h(\bsx) \rd \bsx 
  &=
  \int_D f(\bsx) \, v^h(\bsx) \rd \bsx
  ,
  \qquad
  \forall v^h \in V^h
  .
\end{align}

We can now piece together the different parts of the MDFEM.
To approximate~\eqref{eq:QoI} we make use of~\eqref{eq:MDM-QoI}, replacing the integrals by cubature formulas and the solutions to the PDEs by FE approximations, and hence the MDFEM takes the form
\begin{align}\label{eq:MDFEM-algorithm}
  Q_\epsilon(G(u))
  &:= 
  \sum_{\setu \in \setU_\epsilon} Q_{\setu,n_\setu}(G(u_\setu^{h_\setu}))
  =
  \sum_{\setu \in \setU_\epsilon}
  \sum_{i=0}^{n_\setu-1} w_\setu^{(i)} \, G(u^{h_\setu}_\setu(\cdot, \bsy_\setu^{(i)}))
  ,
\end{align}
where $\setU_\epsilon$ is the active set, $Q_{\setu,n_\setu}$ are cubature rules with nodes $\bsy_\setu^{(i)}$ and weights $w_\setu^{(i)}$, and
\begin{align}
  \label{eq:uu-hu}
  u_\setu^{h_\setu}(\cdot, \bsy_\setu)
  &=
  \sum_{\setv \subseteq \setu}(-1)^{|\setu|-|\setv|} \, u^{h_\setu}(\cdot, \bsy_\setv)
  ,
  \\
  \label{eq:G-uu-hu}
  G(u_\setu^{h_\setu}(\cdot, \bsy_\setu))
  &=
  \sum_{\setv \subseteq \setu}(-1)^{|\setu|-|\setv|} \, G(u^{h_\setu}(\cdot, \bsy_\setv))
  .
\end{align}
We emphasize that in calculating $G(u_\setu^{h_\setu}(\cdot, \bsy_\setu))$ we are solving $2^{|\setu|}$ PDEs, all with the same FE mesh, indicated by the FE mesh diameter $h_\setu$ for all the $\setv$-truncated solutions in the formula above.

The computational cost of the MDFEM algorithm is given as 
\begin{align}\label{eq:cost}
  \cost(Q_\epsilon)
  &:=
  \sum_{\setu \in \setU_\epsilon} n_\setu \times \text{cost of evaluating $G(u_\setu^{h_\setu})$}
  .
\end{align}

There are three sources of error in the MDFEM approximation~\eqref{eq:MDFEM-algorithm}: the truncation error from truncating the infinite sum, the cubature errors in approximating the integrals, and the FEM errors in approximating the solutions to the PDEs. They are gathered into two terms as follows
\begin{align}\label{MDFEMerror}
  \left| \II(G(u)) - Q_\epsilon(G(u)) \right|
  &\le
  \left| \sum_{\setu \notin \setU_\epsilon} I_\setu(G(u_\setu)) \right|
  \\
  \notag
  &
  \qquad
  + 
  \left|
    \sum_{\setu \in \setU_\epsilon} 
    \left(
      I_\setu(G(u_\setu)) 
      -
      I_\setu(G(u^{h_\setu}_\setu)) 
    \right)
    +
    \left(I_\setu-Q_{\setu,n_\setu}\right) 
    (G(u^{h_\setu}_\setu))
  \right|
  . 
\end{align}
A sufficient condition to achieve an approximation error of at most $\epsilon$ is that both of these terms are less than $\epsilon/2$. This forces us to construct the active set $\setU_\epsilon$ such that the first  term is bounded by $\epsilon/2$. For each $\setu \in \setU_\epsilon$ the FE space $V^{h_\setu}$ and the cubature rule $Q_{\setu,n_\setu}$ are then chosen such that the second term is bounded by $\epsilon/2$ while minimizing the computational cost~\eqref{eq:cost}.
This will be the strategy we follow in \RefSec{sec:mainresult}.

\section{Higher-order quasi-Monte Carlo rules for finite-dim\-en\-sional integration with respect to the Gaussian distribution using truncation for anchored integrand functions}\label{sec:HOQMC}

In this section we consider quasi-Monte Carlo rules for approximating integrals over $\R^s$ with respect to the Gaussian distribution. 
Particularly, we are interested in computing $s$-dimensional integrals of the form 
\begin{align}\label{eq:Is}
	I_s(F)
	:=
	\int_{\R^s} F(\bsy) \, \rho(\bsy) \rd \bsy
	,
\end{align}
where, with a slight abuse of notation, in this section $\rho$ is the product Gaussian distribution $\prod_{j=1}^s \rho(y_j)$.
In further usage we are calculating $I_\setu(F_\setu) = I_s(F)$, where $I_\setu$ is the integral w.r.t.\ $\rho_\setu(\bsy_\setu) \rd\bsy_\setu$, and the $s$~dimensions here will be a relabelling of the variables in $\bsy_\setu$ with the function $F = F_\setu = G(u^{h_\setu}_\setu(\cdot, \bsy_\setu))$ given by the anchored decomposition for some $\setu \subset \N$ such that $s = |\setu| < \infty$.
We will show in \RefSec{sec:parametric-regularity}, see \RefLem{lem:norm-uu-Guu}, that the functions $G(u^{h_\setu}_\setu(\cdot, \bsy_\setu))$ belong to the function spaces which we introduce below.

To approximate the integral $I_s(F)$, we first truncate the Euclidean domain to a multidimensional bounded box, then use a linear mapping to transform the truncated integral into one over the unit cube, and finally approximate the integral over the unit cube using suitable cubature rules. More precisely, the truncated and transformed integrals have the following form
\begin{align}\label{eq:IsT}
  I_s^T(F)
  &:=
  \int_{[-T,T]^s} F(\bsy) \, \rho(\bsy) \rd \bsy
  =
  (2T)^s \int_{[0,1]^s}  F(\bsT(\bsy)) \, \rho(\bsT(\bsy)) \rd \bsy
  ,
\end{align}
for some $T>0$ and we define the mapping $\bsT : [0,1]^s \to [-T,T]^s$ by
\begin{align*}
  \bsT(\bsy)
  &:=
  ( 2 T \, y_1 - T, \ldots, 2 T \, y_s - T )
  .
\end{align*}
The resulting integral~\eqref{eq:IsT} is then approximated using an $n$-point QMC rule of the form
\begin{align}\label{eq:Q-R-truncated}
  Q_{s,n}(F)
  &:=
  \frac{(2T)^s}{n} \sum_{i=0}^{n-1} F(\bsT(\bsy^{(i)})) \, \rho(\bsT(\bsy^{(i)}))
  ,
\end{align}
where $\{\bsy^{(i)}\}_{i=0}^{n-1}$ are well chosen cubature points on the unit cube.

\label{pp:paragraph-on-T}We note that it would be possible to truncate the box differently for each dimension, defining a box $[-T_1,T_1] \times \cdots \times [-T_s,T_s]$ as was done, e.g., in \cite{NN17}. We do not pursue such a strategy here since in the application of the MDM we are not immediately making use of the different importances of the dimensions for each projected $F_\setu$.

\subsection{Two reproducing kernel Hilbert spaces}\label{sec:RKHS}

We will introduce two reproducing kernel Hilbert spaces.
The first function space is an anchored Sobolev space on the Euclidean space $\R^s$ with Gaussian measure, cf.~\eqref{eq:Is}, for which we will take into account that our integrand function $F$ is anchored. We will show in \RefSec{sec:parametric-regularity} that under appropriate conditions the functions $G(u_{\setu})$ belong to this first function space.
The second function space is an unanchored Sobolev space on the unit cube to analyse our truncated and mapped integral, cf.\ the right hand side of~\eqref{eq:IsT}.
For this second function space we can then adjust techniques from \cite{DKLNS14} such that interlaced polynomial lattice rules will achieve higher-order convergence in approximating the integral.
Contrary to most modern QMC results we do not introduce weights in the function spaces since we rely on the MDM to keep the number of dimensions limited.
Therefore our error bounds will contain exponential factors in the number of dimensions.
However, these weights will eventually show up when we handle the infinite-variate case, see \RefRem{rem:infinite-variate-RKHS}.

\subsubsection{Anchored Gaussian Sobolev space for anchored functions $H_{\alpha,0,\rho,s}(\R^s)$}\label{sec:aSob-R}

We begin with introducing the univariate function space.
For $\alpha \in \N$ the space $H_{\alpha,0,\rho}(\R)$ consists of integrable functions over $\R$ with respect to the Gaussian distribution having absolutely continuous derivatives up to order $\alpha-1$ for any bounded interval and square integrable derivative of order $\alpha$ over $\R$ with respect to the Gaussian distribution and are anchored at~$0$.
Note that this allows us to use the Lebesgue version of the Fundamental Theorem of Calculus, and hence the Taylor theorem with integral remainder up to order~$\alpha$.
For $F, G \in H_{\alpha,0,\rho}(\R)$ the anchored inner product is defined as
\begin{align}\label{eq:ip-a-R}
  \langle F, G \rangle_{H_{\alpha,0,\rho}(\R)}
  &:=
  \sum_{\tau=1}^{\alpha-1} F^{(\tau)}(0) \, G^{(\tau)}(0)
  +
  \int_\R F^{(\alpha)}(y) \, G^{(\alpha)}(y) \, \rho(y) \rd y
  .
\end{align}
Note that, because we use the anchored decomposition for the MDM, we only consider functions such that $F(0) = 0$ and hence the term for $\tau=0$ which is normally there in the first sum is zero here.
The associated norm is given by $\|\cdot\|_{H_{\alpha,0,\rho}(\R)} := \langle \cdot,\cdot \rangle^{1/2}_{H_{\alpha,0,\rho}(\R)}$.

We note that another common choice of a Gaussian Sobolev space is the \emph{unanchored Gaussian Sobolev space} or \emph{Hermite space}, see, e.g.,~\cite{IL15,IKLP15,DILP18}. Instead of anchoring the values of the function and its derivatives up to order $\alpha-1$ at $0$ as in~\eqref{eq:ip-a-R}, they are integrated out against the Gaussian distribution over $\R$.
However, here we want to benefit from the anchored decomposition and therefore will use an anchored Sobolev space.

The anchored Gaussian Sobolev space for anchored functions $H_{\alpha,0,\rho}(\R)$ is a reproducing kernel Hilbert space with kernel
\begin{align}\label{eq:kernel-a-R}
  K_{\alpha,0,\rho}(x,y)
  &:=
  \sum_{\tau=1}^{\alpha-1} \frac{x^\tau}{\tau!} \frac{y^\tau}{\tau!}
  +
  \mathds{1}\{xy > 0\}
  \int_0^{\min\{|x|,|y|\}} 
    \frac{(|x|-t)^{\alpha-1}}{(\alpha-1)!}
    \frac{(|y|-t)^{\alpha-1}}{(\alpha-1)!}
    \frac1{\rho(t)} 
    \rd t
  ,
\end{align}
where $\mathds{1}\{X\}$ is the indicator function on~$X$.
Such a kernel for $\alpha=1$ was given in~\cite[Section~3.3]{NK14}, but we note that in this case the weight function in the formula for the inner-product and the reproducing kernel cannot be taken the same as in the integral~\eqref{eq:Is} for $\alpha=1$, see \cite[Table~1]{NK14}; here we are in fact interested in $\alpha \ge 2$, but see also \RefRem{rem:randomized-QMC}.
A slightly different kernel for an anchored Sobolev space over the Euclidean space with higher order smoothness, although without taking any weight function into account, was given in~\cite[Section~11.5.1]{NW10}.
For completeness we provide the full derivation of the reproducing kernel for inner products like~\eqref{eq:ip-a-R} for general $\rho$ in \RefApp{app:kernel}.

We define the multivariate space as the tensor product of the univariate spaces.
The kernel of our space of anchored functions is then given by
\begin{align*}
  K_{\alpha,0,\rho,s}(\bsx,\bsy)
  &:=
  \prod_{j=1}^s K_{\alpha,0,\rho}(x_j, y_j) 
  .
\end{align*}
Note that this kernel itself also has the anchored property: if there is a $j \in \{1:s\}$ for which $x_j = 0$ or $y_j = 0$ then $K_{\alpha,0,\rho,s}(\bsx,\bsy) = 0$.
The corresponding inner product is 
\begin{align}\label{eq:ip-a-Rs}
  \notag
  \langle F, G \rangle_{H_{\alpha,0,\rho,s}(\R^s)}
  &:=
  \sum_{\substack{\bstau \in \{1:\alpha\}^s \\ \setv := \{j: \tau_j = \alpha\}}}
  \int_{\R^{|\setv|}} 
    F^{(\bstau)}(\bsy_\setv,\bszero_{-\setv}) \,
    G^{(\bstau)}(\bsy_\setv,\bszero_{-\setv}) \,
    \rho_\setv(\bsy_\setv)
  \rd \bsy_\setv
  \\
  &\hphantom{:}=
  \sum_{\substack{\bstau \in \{1:\alpha\}^s \\ \setv := \{j: \tau_j = \alpha\}}}
  \int_{\R^{|\setv|}} 
    F^{(\bstau)}(\bsy_\setv) \,
    G^{(\bstau)}(\bsy_\setv) \,
    \rho_\setv(\bsy_\setv)
  \rd \bsy_\setv  
,
\end{align}
where, in the first line, $-\setv = \{1:s\}\setminus \setv$ and $(\bsy_\setv,\bszero_{-\setv})$ is a vector of $s$ variables such that $(\bsy)_j = y_j$ for $j \in \setv$ and $0$ otherwise.
The second line follows by our convention that $\bsy_\setv$ is a vector of the appropriate size which takes the value zero outside of~$\setv$.
In many references, see, e.g.,~\cite{BD09}, such inner products are usually written using a double sum, which here then takes the following form
\begin{align}\label{eq:ip-a-Rs-doublesum}
  \notag
  &\langle F, G \rangle_{H_{\alpha,0,\rho,s}(\R^s)}
  \\
  &\qquad=
  \sum_{\setv \subseteq \{1:s\}}\sum_{\bstau_{-\setv} \in \{1:\alpha-1\}^{s-|\setv|}} \int_{\R^{|\setv|}} 
  F^{(\bsalpha_\setv,\bstau_{-\setv})}(\bsy_\setv,\bszero_{-\setv}) \,
  G^{(\bsalpha_\setv,\bstau_{-\setv})}(\bsy_\setv,\bszero_{-\setv}) \,
  \rho_\setv(\bsy_\setv)
  \rd \bsy_\setv
  \\\notag
  &\qquad=
  \sum_{\setv \subseteq \{1:s\}}\sum_{\bstau_{-\setv} \in \{1:\alpha-1\}^{s-|\setv|}} \int_{\R^{|\setv|}} 
  F^{(\bsalpha_\setv,\bstau_{-\setv})}(\bsy_\setv) \,
  G^{(\bsalpha_\setv,\bstau_{-\setv})}(\bsy_\setv) \,
  \rho_\setv(\bsy_\setv)
  \rd \bsy_\setv
  ,
\end{align}
where $(\bsalpha_\setv,\bstau_{-\setv})$ is a vector of $s$ variables such that the $j$th component equals $\tau_j$ for $j \in -\setv$ and $\alpha$ otherwise.
The corresponding norm is given by
$\|\cdot\|_{H_{\alpha,0,\rho,s}(\R^s)} := \langle \cdot, \cdot \rangle^{1/2}_{H_{\alpha,0,\rho,s}(\R^s)}$.
In \RefLem{lem:Taylor} in \RefApp{app:kernel} we show how anchored functions in anchored spaces like $H_{\alpha,0,\rho,s}(\R^s)$ can be represented by a Taylor series with integral remainder.

In analysing the error of the MDM algorithm we will need to be able to get estimates on integrals of functions $F_\setu \in H_{\alpha,0,\rho,|\setu|}(\R^{|\setu|})$.
We will make use of the following properties of the reproducing kernel.

\begin{proposition}\label{prop:integrability}
  For $\alpha \in \N$ we have
  \begin{align}\label{eq:def:M}
    M
    &:=
    \int_\R (K_{\alpha,0,\rho}(y,y))^{1/2} \, \rho(y) \rd{y}
    <
    2.767
    <
    \infty
    ,
  \end{align}
  and, for $\setu \subset \N$, we have
  \begin{align*}
    M_\setu
    &:=
    \int_{\R^{|\setu|}} (K_{\alpha,0,\rho,|\setu|}(\bsy_\setu,\bsy_\setu))^{1/2} \, \rho_\setu(\bsy_\setu) \rd{\bsy_\setu}
    =
    M^{|\setu|}
    <
    \infty
    .
  \end{align*}
  Hence for $F_\setu \in H_{\alpha,0,\rho,|\setu|}(\R^{|\setu|})$ we have
  \begin{align}\label{eq:bound-Iu}
    I_\setu(F_\setu)
    =
    \int_{\R^{|\setu|}} F_\setu(\bsy_\setu) \, \rho_\setu(\bsy_\setu) \rd\bsy_\setu
    &\le
    \|F_\setu\|_{H_{\alpha,0,\rho,|\setu|}(\R^{|\setu|})} \, M_\setu
    =
    \|F_\setu\|_{H_{\alpha,0,\rho,|\setu|}(\R^{|\setu|})} \, M^{|\setu|}
    <
    \infty
    ,
  \end{align}
  such that $F_\setu \in H_{\alpha,0,\rho,|\setu|}(\R^{|\setu|}) \subset L^1(\R^{|\setu|},\rho)$ are integrable w.r.t.\ the Gaussian measure~$\rho_\setu$.
  
  Moreover, if $2 \le \alpha < \infty$ then we also have
  \begin{align}\label{eq:int-Kyy}
    \int_\R K_{\alpha,0,\rho}(y, y) \, \rho(y) \rd{y}
    &<
    \infty
    .
  \end{align}
\end{proposition}
\begin{proof}
For $y > 0$ we have
\begin{align*}
  K_{\alpha,0,\rho}(y,y)
  &=
  \sum_{r=1}^{\alpha-1} \frac{y^{2r}}{(r!)^2}
  +
  \int_0^y \frac{(y-t)^{2(\alpha-1)}}{((\alpha-1)!)^2} \frac1{\rho(t)} \rd t
  \le
  \sum_{r=1}^{\alpha-1} \frac{|y|^{2r} }{(r!)^2}
  +
  \frac1{\rho(y)} 
  \frac{|y|^{2\alpha-1}}{((\alpha-1)!)^2 \, (2\alpha-1)}
  .
\end{align*}
The same bound holds for any $y < 0$.
Thus, using $\left(\sum_{j} |a_j| \right)^{1/2} \le \sum_{j} |a_j|^{1/2}$ we  have
\begin{align*}
  \int_\R (K_{\alpha,0,\rho}(y,y))^{1/2} \, \rho(y) \rd y 
  &\le
  \int_\R
  \left(
    \sum_{r=1}^{\alpha-1} \frac{|y|^r}{r!}
    +
    \frac1{\sqrt{\rho(y)}} 
    \frac{|y|^{\alpha-1/2}}{(\alpha-1)! \, (2\alpha-1)^{1/2}}
  \right) 
  \rho(y) \rd y 
  \notag
  \\
  &=
  \sum_{r=1}^{\alpha-1} \frac{1}{2^{r/2} \, \Gamma(1+r/2)}
  +
  \frac{2^{\alpha+1/4} \, \Gamma(\alpha/2 + 1/4)}{(\alpha-1)! \, (2\alpha-1)^{1/2} \, \pi^{1/4}}
  <
  2.767
  ,
\end{align*}
with the maximum for $\alpha \in \N$ achieved for $\alpha = 3$.
The bound on $I_\setu(F_\setu)$ follows by using the reproducing property of the kernel
\begin{align*}
  F_\setu(\bsy_\setu)
  =
  \langle F_\setu, K_{\alpha,0,\rho,|\setu|}(\cdot,\bsy_\setu) \rangle_{H_{\alpha,0,\rho,|\setu|}(\R^{|\setu|})}
  &\le
  \|F_\setu\|_{H_{\alpha,0,\rho,|\setu|}} \, (K_{\alpha,0,\rho,|\setu|}(\bsy_\setu,\bsy_\setu))^{1/2}
  .
\end{align*}
To show the last claim we write
\begin{align*}
  K_{\alpha,0,\rho}(y, y)
  &=
  \sum_{r=1}^{\alpha-1} \frac{y^{2r}}{(r!)^2}
  +
  \frac{\sqrt{2\pi}}{(2\alpha-1)((\alpha-1)!)^2} \, |y|^{2\alpha-1} \, _2F_2(1/2, 1; 1/2+\alpha, \alpha; y^2/2)
\end{align*}
with $_2F_2$ a generalized hypergeometric function.
The sum over~$r$ will stay finite when integrating against the normal density, similar like above.
For the second part we have
\begin{align*}
  \int_0^\infty |y|^{2\alpha-1} \, _2F_2(1/2, 1; 1/2+\alpha, \alpha; y^2/2) \, \rho(y) \rd{y}
  &=
  \sum_{k \ge 0} \frac{\Gamma(2\alpha) \, \Gamma(2k+1)}{\Gamma(2\alpha+2k)} \frac{2^{\alpha-3/2}}{k! \, \sqrt{\pi}} \Gamma(\alpha + k)
  \\
  &=
  \frac{\Gamma(\alpha-1) \, \Gamma(\alpha+1/2)}{\Gamma(\alpha-1/2)} \frac{2^{\alpha-3/2}}{\sqrt{\pi}}
  ,
\end{align*}
which is finite when $2 \le \alpha < \infty$.
\end{proof}

\begin{remark}\label{rem:infinite-variate-RKHS}
Our discussion so far is on finite-variate spaces $H_{\alpha,0,\rho,|\setu|}(\R^{|\setu|})$ of anchored functions.
We now show how they form an orthogonal decomposition for a space of $s$-variate functions (not just anchored functions) and we will then extend this to infinite-variate functions.
The kernel for the $s$-variate weighted anchored Sobolev space can be expressed as, see, \cite[Example~4.4]{KSWW10b},
\begin{align}\label{eq:def:KK_s}
  \KK_s(\bsx, \bsy)
  &:=
  1
  +
  \sum_{\emptyset \ne \setu \subseteq \{1:s\}} \gamma_\setu \, K_{\alpha,0,\rho,|\setu|}(\bsx_\setu, \bsy_\setu)
  =
  \sum_{\setu \subseteq \{1:s\}} \gamma_\setu \, \prod_{j \in \setu} K_{\alpha,0,\rho}(x_j, y_j)
  ,
\end{align}
where $\{ \gamma_\setu \}_{|\setu| < \infty}$ is a sequence of positive weights, where for $\setu = \emptyset$ we set $\gamma_\emptyset := 1$. For $\setu = \emptyset$ we have the space of constant functions and we define $K_{\alpha,0,\rho,|\emptyset|} := 1$ and $\|f_\emptyset\|_{K_{\alpha,0,\rho,|\emptyset|}} := |f_\emptyset|$.
The weights $\gamma_\setu$ in~\eqref{eq:def:KK_s} model the importance of our subspaces and we will show in \RefSec{sec:mainresult} that they can be taken of product form $\gamma_\setu = \prod_{j \in \setu} \gamma_j$ with $\gamma_j = \sqrt{2} \, b_j$.
In \RefSec{sec:mainresult} we will eventually demand that $\{b_j\}_{j\ge1} \in \ell^{p^*}(\N)$ for $p^* \in (0,1)$.
Hence we know that $\sum_{j \ge 1} \gamma_j < \infty$.
By taking the limit for $s \to \infty$ we obtain the reproducing kernel of the infinite-variate reproducing kernel Hilbert space $\HH_{\alpha,\rho,\bsgamma}(\R^\N)$, see~\cite{GMR14},
\begin{align*}
  \KK(\bsx, \bsy)
  &:=
  \sum_{|\setu| < \infty} \gamma_\setu \, \prod_{j \in \setu} K_{\alpha,0,\rho}(x_j, y_j)
  ,
\end{align*}
with inner product
\begin{align*}
  \langle \FF, \GG \rangle_{\HH_{\alpha,\rho,\bsgamma}(\R^\N)}
  &:=
  \sum_{|\setu| < \infty} \gamma_\setu^{-1} \, \langle F_\setu , G_\setu \rangle_{H_{\alpha,0,\rho,|\setu|}}
  ,
\end{align*}
and where $F_\setu$ and $G_\setu$ are obtained by the anchored decomposition of the functions $\FF$ and~$\GG$.
Since our kernels are the tensor products of a univariate kernel, the subspaces are orthogonal and their intersection only contains the zero function for $\setu \ne \setv$, see \cite{GMR14}.
Technically, the kernel $\KK$ is a reproducing kernel only when $\bsx, \bsy \in \YY$ with $\YY := \{ \bsy \in \R^\N : \KK(\bsy, \bsy) < \infty \}$, but we can use \cite[Condition~(C3)]{GMR14},
\begin{align*}
  \sum_{|\setu| < \infty} \gamma_\setu \, \left( \int_\R K_{\alpha,0,\rho}(y, y) \, \rho(y) \rd{y} \right)^{|\setu|}
  &<
  \infty
  ,
\end{align*}
to show that $\mu(\YY) = 1$ for $\alpha \ge 2$ under the condition of product weights $\gamma_\setu = \prod_{j\in\setu} \gamma_j$ with $\sum_{j\ge1} \gamma_j < \infty$.
This can be shown using~\eqref{eq:int-Kyy} and the same reasoning as in the proof of \RefProp{prop:active-set-truncation-error}.
This means $\YY$ can be replaced by $\R^\N$ in the ``almost everywhere'' sense for the infinite-variate reproducing kernel Hilbert space and we can write $\R^{|\setu|}$ as the domain for each subspace instead of the formal domain~$\YY$.
Finally, this condition allows us to claim convergence of $s$-truncated functions to the infinite-variate function in the Hilbert space and to claim equality for the MDM form of the infinite-dimensional integral
\begin{align*}
  \II(\FF)
  &=
  \int_{\R^\N} \FF(\bsy) \rd\mu(\bsy)
  =
  \sum_{|\setu| < \infty} I_\setu(F_\setu)
  ,
\end{align*}
see again \cite{GMR14}.
We refer to \RefRem{rem:FF-in-HH} for the verification that our integrand function $\FF = G(u)$, with $u(\cdot, \bsy)$ being the solution of the PDE, belongs to $\HH_{\alpha,\rho,\bsgamma}(\R^\N)$ under the studied conditions.
\end{remark}

\subsubsection{Unanchored Sobolev space on the unit cube $H_{\alpha,s}([0,1]^s)$}\label{sec:uSob-unitcube}

We need a function space over the unit cube to analyse the error for our cubature method of choice, which will be interlaced polynomial lattice rules.
For this, let us define the \emph{unanchored Sobolev space over the unit cube} $H_{\alpha,s}([0,1]^s)$.
This is the same space as was used in, e.g., \cite{BD09}.
For $\alpha \in \N$ the univariate space $H_\alpha([0,1])$ consists of integrable functions over $[0,1]$ having absolutely continuous derivatives up to order $\alpha-1$ and square integrable derivative of order~$\alpha$.
For the univariate case the unanchored inner product is
\begin{align}\label{eq:ip-u-unitcube}
  \langle F, G \rangle_{H_\alpha([0,1])}
  &:=
   \sum_{\tau=0}^{\alpha-1} \left(\int_0^1 F^{(\tau)}(y) \rd y \right) \left(\int_0^1 G^{(\tau)}(y) \rd y\right)
   +
   \int_0^1 F^{(\alpha)}(y) \, G^{(\alpha)}(y) \rd y
  ,
\end{align}
with norm $\|\cdot\|_{H_\alpha([0,1])} := \langle \cdot, \cdot \rangle^{1/2}_{H_\alpha([0,1])}$.
Note that contrary to our Gaussian function space which is taking benefit of the functions being anchored here we must include the typical $\tau=0$ term in the first sum.

The  multivariate space is the tensor product of the univariate spaces with inner product
\ifpreprint
\begin{multline}\label{eq:ip-u-unitcube-s}
  \langle F, G \rangle_{H_{\alpha,s}([0,1]^s)}
  \\
  :=
  \sum_{\substack{\bstau \in \{0:\alpha\}^s \\ \setv := \{j: \tau_j = \alpha\}}}
  \int_{[0,1]^{|\setv|}}
  \left(
    \int_{[0,1]^{s-|\setv|}}
    F^{(\bstau)}(\bsy) \rd \bsy_{-\setv}
  \right) 
  \left( 
    \int_{[0,1]^{s-|\setv|}}
    G^{(\bstau)}(\bsy) \rd \bsy_{-\setv}
  \right) 
  \rd \bsy_\setv
  ,
\end{multline}
\else
\begin{align}\label{eq:ip-u-unitcube-s}
  \langle F, G \rangle_{H_{\alpha,s}([0,1]^s)}
  &
  :=
  \sum_{\substack{\bstau \in \{0:\alpha\}^s \\ \setv := \{j: \tau_j = \alpha\}}}
  \int_{[0,1]^{|\setv|}}
  \left(
    \int_{[0,1]^{s-|\setv|}}
    F^{(\bstau)}(\bsy) \rd \bsy_{-\setv}
  \right) 
  \left( 
    \int_{[0,1]^{s-|\setv|}}
    G^{(\bstau)}(\bsy) \rd \bsy_{-\setv}
  \right) 
  \rd \bsy_\setv
  ,
\end{align}
\fi
and norm $\|\cdot\|_{H_{\alpha,s}([0,1]^s)} := \langle \cdot, \cdot \rangle^{1/2}_{H_{\alpha,s}([0,1]^s)}$.

Let us denote integration over the unit cube for functions $F \in H_{\alpha,s}([0,1]^s)$ by
\begin{align*}
  I_{[0,1]^s}(F)
  &:=
  \int_{[0,1]^s} F(\bsy) \rd\bsy
  ,
\end{align*}
and a QMC rule using a point set $P_n = \{ \bsy^{(0)}, \ldots, \bsy^{(n-1)} \}$ over the unit cube by
\begin{align*}
  Q_{[0,1]^s,P_n}(F)
  &:=
  \frac{1}{n} \sum_{i=0}^{n-1} F(\bsy^{(i)})
  .
\end{align*}
For this space we can construct interlaced polynomial lattice rules which achieve the almost optimal order of convergence.
The full derivation of the following result is given in \RefApp{app:IPLR}.

\begin{theorem}\label{thm:error-bound-IPLR}
  For $\alpha \in \N$, with $\alpha \ge 2$, let $F \in H_{\alpha,s}([0,1]^s)$.
  For any $m \in \N$ an interlaced polynomial lattice rule of order $\alpha$ with point set $P_{n,\alpha}$ with $n = 2^m$ points can be constructed with cost $O(\alpha s n \log(n))$ such that 
  \begin{align*}
    \left| I_{[0,1]^s}(F) - Q_{[0,1]^s,P_{n,\alpha}}(F) \right|
    &\le 
    \frac{\widetilde{C}_{\alpha,\lambda,s}}{n^\lambda} \,
    \|F\|_{H_{\alpha,s}([0,1]^s)}
    ,
    &&
    \forall \lambda\in [1,\alpha)
    ,
  \end{align*}
  where
  \begin{align}\label{eq:def:Ctilde-alpha-lambda-s}
    \widetilde{C}_{\alpha, \lambda,s}
    &:=
    4^\lambda \,
    2^{\alpha(\alpha-1)s/2}
    \left[
      \left(1+\frac1{2^{\alpha/\lambda}-2}\right)^{\alpha s} - 1
    \right]^\lambda
    .
  \end{align}
\end{theorem}
\begin{proof}
  This result follows from combining Propositions~\ref{prop:wce-uSob-unit-cube} and~\ref{prop:bound-E} in the appendix.
  The construction cost follows from the algorithm for product weights, which we set all equal to~$1$, in~\cite{DKLNS14}.
\end{proof}

\subsubsection{Norm in $H_{\alpha,s}([0,1]^s)$ after mapping and truncating from $H_{\alpha,0,\rho,s}(\R^s)$}

To complete the analysis of mapping and truncating the integral~\eqref{eq:Is} to~\eqref{eq:IsT} and using the result from \RefThm{thm:error-bound-IPLR} we need to show a bound on the norm in $H_{\alpha,s}([0,1]^s)$ of the mapped function in terms of the norm in $H_{\alpha,0,\rho,s}(\R^s)$ of the original function.
Since the proof is quite long it is presented in the appendix.

\begin{proposition}\label{prop:embedding-F-rho-T}
  For any $F \in H_{\alpha,0,\rho,s}(\R^s)$ with $\alpha \in \N$ and $T \ge 1/(2\sqrt{2})$, the function $(F \rho) \circ \bsT : [0,1]^s \to \R^s$ belongs to $H_{\alpha,s}([0,1]^s)$ and
  \begin{align*}
    \|(F \rho) \circ \bsT\|_{H_{\alpha,s}([0,1]^s)} 
    &\le 
    C_{1,\alpha}^s \, T^{(\alpha-1/2)s} \, \|F\|_{H_{\alpha,0,\rho,s}(\R^s)}
    ,
  \end{align*}	
  where
  \begin{align}\label{eq:def:C1}
    C_{1,\alpha}
    &:=
    \alpha! \, 2^{3\alpha} \left( \alpha\, (1+\alpha/2) \frac{1}{\sqrt{2\pi}} \, \Gamma(2\alpha) \, I_0(1/2) \right)^{1/2}
    ,
  \end{align}
  and $I_0(\cdot)$ is the modified Bessel function of the first kind of order~$0$ with $I_0(1/2) \approx 1.06348$.
\end{proposition}
\begin{proof}
  See \RefApp{app:proof-embedding}.
\end{proof}

\subsection{Higher-order quasi-Monte Carlo for integration over $\R^s$}

We are now ready to state the error bound of the method described in the beginning of this section.

\begin{theorem}\label{thm:HO-Q-R-truncated}
  For $s \in \N$, $\alpha \in \N$, with $\alpha \ge 2$, let $F \in H_{\alpha,0,\rho,s}(\R^s)$.
  Let $P_{n,\alpha} = \{ \bsy^{(0)}, \ldots, \bsy^{(n-1)} \}$ be the point set of an interlaced polynomial lattice rule of order $\alpha$ with $n = 2^m$ points, with $m \in \N$, according to \RefThm{thm:error-bound-IPLR}.
  Then, for any $\lambda \in [1,\alpha)$ and by taking $T = 2 + 2 \sqrt{\lambda \ln(n)}$, the QMC rule $Q_{s,n}$, defined in~\eqref{eq:Q-R-truncated}, using the point set $P_{n,\alpha}$ has an error bounded as
  \begin{align}\label{eq:Q-R-truncated-error-bound}
    \left| I_s(F) - Q_{s,n}(F) \right| 
    &\le 
    C_{\alpha,\lambda,s} \,
    \frac{(\ln(n))^{(\alpha/2+1/4)s}}{n^\lambda} \,
    \|F\|_{H_{\alpha,0,\rho,s}(\R^s)}
    ,
\end{align}
where $C_{\alpha,\lambda,s}$ is a constant, independent of $F$ and $n$, defined below by~\eqref{eq:def:C-alpha-lambda-s}.
\end{theorem}
\begin{proof}
The error splits into two terms
\begin{align}\label{eq:Is-Qsn-error-split}
  &\left| I_s(F) - Q_{s,n}(F) \right|
  \notag
  \\
  &\qquad \le 
  \left|I_s(F)  -  \int_{[-T,T]^s} F(\bsy) \, \rho(\bsy) \rd \bsy\right| 
  + 
  \left|(2T)^s\int_{[0,1]^s}  F(\bsT(\bsy)) \, \rho(\bsT(\bsy)) \rd \bsy - Q_{s,n}(F)  \right| 
  .
\end{align}

The first term is the domain truncation error.
Similar as in \RefProp{prop:integrability}, using the reproducing property of the kernel $K_{\alpha,0,\rho,s}$, we obtain
\begin{align}\label{eq:intFT}
  \left|I_s(F)  -  \int_{[-T,T]^s} F(\bsy) \, \rho(\bsy) \rd \bsy\right|
  &\le 
  \|F\|_{H_{\alpha,0,\rho,s}(\R^s)} \,
  \int_{\R^s \setminus [-T,T]^s}
  (K_{\alpha,0,\rho,s}(\bsy, \bsy))^{1/2} \,
  \rho(\bsy) \rd \bsy
  .
\end{align}
To obtain a bound on the truncation error we will bound the integral on the right hand side.
We have
\begin{multline}\label{eq:intK}
  \int_{\R^s \setminus [-T,T]^s} 
  (K_{\alpha,0,\rho,s}(\bsy, \bsy))^{1/2} \,
  \rho(\bsy) \rd \bsy
  \\
  \le
  2\,  \sum_{j=1}^s \int_\R \cdots \int_T^{+\infty}\cdots \int_\R 
  \prod_{i=1}^s (K_{\alpha,0,\rho}(y_i, y_i))^{1/2} \,
  \rho(\bsy) \rd y_1 \cdots \rd y_j \cdots \rd y_s
  .
\end{multline}
We will estimate each of the above integrals. Using the same argument as in the proof of \RefProp{prop:integrability} we have for any $T \ge 2$
\begin{align}\label{eq:boundker}
  \int_T^{+\infty} (K_{\alpha,0,\rho}(y,y))^{1/2} \, \rho(y) \rd y
  &\le
  \int_T^{+\infty}
  \left(
    \sum_{r=1}^{\alpha-1} \frac{y^r}{r!}
    +
    \frac1{\sqrt{\rho(y)}} 
    \frac{y^{\alpha-1/2}}{(\alpha-1)! \, (2\alpha-1)^{1/2}}
  \right) \rho(y) \rd y
  \notag
  \\
  &\le
  \int_T^{+\infty}
  \left(
    \sum_{r=1}^{\alpha-1} \frac{y^r}{r!}
    +
    \frac{y^\alpha}{\alpha!}
    \frac1{\sqrt{\rho(y)}} 
    \frac1{T^{1/2}}
    \frac{\alpha}{(2\alpha-1)^{1/2}}
  \right) \rho(y) \rd y
  \notag
  \\
  &\le
  \int_T^{+\infty}
  \max\left\{ \sqrt{\rho(y)} , \frac1{T^{1/2}} \frac{\alpha}{(2\alpha-1)^{1/2}}  \right\}
  \left( \sum_{r=1}^\alpha \frac{y^r}{r!} \right)
  \sqrt{\rho(y)} \rd y
  \notag
  \\
  &\le
  \sqrt{\frac{\alpha}{2}}
  \int_T^{+\infty} \exp(y) \,
  \sqrt{\rho(y)} \rd y
  ,
\end{align}
where we used $\alpha/\sqrt{2\alpha-1} \le \sqrt{\alpha}$ for $\alpha \ge 1$, $\sqrt{\rho(y)} \le 1/\sqrt{y} \le 1/\sqrt{T}$ for $y \ge T$, and in the last step filled in $T = 2$.
Moreover, we have for $T > 2$, and with the substitution $t = y/2 - 1$,
\begin{align*}
  \int_T^{+\infty}
  \exp(y) \,
  \sqrt{\rho(y)} \rd y
  &=
  \frac{\e}{(2\pi)^{1/4}} 
  \int_T^{+\infty} 
  \exp\left(-\left(\frac{y}{2}-1\right)^2\right)
  \rd y
  \\
  &\le
  \frac{\e}{(2\pi)^{1/4}} 
  \int_T^{+\infty} 
  \exp\left(-\left(\frac{y}{2}-1\right)^2\right)
  \frac{y/2-1}{T/2-1}
  \rd y
  \\
  &=
  \frac{\e}{(2\pi)^{1/4}(T/2 -1)}
  \int_{T/2-1}^{+\infty}
  2 \, t \,
  \exp(-t^2)
  \rd t
  \\
  &=
  \frac{2\,\e}{(2\pi)^{1/4}} \, \frac{\e^{-(T/2-1)^2}}{T-2}
  .
\end{align*}
Inserting this into~\eqref{eq:boundker} yields
\begin{align*}
  \int_T^{+\infty} (K_{\alpha,0,\rho}(y,y))^{1/2} \, \rho(y) \rd y 
  &\le
  \sqrt{\frac{\alpha}{2}} \,
  \frac{2 \,\e}{(2\pi)^{1/4}} \,
  \frac{\e^{-(T/2-1)^2}}{T-2}
  .
\end{align*}
Applying this inequality together with~\eqref{eq:def:M} and~\eqref{eq:intK} to~\eqref{eq:intFT} we obtain
\begin{align}\label{eq:Is:truncation-error}
  \left| I_s(F)  -  \int_{[-T,T]^s} F(\bsy) \, \rho(\bsy) \rd \bsy \right|
  &\le
  C_{3,\alpha,s} \,
  \frac{\e^{-(T/2-1)^2}}{T-2} \,
  \|F\|_{H_{\alpha,0,\rho,s}(\R^s)}
  ,
\end{align}
with
\begin{align}\label{eq:def:C3}
  C_{3,\alpha,s}
  &:=
  2 s
  M^{s-1}  
  \sqrt{\frac{\alpha}{2}} \,
  \frac{2 \,\e}{(2\pi)^{1/4}}
  <
  5 \, s \, M^{s-1} \, \sqrt{\alpha}
  .
\end{align}

We move to the second term of the total error which is the cubature error. Using the result of \RefThm{thm:error-bound-IPLR} and \RefProp{prop:embedding-F-rho-T} we have for any $\lambda \in [1, \alpha)$
\begin{align}
  \notag
  &\left| (2T)^s\int_{[0,1]^s} F(\bsT(\bsy)) \, \rho(\bsT(\bsy)) \rd \bsy - Q_{s,n}(F) \right|
  \\\notag
  &\qquad=
  (2T)^s \left| \int_{[0,1]^s} F(\bsT(\bsy)) \, \rho(\bsT(\bsy)) \rd \bsy - \frac1{n} \sum_{i=0}^{n-1} F(\bsT(\bsy^{(i)})) \, \rho(\bsT(\bsy^{(i)})) \right|
  \\\notag
  &\qquad\le
  (2T)^s \, \frac{\widetilde{C}_{\alpha,\lambda,s}}{n^\lambda} \, \|(F \rho) \circ \bsT\|_{H_{\alpha,s}([0,1]^s)} 
  \notag
  \\\notag
  &\qquad\le
  2^s \,
  \widetilde{C}_{\alpha,\lambda,s} \, C_{1,\alpha}^s \, T^{(\alpha+1/2)s} \, \frac1{n^\lambda} \, \|F\|_{H_{\alpha,0,\rho,s}(\R^s)} 
  \\\label{eq:IsT:quadrature-error}
  &\qquad=
  C_{4,\alpha,\lambda,s} \, T^{(\alpha+1/2)s} \, \frac1{n^\lambda} \, \|F\|_{H_{\alpha,0,\rho,s}(\R^s)} 
  ,
\end{align}
where
\begin{align}\label{eq:def:C4}
  C_{4,\alpha,\lambda,s}
  &:=
  2^s \, \widetilde{C}_{\alpha,\lambda,s} \, C_{1,\alpha}^s
  ,
\end{align}
with $C_{1,\alpha}$ and $\widetilde{C}_{\alpha,\lambda,s}$, respectively, defined in~\eqref{eq:def:C1} and in~\eqref{eq:def:Ctilde-alpha-lambda-s}.

Combining~\eqref{eq:Is-Qsn-error-split},~\eqref{eq:Is:truncation-error} and~\eqref{eq:IsT:quadrature-error} leads to
\begin{align*}
  \left| I_s(F) - Q_{s,n}(F) \right|
  &\le
  \left[
  C_{3,\alpha,s} \, \frac{\e^{-(T/2-1)^2}}{T-2}
  +
  C_{4,\alpha,\lambda,s} \, T^{(\alpha+1/2)s} \, \frac1{n^\lambda}
  \right]
  \, \|F\|_{H_{\alpha,0,\rho,s}(\R^s)}
  .
\end{align*}
To balance the dominating terms in the square brackets we choose $T = 2 + 2 \sqrt{\lambda \ln(n)}$ such that $\e^{-(T/2-1)^2} = n^{-\lambda}$. Hence,
\begin{align*}
  &\left| I_s(F) - Q_{s,n}(F) \right| 
  \\
  &\qquad\le
  \left[
  C_{3,\alpha,s} \,
  \frac{1}{2 \sqrt{\lambda \ln(n)}}
  +
  C_{4,\alpha,\lambda,s} \, \left(2+2\sqrt{\lambda \ln(n)}\right)^{(\alpha+1/2)s}
  \right]
  \, \frac1{n^\lambda}
  \, \|F\|_{H_{\alpha,0,\rho,s}(\R^s)} 
  \\
  &\qquad\le
  \left[
  C_{3,\alpha,s} \,
  \frac{1}{2\sqrt{\lambda \ln(n)}}
  +
  C_{4,\alpha,\lambda,s} \,
  \left(\left(\frac{2}{\sqrt{\ln 2}} + 2\sqrt{\lambda}\right) \sqrt{\ln(n)}\right)^{(\alpha+1/2)s}
  \right]
  \, \frac1{n^\lambda}
  \, \|F\|_{H_{\alpha,0,\rho,s}(\R^s)} 
  \\
  &\qquad\le
  C_{\alpha,\lambda,s} \, \frac{(\ln(n))^{(\alpha/2+1/4)s}}{n^\lambda} \, \|F\|_{H_{\alpha,0,\rho,s}(\R^s)} 
  ,
\end{align*}
where in the second inequality we used $2+2\sqrt{\lambda\ln(n)} \le \left(\frac{2}{\sqrt{\ln 2}} + 2\sqrt{\lambda}\right)\sqrt{\ln(n)}$ for $n\ge 2$ and
\begin{align}\label{eq:def:C-alpha-lambda-s}
  C_{\alpha,\lambda,s}
  &:=
  \max\left\{
  \frac{C_{3,\alpha,s}} {2\sqrt{\lambda \ln(2)}}
  ,
  C_{4,\alpha,\lambda,s} \left(\frac{2}{\sqrt{\ln 2}} + 2\sqrt{\lambda}\right)^{(\alpha+1/2)s}
  \right\}
  ,
\end{align}
with $C_{3,\alpha,s}$ and $C_{4,\alpha,\lambda,s}$, respectively, defined in~\eqref{eq:def:C3} and~\eqref{eq:def:C4}.
\end{proof}

\begin{remark}\label{rem:randomized-QMC}
	 A similar result as that of \RefThm{thm:HO-Q-R-truncated} can be shown for the anchored Gaussian Sobolev space with first order smoothness $H_{\alpha,0,\rho,s}(\R^s)$ with $\alpha = 1$ using \emph{randomly digitally shifted polynomial lattice rules} which achieve the optimal convergence rate in $H_{\alpha,s}([0,1]^s)$ with $\alpha = 1$, see~\cite{DKPS05}.
	 However, in this case we have to be slightly careful as~\eqref{eq:int-Kyy} does not hold for $\alpha=1$.
	 We note however that we can change from $y_j \sim \calN(0,1)$ to $y_j \sim \calN(0, c^2)$ with $0 < c < 1$ in~\eqref{eq:randomfield-Z}, in effect changing the variance of the normal distribution that we integrate against.
	 To keep the law of the random field unchanged we will have to divide each $y_j$ by~$c$ in~\eqref{eq:randomfield-Z}.
	 If we keep the weight in the kernel unchanged then the integral~\eqref{eq:int-Kyy} w.r.t.\ $\calN(0, c^2)$ will be finite, see also \cite[Table~1]{NK14}.
	 The effect of dividing $y_j$ by $c$ can now be moved into the basis functions $\phi_j$ and eventually ends up as multiplying each $b_j$ with~$c$.
	 Now take $c = 1 - \delta$ for arbitrarily small $\delta > 0$.
	 For condition~\eqref{eq:condition:kappa} we then obtain $\kappa_c = \kappa/c \le (1+\delta) \, \kappa < \infty$ if before we had $\kappa < \infty$.
	 The same remark holds for condition~\eqref{eq:condition:kappa'} which asks $\kappa < \ln(2) / \alpha$ and which we will need in the next section.
	 We here obtain $\kappa_c \le (1+\delta) \, \kappa < \ln(2) / \alpha$ if before we had $\kappa < \ln(2) / \alpha$.
	 Combining such randomized cubature rules with a suitable truncation of the Euclidean domain gives a similar convergence rate as in~\eqref{eq:Q-R-truncated-error-bound}, however, of order $\lambda \in [\frac{1}{2}, 1)$.
\end{remark}

\begin{remark}\label{rem:alternative-function-spaces}
	An alternative approach is to embed the function $(F \rho) \circ \bsT$ into the \emph{anchored Sobolev space over the unit cube} and then use \emph{higher-order polynomial lattice rules}, see, e.g.,~\cite{DG14,NN21}. However, a non-trivial result similar to \RefProp{prop:embedding-F-rho-T} is then needed to obtain an explicit formula for the embedding constant.
	One of the reasons we choose our approach of mapping to the unanchored Sobolev space is that then the technique of \cite{DKLNS14} to make use of \emph{interlaced polynomial lattice rules} could be used with a construction cost of $O(\alpha s n \log(n))$ as given in \RefThm{thm:error-bound-IPLR}.
	The construction cost of higher-order polynomial lattice rules on the other hand grows exponentially in $n$ with respect to the smoothness, see~\cite{BDLNP12}.
	We remark that this excessive construction cost could still be avoided by making use of yet another embedding of the anchored Sobolev space over the unit cube into the unanchored Sobolev space on the unit cube, see \cite[Example~2.1]{GHHR17}.
	In that case the interlaced polynomial lattice rules from \RefThm{thm:error-bound-IPLR} could then also be used.
\end{remark}

\section{Parametric regularity of the PDE solution}\label{sec:parametric-regularity}

We next derive bounds for mixed derivatives of the solution $u(\cdot,\bsy)$ with respect to~$\bsy$.
For $\alpha \in \N$ and $\setu \subset \N$ let us define the Bochner norm based on the spaces $H_{\alpha,0,\rho,|\setu|}(\R^{|\setu|})$, with inner product~\eqref{eq:ip-a-Rs}, and $V$ by
\begin{align}\label{eq:def:Bochnernorm}
    \|u_\setu\|^2_{H_{\alpha,0,\rho,|\setu|}(\R^{|\setu|};V)}
    &:=
    \sum_{\substack{ \bstau_\setu \in \{1:\alpha\}^{|\setu|} \\
     			     \setv := \{j : \tau_j = \alpha\} }}
      \int_{\R^{|\setv|}} 
     \|(\partial^{\bstau_\setu}_{\bsy_\setu} u_\setu)(\cdot,\bsy_\setv)\|_V^2
     \,
     \rho_\setv(\bsy_\setv)
     \rd \bsy_\setv
    .
\end{align}
Note that $\partial^{\bstau_\setu}_{\bsy_\setu} u_\setu$ is a function which only depends on the variables in $\setu$ and inside the integral we evaluate this function at $\bsy_\setv$ with $\setv \subseteq \setu$ setting all $y_j = 0$ for $j \notin \setv$.
The following result was also used in~\cite[Lemma~2]{NN21} for the analysis of the MDFEM in the uniform case and allows us to use the regularity analysis on $u(\cdot,\cdot_\setu)$ instead of on~$u_\setu$, since their norms in $H_{\alpha,0,\rho,|\setu|}(\R^{|\setu|};V)$ coincide.
Note however that $u(\cdot,\cdot_\setu) \notin H_{\alpha,0,\rho,|\setu|}(\R^{|\setu|};V)$ since it does not satisfy the anchored properties, except maybe in exceptional cases.

\begin{lemma}\label{lem:norm-uu-u-truncated-u}
	For $\alpha \in \N$ and $\setu \subset \N$, let $u_\setu$ be obtained by the anchored decomposition~\eqref{eq:decomposition-u} and let $u(\cdot, \cdot_\setu)$ be the $\setu$-truncated solution, cf.~\eqref{eq:PDE-weak-form-v}.
	Let $G$ be a bounded linear functional on $V$ such that $|G(v)| \le \|G\|_{V^*} \, \|v\|_V$ for all $v \in V$.
	Then it holds that
	\begin{align*}
	  \|u_\setu\|_{H_{\alpha,0,\rho,|\setu|}(\R^{|\setu|};V)}
	  &=
	  \|u(\cdot, \cdot_\setu)\|_{H_{\alpha,0,\rho,|\setu|}(\R^{|\setu|};V)}
	  \intertext{and}
	  \|G(u_\setu)\|_{H_{\alpha,0,\rho,|\setu|}(\R^{|\setu|})}
	  &\le
	  \|G\|_{V^*} \,
	  \|u(\cdot, \cdot_\setu)\|_{H_{\alpha,0,\rho,|\setu|}(\R^{|\setu|};V)}
	  .
	\end{align*}
\end{lemma}
\begin{proof}
  This follows directly from property~\eqref{eq:prop2} of the anchored decomposition and~\eqref{eq:def:Bochnernorm}. 
\end{proof}

For a given $\bsy_\setu \in \R^\N_\setu$ and with $v(\cdot, \bsy_\setu) \in V$ let us introduce the notation
\begin{align*}
  \|v(\cdot, \bsy_\setu)\|^2_{V,a_{\bsy_\setu}}
  &:= 
  \int_D a(\bsx,\bsy_\setu) \, |\nabla v(\bsx, \bsy_\setu)|^2 \rd \bsx
  .
\end{align*}
Note that $\|\cdot\|_{V,a_{\bsy_\setu}}$ depends on~$\bsy_\setu$.
It is easy to see that for every $v(\cdot, \bsy_\setu) \in V$
\begin{align}
  \label{eq:norm-Va-lbound}
  a_{\min}(\bsy_\setu) \, \|v(\cdot, \bsy_\setu)\|_V^2
  &\le
  \|v(\cdot, \bsy_\setu)\|^2_{V,a_{\bsy_\setu}}
\end{align}
and additionally, when $v(\cdot, \bsy_\setu) = u(\cdot, \bsy_\setu) \in V$ is the $\setu$-truncated solution, cf.~\eqref{eq:PDE-weak-form-v}, we obtain
\begin{align}
  \label{eq:u-norm-Va-ubound}
  \|u(\cdot, \bsy_\setu)\|^2_{V,a_{\bsy_\setu}}
  =
  \int_D f(\bsx) \, u(\bsx, \bsy_\setu) \rd\bsx
  \le
  \|f\|_{V^*} \, \|u(\cdot,\bsy_\setu)\|_V
  &\le
  \frac{\|f\|^2_{V^*}}{a_{\min}(\bsy_\setu)}
  ,
\end{align}
where $f \in V^*$ is the right hand side of the PDE and where we used~\eqref{eq:LM} for $\bsy = \bsy_\setu$.

The following result is modified from \cite[Proposition~3.1]{Kaz18}, see also~\cite[Theorem~4.1]{BCDM17}, and accounts for the truncation to an arbitrary set~$\setu \subset \N$.
 
\begin{proposition}\label{prop:sum-norm-Va}
	Given $\alpha \in \N$, if there exists a sequence $\{b_j\}_{j \ge 1}$ with $0 < b_j \le 1$ for all~$j$ and the constant $\kappa$, defined in~\eqref{eq:condition:kappa}, satisfies
	\begin{align}\label{eq:condition:kappa'}
	  \kappa
	  =
	  \left\|\sum_{j \ge 1} \frac{|\phi_j|}{b_j} \right\|_{L^\infty(D)}
	  <
	  \frac{\ln(2)}{\alpha}
	  ,
	\end{align}
	then, for $\mu$~almost every $\bsy \in \R^\N$, such that $u(\cdot,\bsy_\setu)$ is the solution of the $\setu$-truncated problem,
	it holds that
	\begin{align*}
	  \sum_{\bstau_\setu \in \{1:\alpha\}^{|\setu|}}
	  \bsb_\setu^{-2\bstau_\setu} \,
	  \|(\partial^{\bstau_\setu}_{\bsy_\setu} u(\cdot, \cdot_\setu))(\cdot, \bsy_\setu)\|_{V,a_{\bsy_\setu}}^2 
	  &\le
	  C_{\kappa,\alpha} \, \|u(\cdot, \bsy_\setu)\|_{V,a_{\bsy_\setu}}^2
	  ,
	\end{align*}
	with
	\begin{align*}
	  \bsb_\setu^{\bstau_\setu}
	  &:=
	  \prod_{j \in \setu}b_j^{\tau_j},
	\end{align*}
	and
	$C_{\kappa,\alpha}
	:=
	\sum_{k=0}^\infty \delta_{\kappa, \alpha}^k
	<
	\infty$,
    with $0 < \delta_{\kappa,\alpha} < 1$ being a constant depending on $\kappa$ and $\alpha$ such that $\kappa < \delta_{\kappa,\alpha} \ln(2) / \alpha$.
This implies
	\begin{align*}
	\sum_{\bstau_\setu \in \{1:\alpha\}^{|\setu|}}
	\|(\partial^{\bstau_\setu}_{\bsy_\setu} u(\cdot, \cdot_\setu))(\cdot, \bsy_\setu)\|_V^2 
	\le
	\bsb_\setu^2
	\,
	C_{\kappa,\alpha} \, 
    \frac{\|f\|_{V^*}^2}{(a_{\min}(\bsy_\setu))^2}
	.
	\end{align*}
\end{proposition}
\begin{proof}
    The last inequality in the statement follows by using $0 < b_j \le 1$ which implies $\bsb_\setu^{-2} \le \bsb_\setu^{-2\bstau_\setu}$ for every $\bstau_\setu \in \{1:\alpha\}^{|\setu|}$ together with~\eqref{eq:norm-Va-lbound} and~\eqref{eq:u-norm-Va-ubound}.
\end{proof}

The previous result can now be used to show a bound on the norm of $u(\cdot,\cdot_\setu)$ and as a consequence also on the norm of~$u_\setu$.

\begin{lemma}\label{lem:norm-u}
	Assume the sequence $\{b_j\}_{j\ge 1}$ satisfies the assumptions of \RefProp{prop:sum-norm-Va} for a given $\alpha \in \N$, and, additionally that $\{b_j\}_{j\ge 1} \in \ell^{p^*}(\N)$ for some $p^* \in (0,1]$.
	Then it holds that
	\begin{align*}
        \|u(\cdot, \cdot_\setu)\|_{H_{\alpha,0,\rho,|\setu|}(\R^{|\setu|};V)}
        &\le
        \bsb_\setu \, 2^{|\setu|/2}
        \,
        C_{\kappa, \alpha}' \,
        \|f\|_{V^*}
        ,
    \end{align*}
    with $\bsb_\setu = \prod_{j \in \setu} b_j$ and
    \begin{align}\label{eq:def:C'}
      C_{\kappa, \alpha}'
      &:=
      C_{\kappa, \alpha} ^{1/2} \, \exp\!\left( \sum_{j \ge 1} \left[ (\kappa b_j)^2 + \frac{2 \kappa b_j}{\sqrt{2 \pi}} \right] \right)
      <
      \infty
      .
    \end{align}
\end{lemma}
\begin{proof}
	Using the definition of the Bochner norm~\eqref{eq:def:Bochnernorm}, but with the $H_{\alpha,0,\rho,|\setu|}$ norm written as a double sum, cf.~\eqref{eq:ip-a-Rs-doublesum}, we have
	\begin{align}\label{eq:boundBnorm1}
    	\|u(\cdot, \cdot_\setu)\|_{H_{\alpha,0,\rho,|\setu|}(\R^{|\setu|};V)}^2
    	&=
    	\sum_{\setv \subseteq \setu}
    	\int_{\R^{|\setv|}} 
    	\sum_{\bstau_{\setu\setminus \setv} \in \{1:\alpha-1\}^{|\setu|-|\setv|}} 
    	\|(\partial^{(\bsalpha_\setv,\bstau_{\setu \setminus \setv})}_{\bsy_\setu} u(\cdot, \cdot_\setu))(\cdot,\bsy_\setv)\|_V^2
    	\,
    	\rho_\setv(\bsy_\setv)
    	\rd \bsy_\setv
    	\notag
    	\\ 
    	&=
    	\sum_{\setv \subseteq \setu}
    	\int_{\R^{|\setv|}} 
    	\sum_{\substack{\bstau_\setu \in \{1:\alpha\}^{|\setu|} \\ \text{s.t. } \tau_j = \alpha \text{ for } j \in \setv \\ \text{and } \tau_j < \alpha \text{ for } j \notin \setv}}
    	\|(\partial^{(\bstau_\setu)}_{\bsy_\setu} u(\cdot, \cdot_\setu))(\cdot,(\bsy_\setv, \bszero_{\setu\setminus\setv}))\|_V^2
    	\,
    	\rho_\setv(\bsy_\setv)
    	\rd \bsy_\setv
    	\notag
    	\\
    	&\le
    	\bsb_\setu^2
    	\,
    	C_{\kappa,\alpha} \, \|f\|_{V^*}^2 
    	\sum_{\setv \subseteq \setu}
    	\int_{\R^{|\setv|}} 
    	\frac1{(a_{\min}(\bsy_\setv))^2} 
    	\,
    	\rho_\setv(\bsy_\setv)
    	\rd \bsy_\setv
	    ,
	\end{align}
	where we applied \RefProp{prop:sum-norm-Va} with $\bsy_\setu = (\bsy_\setv, \bszero_{\setu\setminus\setv}) \in \R^\N_\setv \subseteq \R^\N_\setu$ for each $\setv \subseteq \setu$.
	Now we estimate the sum in the last expression.
	Our strategy is similar to that in \cite[proof of Theorem~13]{HS19}.
	Note that for any $\bsy_\setv \in \R^\N_\setv$ we have
	\begin{align*}
    	\frac1{(a_{\min}(\bsy_\setv))^2} 
    	&\le
    	\exp\left(2\sup_{\bsx\in D} \sum_{j \in \setv} |y_j| |\phi_j(\bsx)|\right)
    	\le
    	\exp\left(2  \left(\sup_{j \in \setv} |y_j| \, b_j \right) 
	                 \left(\sup_{\bsx\in D} \sum_{j \in \setv} \frac{|\phi_j(\bsx)|}{b_j} \right)\right)
    	\\
    	&=
    	\exp\left( 2  \kappa  \sup_{j \in \setv} |y_j| \, b_j \right) 
    	\le 
    	\exp\left( 2  \kappa  \sum_{j \in \setv} |y_j| \, b_j \right) 
    	.
	\end{align*}
	Therefore, we have
	\begin{align}\label{eq:boundBnorm2}
    	\int_{\R^{|\setv|}} \frac1{(a_{\min}(\bsy_\setv))^2} \, \rho_\setv(\bsy_\setv) \rd \bsy_\setv
    	&\le
    	\int_{\R^{|\setv|}} 
    	\exp\left( 2  \kappa  \sum_{j \in \setv} |y_j| \, b_j \right)
    	\,
    	\rho_\setv(\bsy_\setv)
    	\rd \bsy_\setv
    	\notag
    	\\
    	&=
    	\prod_{j \in \setv}
    	\left(
    	2 \exp\left(2(\kappa b_j)^2\right) \int_0^\infty \frac{\exp(-(y-2 \kappa b_j)^2/2)}{\sqrt{2\pi}} \rd y
    	\right)
    	\notag
    	\\
    	&\le
    	\exp\left( \sum_{j \in  \setv} \left[ 2(\kappa b_j)^2 + \frac{4 \kappa b_j}{\sqrt{2 \pi}} \right] \right)
    	,
	\end{align}
	where we used that the integral in the second line can be interpreted as the cumulative standard normal distribution evaluated at $2 \kappa b_j$, and this can be bounded by $\exp(2 (2 \kappa b_j) / \sqrt{2\pi}) / 2$, see, e.g., ~\cite[p.~355]{GKNSSS15}.
	Inserting~\eqref{eq:boundBnorm2} into~\eqref{eq:boundBnorm1} we obtain
	\begin{align*}
    	\|u(\cdot, \cdot_\setu)\|_{H_{\alpha,0,\rho,|\setu|}(\R^{|\setu|};V)}^2
    	&\le
    	\bsb_\setu^2
    	\,
    	C_{\kappa,\alpha} \, \|f\|_{V^*}^2 
    	\sum_{\setv \subseteq \setu}
    	\exp\left( \sum_{j \in \setv} \left[ 2(\kappa b_j)^2 + \frac{4 \kappa b_j}{\sqrt{2 \pi}} \right] \right)
    	\\
    	&\le
    	\bsb_\setu^2
    	\,
    	C_{\kappa,\alpha} \, \|f\|_{V^*}^2 \,
    	2^{|\setu|} \, \exp\left(\sum_{j \ge 1} \left[ 2(\kappa b_j)^2 + \frac{4 \kappa b_j}{\sqrt{2 \pi}} \right] \right)
    	.
	\end{align*}
	Since $\{b_j\}_{j\ge 1} \in \ell^{p^*}(\N) \subseteq \ell^1(\N) \subset \ell^2(\N)$ for $p^* \in (0,1]$ the sum in the last expression is finite.
	Taking the square root on both sides finishes the proof.
\end{proof}

Combining \RefLem{lem:norm-uu-u-truncated-u} and \RefLem{lem:norm-u} we obtain bounds for the norms of~$u_\setu$ and~$G(u_\setu)$.
Note that the arguments used to arrive at these bounds are based on the weak formulation with $u(\cdot,\bsy_\setu) \in V$.
They remain valid for the approximation~$u_\setu^{h_\setu}$, in which we combine the FE approximations  $u^{h_\setu}(\cdot,\bsy_\setv) \in V^{h_\setu}$ for all $\setv \subseteq \setu$, since $V^{h_\setu} \subset V$, with constants independent of~$h_\setu$.
Note that this requires us to use the same FE mesh diameter $h_\setu$ for all the $\setv$-truncated solutions which we use to calculate~$u_\setu^{h_\setu}$, see~\eqref{eq:uu-hu}.

\begin{lemma}\label{lem:norm-uu-Guu}
    Assume the sequence $\{b_j\}_{j\ge 1}$ satisfies the assumptions of \RefProp{prop:sum-norm-Va} for a given $\alpha \in \N$, and, additionally that $\{b_j\}_{j\ge 1} \in \ell^{p^*}(\N)$ for some $p^* \in (0,1]$.
    Let $G \in V^*$ be a bounded linear functional.
    Then it holds that $u_\setu \in H_{\alpha,0,\rho,|\setu|}(\R^{|\setu|};V)$ with
    \begin{align*}
        \|u_\setu\|_{H_{\alpha,0,\rho,|\setu|}(\R^{|\setu|};V)}
        &\le
        \bsb_\setu \, 2^{|\setu|/2}
        \,
        C_{\kappa, \alpha}' \,
        \|f\|_{V^*}
        ,
    \intertext{and $G(u_\setu) \in H_{\alpha,0,\rho,|\setu|}(\R^{|\setu|})$ with}
        \|G(u_\setu)\|_{H_{\alpha,0,\rho,|\setu|}(\R^{|\setu|})} 
        &\le
        \bsb_\setu \, 2^{|\setu|/2} \,
        C_{\kappa, \alpha}' \,
        \|G\|_{V^*} \,
        \|f\|_{V^*}
        ,
	\end{align*}
	with $C_{\kappa, \alpha}'$ given in~\eqref{eq:def:C'}.
	Moreover, for $V^{h_\setu} \subset V$, we have $u_\setu^{h_\setu} \in H_{\alpha,0,\rho,|\setu|}(\R^{|\setu|};V)$ and $G(u_\setu^{h_\setu}) \in H_{\alpha,0,\rho,|\setu|}(\R^{|\setu|})$ with the same bounds on their norms as given for $u_\setu$ and~$G(u_\setu)$ up to a multiplicative constant independent of $h_\setu$.
\end{lemma}

\section{Finite element approximation error}\label{sec:FE-approximation-error}

For bounding the error we also need results on the FE approximation error, which is the error between the true solution $u(\cdot,\bsy_\setv) \in V$ and the discretized solution $u^{h_\setu}(\cdot,\bsy_\setv) \in V^{h_\setu} \subset V$ of the weak form, cf.~\eqref{eq:PDE-weak-form-v} and~\eqref{eq:PDE-weak-form-v-h}, where~\eqref{eq:uu-hu} dictates which FE approximations we need to take.
The FE approximation error depends on the spatial regularity of the solution which depends on the smoothness of the domain~$D$, the right hand side of the PDE~$f$ and the spatial regularity of the diffusion coefficient $a(\cdot,\bsy_\setv) = \exp(Z(\cdot,\bsy_\setv))$.

Our assumptions are quite standard, see, e.g., \cite[Proposition~15]{HS19} and also \cite[Theorem~2.4 and Theorem~2.5]{DKLNS14}.
We will assume that, given
\begin{align}\label{eq:D-smoothness}
  D \subset \R^d \quad \text{ is a bounded polyhedron with plane faces}
  ,
\end{align}
there exists a $t \in (0, \infty)$ such that
\begin{align}\label{eq:a-smoothness}
  a \in L_{p,\rho}(\R^\N;C^t(\overline D)) \qquad \text{ for all } p \in [1, \infty)
  ,
\end{align}
and for which there exists a positive sequence $\{b_j\}_{j\ge1}$, with $0 < b_j \le 1$ for all~$j$, satisfying condition~\eqref{eq:condition:kappa} and~\eqref{eq:condition:pstar-summability}, which will be further strengthened, see \RefThm{thm:main-theorem}, as well as
\begin{align}\label{eq:f-G-smoothness}
  f \in H^{-1+t}(D)
  \qquad\text{ and }\qquad
  G \in H^{-1+t}(D)
  ,
\end{align}
then, for all $p \in [1, \infty)$ and $\tau < 2 t$ there is a constant $C > 0$ such that, for any $h_\setv > 0$, we have an FEM, using a continuous piecewise polynomial basis of total degree $r \ge \lceil \tau / 2 \rceil$, for which
\begin{align}\label{eq:Lprho-FE-error-G}
  \left\| G(u(\cdot, \cdot_\setv)) - G(u^{h_\setv}(\cdot, \cdot_\setv)) \right\|_{L_{p,\rho}(\R^\N)}
  &\le
  C \, h_\setv^\tau
  .
\end{align}
Here $C$ depends on the spatial regularity of $D$, $a$, $f$ and $G$ through~\eqref{eq:D-smoothness}--\eqref{eq:f-G-smoothness}, but is independent of~$h_\setv$ and~$\setv$.
We note that to handle singularities, due to, e.g., reentrant corners for $d=2$, one either has to change the norms to incorporate a weight function as in \cite[Proposition~2.3 and Remark~2.4]{HS19ML} or use local mesh refinement as, e.g., in~\cite{GKNSS18}.
We refer to \cite{HS19ML} for the definitions of the norms and omit such details here.

\section{MDFEM}\label{sec:mainresult}

In the next subsections we explain the cost model for the MDFEM algorithm and select the active set, cubature rules and finite element approximations based on \emph{a priori} error estimates.
The main complexity result will be presented in \RefThm{thm:main-theorem}, after which we compare the MDFEM with the QMCFEM and MLQMCFEM algorithms.

\subsection{Computational cost}\label{sec:cost-model}

The total %computational 
cost of the MDFEM algorithm~\eqref{eq:MDFEM-algorithm} is comprised of the costs of computing $Q_{\setu,n_\setu}(G(u^{h_\setu}_\setu))$ for all $\setu \in \setU(\epsilon)$, where $\setU(\epsilon)$ is the ``active set'' determined to reach a given error request~$\epsilon > 0$.
For each $\setu \in \setU(\epsilon)$ the cost of computing $Q_{\setu,n_\setu}(G(u^{h_\setu}_\setu))$ is given by $n_\setu$ times the cost of computing $G(u^{h_\setu}_\setu)$, see~\eqref{eq:cost}.
Based on the decomposition~\eqref{eq:G-uu-hu} the cost of evaluating $G(u^{h_\setu}_\setu)$ is bounded by $2^{|\setu|}$ times the cost of evaluating the FE approximation of the $\setu$-truncated solution, where we assume the cost of approximating $u^{h_\setu}(\cdot, \bsy_\setu)$ to be dominating those of $u^{h_\setu}(\cdot, \bsy_\setv)$ for all $\setv \subseteq \setu$ since the stiffness matrix involves calculating~\eqref{eq:a-Z-truncated} at a cost which we will estimate at $O(|\setv|)$.
Technically this cost could be avoided by using a Gray code ordering of enumerating the sets $\setv \subseteq \setu$.
Hence, we have
\begin{align*}
  \cost(Q_\epsilon) 
  =
  O\!\left(
  \sum_{\setu \in \setU(\epsilon)}
  n_\setu \,
  2^{|\setu|}
  \times 
  \text{cost of evaluating } G(u^{h_\setu}(\cdot, \cdot_\setu))
  \right)
  .
\end{align*}
For each $\bsy_\setu$, and hence for each $\bsy_\setv$ obtained from this $\bsy_\setu$ for $\setv \subseteq \setu$, the cost of evaluating the FE approximation $u^{h_\setu}(\cdot, \bsy_\setv)$ is given by the cost of assembling the stiffness matrix plus the cost of solving the linear system.
Due to the locality of the $O(h_\setu^{-d})$ basis functions of $V^{h_\setu}$, the stiffness matrix is sparse and  has $O(h_\setu^{-d})$ nonzero entries.
Each entry in turn needs at most $O(|\setu|)$ operations to evaluate the diffusion parameter, cf.~\eqref{eq:a-Z-truncated}, which could be avoided, see the previous remark, but we will leave this cost in, since there are other terms more dominating, cf.~\eqref{eq:uniformly-bounded} forthcoming.
We assume
\begin{align}
  \label{eq:cost-of-solve}
  &\text{cost of solving the linear system} 
  =
  O(h_\setu^{-d \, (1+\ddelta)})
  =
  O(h_\setu^{-\dd})
  ,
  \\
  \notag
  &\qquad\text{with } \dd = d \, (1+\ddelta)
  \text{ for some } \ddelta \ge 0
  .
\end{align}
E.g., in \cite[Section~10]{HS19} it is assumed that the cost is nearly linear and that $\ddelta > 0$ can be chosen arbitrarily small, making use of~\cite[Corollary~17]{Her19}.
Thus, using $(m+1) \le 2 m$ for $m \in \N$, we have
\begin{align*}
  \text{cost of evaluating } G(u^{h_\setu }(\cdot, \cdot_\setu))
  =
  O(h_\setu^{-d} \, |\setu| + h_\setu^{-d \, (1+\ddelta)})
  =
  O(h_\setu^{-d \, (1+\ddelta)} \, |\setu|)
  .
\end{align*}
To simplify the further exposition, we will write $\pounds_\setu := 2^{|\setu|} |\setu|$ and $\dd := d \, (1+\ddelta)$.
Therefore, the total computational cost of the MDFEM is
\begin{align}\label{eq:cost-MDFEM}
  \cost(Q_\epsilon)
  =
  O\!\left(  
    \sum_{\setu \in \setU(\epsilon)} n_\setu \, h_\setu^{-\dd} \, \pounds_\setu
  \right)
  .
\end{align}

\subsection{Error analysis}

We will now give an \emph{a priori} estimate of the total error of the MDFEM algorithm.
We will use the higher-order convergence of our QMC based cubature rules from \RefSec{sec:HOQMC}, together with the required parametric regularity results of the PDE solution from \RefSec{sec:parametric-regularity}, and the convergence of the FE approximations from \RefSec{sec:FE-approximation-error}.
For the convergence of the FE approximation we remind the reader that similar assumptions as given in~\eqref{eq:D-smoothness}--\eqref{eq:f-G-smoothness}, possibly extended with weighted norms or local mesh refinements, are needed to allow the convergence of~\eqref{eq:Lprho-FE-error-G} to hold.

The next result holds for any choice of~$\setU(\epsilon)$ and with the MDFEM algorithm $Q_\epsilon$ making use of this active set.
The actual choice of the active set, as well as the cubature rules and FE approximations, to reach a certain error request $\epsilon$ will be shown in the next sections based on the error bound in the next result.

We remark that for $\setu = \emptyset$ there is no integral to approximate since $I_\emptyset(F_\emptyset) = F_\emptyset$ with $F_\emptyset = \FF(\bszero)$, which requires a single function evaluation to be computed.
Hence, for $n_\emptyset = 0$ the absolute value of the cubature error is $\|F_\emptyset\|_{H_{\alpha,0,\rho,0}} = |F_\emptyset|$, while for $n_\emptyset \ge 1$, we set $Q_{\emptyset,n_\emptyset}(F_\emptyset) = F_\emptyset = \FF(\bszero)$ and the cubature error is~$0$.
To cover this case easily we define $0^0 = 1$ and $C_{\alpha,\lambda,0} := 1$.

\begin{proposition}\label{prop:bound-total-error}
     Let the set $\setU(\epsilon) \subset \N^\N$ be given and assume that the conditions of \RefLem{lem:norm-uu-Guu} are satisfied for a given $\alpha \in \N$, and hence there is a sequence $\{b_j\}_{j\ge1} \in \ell^{p^*}(\N)$ for some $p^* \in (0,1]$, with $0 < b_j \le 1$ and $\kappa < \ln(2) / \alpha$.
     Let $G \in V^*$ be a bounded linear functional.
     Assume $D$, $a$, $f$ and $G$ have sufficient spatial regularity such that the application of $G$ to the FE approximations of the truncated solutions converge like $O(h^\tau)$ as in~\eqref{eq:Lprho-FE-error-G}.
     Let the cubature rules be defined as the transformed interlaced polynomial lattice rules with interlacing factor~$\alpha$ as in \RefThm{thm:HO-Q-R-truncated}, and hence $\alpha \ge 2$, and, with $m_\setu \in \N$, using $n_\setu = 2^{m_\setu} \ge 2$ points.
     For $n_\setu = 0$ and $n_\setu = 1$ we take the zero approximation.
     Then the error of the MDFEM algorithm $Q_\epsilon$, see~\eqref{eq:MDFEM-algorithm}, based on the given set $\setU(\epsilon)$ can be bounded as
\begin{multline}\label{eq:bound-total-error}
    \left|\II(G(u)) -Q_\epsilon(G(u))\right|
    \lesssim
    \sum_{\setu \notin \setU(\epsilon)} \gamma_\setu \, M_\setu
    \\+
    \max\left\{1, \max_{\setu \in \setU(\epsilon)} \left(\frac{\ln(n_\setu)}{|\setu|}\right)^{\alpha_1 |\setu|}\right\}
    \sum_{\setu \in \setU(\epsilon)} 
    \left( 
    \frac{\gamma_\setu \, C_{\alpha,\lambda,|\setu|} \, |\setu|^{\alpha_1|\setu|}}{\max\{1, n_\setu^\lambda\}} 
    +
    2^{|\setu|} \, h_\setu^\tau 
    \right)
    ,
\end{multline}
for any $\lambda \in [1,\alpha)$, with the corresponding truncation points for the cubature rules chosen in accordance with this $\lambda$, i.e., $T_\setu = 2 + 2 \sqrt{\lambda \ln(n_\setu)}$, and
where $\alpha_1=\alpha/2+1/4$, $C_{\alpha,\lambda,|\setu|}$ is given by~\eqref{eq:def:C-alpha-lambda-s}, $M_\setu = M^{|\setu|}$ with $M$ given by~\eqref{eq:def:M} and $\gamma_\setu := \prod_{j \in \setu} \gamma_j$ with
\begin{align}\label{eq:gamma-j-choice}
  \gamma_j
  &=
  \sqrt{2} \, b_j
  ,
\end{align}
and $\gamma_\emptyset := 1$.
The first term in~\eqref{eq:bound-total-error} corresponds to the truncation error, while the second term corresponds to the combined error due to the FE approximations and QMC cubatures.
The hidden constant in~\eqref{eq:bound-total-error} is independent of the choice of the active set $\setU(\epsilon)$ and the choices of $n_\setu$ and $h_\setu$.
\end{proposition}

\begin{proof}
The error splits into three terms:
\begin{multline*}
    \II(G(u)) - Q_\epsilon(G(u))
    =
    \left(
      \II(G(u)) - \sum_{\setu \in \setU(\epsilon)} I_\setu(G(u_\setu))
    \right)
    \\
    +
    \left(
      \sum_{\setu \in \setU(\epsilon)} I_\setu\left( G(u_\setu) - G(u^{h_\setu}_\setu) \right)
    \right)
    +
    \left(
      \sum_{\setu \in \setU(\epsilon)} \left(I_\setu - Q_{\setu,n_\setu}\right)(G(u^{h_\setu}_\setu))
    \right)
    .
\end{multline*}
The first term is the truncation error from truncating to only a finite number of decomposed elements.
The second term stems from the spatial discretization of the FE approximations.
The last term is the cubature error arising from using cubature rules to approximate the integrals.

\paragraph{1.}
Since we know that $G(u_\setu) \in H_{\alpha,0,\rho,|\setu|}(\R^{|\setu|})$, see \RefLem{lem:norm-uu-Guu}, we can make use of~\eqref{eq:bound-Iu} in \RefProp{prop:integrability}.
Therefore, with $\{\gamma_\setu\}_{|\setu|<\infty}$ a sequence of positive weights, the truncation error can be bounded as
\begin{align*}
    \left|\II(G(u)) - \sum_{\setu \in \setU(\epsilon)} I_\setu(G(u_\setu))\right|
    &\le 
    \sum_{\setu \notin \setU(\epsilon)} 
    \|G(u_\setu)\|_{H_{\alpha,0,\rho,|\setu|}(\R^{|\setu|})}
    \, M_\setu
    \\
    &\le
    \left(\sup_{\setu \notin \setU(\epsilon)} 
    \gamma_\setu^{-1} \, \|G(u_\setu)\|_{H_{\alpha,0,\rho,|\setu|}(\R^{|\setu|})}
    \right)
    \left(
      \sum_{\setu \notin \setU(\epsilon)} \gamma_\setu \, M_\setu
    \right)
    .
\end{align*}

\paragraph{2.}
For the error due to the FE approximations we have
\begin{align*}
    \left| \sum_{\setu \in \setU(\epsilon)} I_\setu\left(G(u_\setu) - G(u^{h_\setu}_\setu)\right) \right|
    &\le
    \sum_{\setu \in \setU(\epsilon)} 
    \int_{\R^{|\setu|}} \left| G(u_\setu(\cdot, \bsy_\setu)) - G(u^{h_\setu}_\setu(\cdot, \bsy_\setu)) \right|  \rd\mu(\bsy_\setu)
    .
\end{align*}
Due to the linearity of $G$ we have
\begin{align*}
    \left|G(u_\setu(\cdot, \bsy_\setu))- G(u^{h_\setu}_\setu(\cdot, \bsy_\setu))\right|
    &=
    \left|
    G\!\left(\sum_{\setv \subseteq \setu}(-1)^{|\setu|-|\setv|} \, u(\cdot, \bsy_\setv)\right)
    -
    G\!\left(\sum_{\setv \subseteq \setu}(-1)^{|\setu|-|\setv|} \, u^{h_\setu}(\cdot, \bsy_\setv)\right)
    \right|
    \\
    &=
    \left|\sum_{\setv \subseteq \setu}(-1)^{|\setu|-|\setv|} 
    \left(
    G(u(\cdot, \bsy_\setv))
    -
    G(u^{h_\setu}(\cdot, \bsy_\setv))
    \right)\right|
    \\
    &\le 
    \sum_{\setv \subseteq \setu}
    \left|
    G(u(\cdot, \bsy_\setv))
    -
    G(u^{h_\setu}(\cdot, \bsy_\setv))
    \right|
    .
\end{align*}
Thus, using the assumption of an FE approximation error bound as in~\eqref{eq:Lprho-FE-error-G}, we have
\begin{align*}
  &\int_{\R^{|\setu|}} 
  \left|G(u_\setu(\cdot, \bsy_\setu))- G(u^{h_\setu}_\setu(\cdot, \bsy_\setu))\right|  \rd\mu(\bsy_\setu) 
  \le
  \int_{\R^{|\setu|}} 
  \sum_{\setv \subseteq \setu}
  \left| G(u(\cdot, \bsy_\setv)) - G(u^{h_\setu}(\cdot, \bsy_\setv)) \right|
  \rd\mu(\bsy_\setu)
  \\
  &\qquad=
  \sum_{\setv \subseteq \setu} \int_{\R^{|\setv|}}
  \left|
  G(u(\cdot, \bsy_\setv))
  -
  G(u^{h_\setu}(\cdot, \bsy_\setv))
  \right|
  \rd\mu(\bsy_\setv)
  \lesssim
  \sum_{\setv \subseteq \setu} 
  h_\setu^\tau 
  \lesssim
  2^{|\setu|} \, h_\setu^\tau
  ,
\end{align*}
with the hidden constant independent of $h_\setu$ and~$\setu$ but dependent on the spatial regularity conditions of $D$, $a$, $f$ and $G$, see, e.g.,~\eqref{eq:D-smoothness}--\eqref{eq:f-G-smoothness}.
Hence the error incurred by the FE approximation can be bounded as
\begin{align*}
  \sum_{\setu \in \setU(\epsilon)} 
  I_\setu\left
  (G(u_\setu)
  -
  G(u^{h_\setu}_\setu)
  \right)
  \lesssim
  \sum_{\setu \in \setU(\epsilon)} 
  2^{|\setu|}  h_\setu^\tau 
 ,
\end{align*}
with the hidden constant independent of $h_\setu$.

\paragraph{3.}
From \RefLem{lem:norm-uu-Guu} we also know $G(u_\setu^{h_\setu}) \in H_{\alpha,0,\rho,|\setu|}(\R^{|\setu|})$.
For $\setu \ne \emptyset$ we obtain from \RefThm{thm:HO-Q-R-truncated}, which holds for $n_\setu = 2^{m_\setu} \ge 2$,
\begin{align*}
  \left| \left(I_\setu-Q_{\setu,n_\setu}\right) 
  (G(u^{h_\setu}_\setu)) \right| 
  &\le
  \|G(u^{h_\setu}_\setu)\|_{H_{\alpha,0,\rho,|\setu|}(\R^{|\setu|})} \,
  C_{\alpha,\lambda,|\setu|} \,
  \frac{(\ln(n_\setu))^{\alpha_1|\setu|}}{n_\setu^\lambda}
  \\
  &\le
  \|G(u^{h_\setu}_\setu)\|_{H_{\alpha,0,\rho,|\setu|}(\R^{|\setu|})} \,
  C_{\alpha,\lambda,|\setu|} \,
  \frac{\max\left\{1,(\ln(n_\setu))^{\alpha_1|\setu|}\right\}}{\max\{1, n_\setu^\lambda\}}
  ,
\end{align*}
for any $\lambda \in [1,\alpha)$, and where $\alpha_1 = \alpha/2 + 1/4$ and $C_{\alpha,\lambda,|\setu|}$ is as defined in~\eqref{eq:def:C-alpha-lambda-s}.
The second bound also holds for $n_\setu = 0$ and $n_\setu = 1$, for which in both cases we take the zero approximation, since $M_\setu = M^{|\setu|} \le C_{\alpha,\lambda,|\setu|}$ which follows from $C_{\alpha,\lambda,|\setu|} \ge (2/\sqrt{\ln(2)} + 2\sqrt{\lambda})^{(\alpha+1/2)|\setu|} \ge (2/\sqrt{\ln(2)} + 2)^{|\setu|} \approx 4.40224^{|\setu|}$ and $M < 2.767$, and we used $C_{4,\alpha,\lambda,s} \ge 1$, see \RefProp{prop:integrability} and the constants referenced from~\eqref{eq:def:C-alpha-lambda-s}.
Therefore, with $\{\gamma_\setu\}_{|\setu|<\infty}$ a sequence of positive weights, the cubature error is bounded as
\begin{multline*}
    \left|
    \sum_{\setu \in \setU(\epsilon)} \left(I_\setu-Q_{\setu,n_\setu}\right)(G(u^{h_\setu}_\setu))
    \right|
    \le 
    \left(\sup_{\setu \in \setU(\epsilon)} 
      \gamma_\setu^{-1} \, \|G(u^{h_\setu}_\setu)\|_{H_{\alpha,0,\rho,|\setu|}(\R^{|\setu|})}
    \right)
    \\
    \max\left\{ 1, \max_{\setu \in \setU(\epsilon)} \left(\frac{\ln(n_\setu)}{|\setu|}\right)^{\alpha_1 |\setu|} \right\}
    \sum_{\setu \in \setU(\epsilon)} \frac{\gamma_\setu \, C_{\alpha,\lambda,|\setu|} \, |\setu|^{\alpha_1 |\setu|}}{\max\{1, n_\setu^\lambda\}}
    ,
\end{multline*}
where for $\setu = \emptyset$ we interpret $0^0$ as $1$.
We remark that we deliberately pulled out $|\setu|^{-\alpha_1 |\setu|}$ to control the logarithmic factor in $n_\setu$ later in \RefThm{thm:main-theorem}.
This technique was also used in~\cite{NN21} for the MDFEM in the uniform case.

To show that the above formula also holds for $\setu = \emptyset$ we recall from \RefRem{rem:infinite-variate-RKHS} that we have $\|F_\emptyset\|_{H_{\alpha,0,\rho,0}} = |F_\emptyset| = |\FF(\bszero)|$.
Thus, for $n_\emptyset = 0$ the absolute value of the cubature error is $|\FF(\bszero)|$ while, if for $n_\emptyset \ge 1$ we set $Q_{\emptyset,n_\emptyset}(F_\emptyset) = F_\emptyset$, the error is zero for $n_\emptyset \ge 1$.
Hence for any $n \in \N_0$ we have
\begin{align*}
  |I_\emptyset(F_\emptyset) - Q_{\emptyset,n_\emptyset}(F_\emptyset)|
  &\le
  \|F_\emptyset\|_{H_{\alpha,0,\rho,0}} \, \max\{1, n_\emptyset\}^{-\lambda}
  .
\end{align*}

\smallskip

We have now bounded all three contributions to the error.
For the truncation error and the cubature error we still want to choose the weights $\gamma_\setu$.
For the truncation error we obtain from \RefLem{lem:norm-uu-Guu} that
\begin{align*}
        \sup_{\setu \not\in \setU(\epsilon)} \gamma_\setu^{-1} \,
        \|G(u_\setu)\|_{H_{\alpha,0,\rho,|\setu|}(\R^{|\setu|})}
        &\le
        \sup_{|\setu| < \infty} \gamma_\setu^{-1} \,
        \|G(u_\setu)\|_{H_{\alpha,0,\rho,|\setu|}(\R^{|\setu|})}
        \\
        &\le
        \left(
        \sup_{|\setu| < \infty} 
        \gamma_\setu^{-1} \,
        \bsb_\setu \, 2^{|\setu|/2} 
        \right)
        \,
        C_{\kappa, \alpha}' \,
        \|G\|_{V^*} \,
        \|f\|_{V^*}
        ,
\end{align*}
with $\bsb_\setu = \prod_{j\in\setu} b_j$ and $C_{\kappa, \alpha}' < \infty$, see~\eqref{eq:def:C'}, under the assumptions of $\{b_j\}_{j\ge1} \in \ell^{p^*}(\N)$ for $p^* \in (0,1]$.
For the cubature error we obtain, also from \RefLem{lem:norm-uu-Guu}, that such a bound holds with $u_\setu^{h_\setu}$ in place of $u_\setu$ and hence we also have
\begin{align*}
        \sup_{\setu \in \setU(\epsilon)} \gamma_\setu^{-1} \,
        \|G(u_\setu^{h_\setu})\|_{H_{\alpha,0,\rho,|\setu|}(\R^{|\setu|})}
        &\le
        \sup_{|\setu| < \infty} \gamma_\setu^{-1} \,
        \|G(u_\setu^{h_\setu})\|_{H_{\alpha,0,\rho,|\setu|}(\R^{|\setu|})}
        \\
        &\lesssim
        \left(
        \sup_{|\setu| < \infty} 
        \gamma_\setu^{-1} \,
        \bsb_\setu \, 2^{|\setu|/2} 
        \right)
        \,
        C_{\kappa, \alpha}' \,
        \|G\|_{V^*} \,
        \|f\|_{V^*}
        .
\end{align*}
By choosing $\gamma_j = \sqrt{2} \, b_j$, and with $\gamma_\setu = \prod_{j \in \setu} \gamma_j$, we have
\begin{align*}
  \sup_{|\setu| < \infty} 
    \gamma_\setu^{-1} \,
    \bsb_\setu \, 2^{|\setu|/2}
  &=
  1
  .
\end{align*}
Combining all three errors we obtain the claimed bound for the total error.
\end{proof}

\begin{remark}\label{rem:FF-in-HH}
In the previous proof we made use of a supremum-norm over all subspaces of the infinite-variate space, namely $\sup_{|\setu| < \infty} \gamma_\setu^{-1} \, \|F_\setu\|_{H_{\alpha,0,\rho,|\setu|}(\R^{|\setu|})}$.
Since we know from \RefLem{lem:norm-uu-Guu} that $\|F_\setu\|_{H_{\alpha,0,\rho,|\setu|}(\R^{|\setu|})} \lesssim \bsb_\setu \, 2^{|\setu|/2}$, where $\bsb_\setu = \prod_{j \in \setu} b_j$, we have chosen product weights $\gamma_\setu = \prod_{j \in \setu} \gamma_j$ with $\gamma_j = \sqrt2 \, b_j$ such that this supremum-norm is finite.
This same choice of weights $\gamma_\setu$ also shows that our infinite-variate function $\FF = \sum_{|\setu| < \infty} F_\setu$ has finite norm in the infinite-variate reproducing kernel Hilbert space $\HH_{\alpha,\rho,\bsgamma}(\R^\N)$ which we introduced in \RefRem{rem:infinite-variate-RKHS}, since
\begin{align*}
  \|\FF\|_{\HH_{\alpha,\rho,\bsgamma}(\R^\N)}^2
  &=
  \sum_{|\setu| < \infty} \gamma_\setu^{-1} \, \|F_\setu\|_{H_{\alpha,0,\rho,|\setu|}(\R^{|\setu|})}^2
  \lesssim
  \sum_{|\setu| < \infty} \prod_{j \in \setu} \frac{2 \, b_j^2}{\sqrt2 \, b_j}
  =
  \prod_{j \ge 1} (1 + \sqrt2 \, b_j)
  ,
\end{align*}
which is finite when $\sum_{j \ge 1} b_j < \infty$ (using the technique as in the proof of \RefProp{prop:active-set-truncation-error}).
The summability of the $b_j$ is implied by our assumption in \RefProp{prop:bound-total-error} which demands $\{b_j\}_{j\ge1} \in \ell^{p^*}(\N)$ for some $p^* \in (0,1]$.
\end{remark}

\subsection{Selection of the active set}\label{sec:active-set}

Based on the expression of the truncation error in \RefProp{prop:bound-total-error} we can choose the active set to reach a truncation error upper bounded by~$\epsilon/2$ up to multiplicative constants.

\begin{proposition}\label{prop:active-set-truncation-error}
	Under the conditions of \RefProp{prop:bound-total-error} with $p^* \in (0,1)$, for which $\{b_j\}_{j\ge1} \in \ell^{p^*}(\N)$, let the MDFEM active set be chosen by
\begin{align}\label{eq:active-set}
  \setU(\epsilon)
  =
  \setU(\epsilon, p^*)
  &:=
  \left\{
    \setu
    :
    \gamma_\setu \, M_\setu 
    >
    \left(
      \frac{\epsilon/2}{\sum_{|\setv| < \infty} (\gamma_\setv \, M_\setv)^{p^*}}
    \right)^{1/(1-p^*)}
  \right\}
  ,
\end{align} 
    with $\gamma_j = \sqrt{2} \, b_j$ for all $j$ as in~\eqref{eq:gamma-j-choice} and with $\gamma_\setu = \prod_{j\in\setu} \gamma_j$.
    Then
	\begin{align*}
	  \sum_{|\setu| < \infty} (\gamma_\setu \, M_\setu)^{p^*}
	  &<
	  \infty
	  ,
	\end{align*}
	and the MDFEM truncation error is bounded as
	\begin{align*}
        \left|
          \II(G(u)) - \sum_{\setu \in \setU(\epsilon)} I_\setu(G(u_\setu))
        \right|
        &\lesssim
        \frac{\epsilon}{2}
        .
	\end{align*}
\end{proposition}
\begin{proof}
For the first claim we have the implications, with $\gamma_\setu = \prod_{j\in\setu} \gamma_j$ and $M_\setu = M^{|\setu|}$,
\begin{align*}
  &\sum_{|\setu| < \infty} (\gamma_\setu \, M_\setu)^{p^*}
  =
  \lim_{s \to \infty} \sum_{\setu \subseteq \{1:s\}} (\gamma_\setu \, M_\setu)^{p^*}
  =
  \lim_{s \to \infty} \prod_{j=1}^s \left( 1 + (\gamma_j \, M)^{p^*} \right)
  <
  \infty
  \\
  &\Leftrightarrow\quad
  \ln\Big( \prod_{j\ge1} \left( 1 + (\gamma_j \, M)^{p^*} \right) \Big)
  =
  \sum_{j \ge 1} \ln\left( 1 + (\gamma_j \, M)^{p^*} \right) \Big)
  <
  \infty
  \\
  &\Leftarrow\quad
  \sum_{j \ge 1} (\gamma_j \, M)^{p^*}
  <
  \infty
  \\
  &\Leftrightarrow\quad
  (\sqrt2 \, M)^{p^*} 
  \sum_{j \ge 1} b_j^{p^*}
  <
  \infty
  ,
\end{align*}
where we used $\ln(1+x) \le x$ for $x > -1$ and the last line is true since $\{ b_j \}_{j \ge 1} \in \ell^{p^*}(\N)$.
For the second claim we have from \RefProp{prop:bound-total-error}
\begin{multline*}
    \left|
      \II(G(u)) - \sum_{\setu \in \setU(\epsilon)} I_\setu(G(u_\setu)) 
    \right|
    \lesssim
    \sum_{\setu \notin \setU(\epsilon,p^*)} \gamma_\setu \, M_\setu
    =
    \sum_{\setu \notin \setU(\epsilon,p^*)} 
     \left(\gamma_\setu \, M_\setu\right)^{(1-p^*)}  \left(\gamma_\setu \, M_\setu\right)^{p^*} 
    \\
    \le
    \sum_{\setu \notin \setU(\epsilon,p^*)}  
     \frac{\epsilon/2}{\sum_{|\setv| < \infty} (\gamma_\setv \, M_\setv)^{p^*}} \left(\gamma_\setu \, M_\setu\right)^{p^*}
     \le
     \frac{\epsilon}{2}
     .\qedhere
\end{multline*}
\end{proof}

The following result states that both the cardinality of the active set, as well as the cardinalities of each of the individual sets in the active set, increase very slowly with decreasing~$\epsilon$.
The (upper bound of the) cardinality of the active set also gets smaller for decreasing~$p^*$.

\begin{proposition}\label{prop:cardinality-active-set}
	Given $\gamma_\setu = \prod_{j\in\setu} \gamma_j$ with $\{\gamma_j\} \in \ell^{p^*}(\N)$ for some $p^*\in (0,1)$ and $\setU(\epsilon) = \setU(\epsilon,p^*)$ chosen as in~\eqref{eq:active-set}, then
	for any $\epsilon > 0$ it holds that
	\begin{align*}
	  |\setU(\epsilon, p^*)|
	  &<
	  \left(\frac{2}{\epsilon} \right)^{p^*/(1-p^*)}
	  \left(\sum_{|\setu| < \infty} (\gamma_\setu \, M_\setu)^{p^*} \right)^{1/(1-p^*)}
	  \lesssim 
	  \epsilon^{-p^*/(1-p^*)}
	  ,
	\end{align*} 
	and
	\begin{align}\label{eq:U-max-size}
	  d(\epsilon,p^*)
	  =
	  d(\setU(\epsilon,p^*))
	  :=
	  \max_{\setu \in \setU(\epsilon,p^*)} |\setu|
	  =
	  O\!\left( \frac{\ln(\epsilon^{-1})}{\ln(\ln(\epsilon^{-1}))} \right) 
	  =
	  o(\ln(\epsilon^{-1}))
	  \qquad
	  \text{for } \epsilon \to 0
	  .
	\end{align}
\end{proposition}
\begin{proof}
  See~\cite{PW11,NN21}.
\end{proof}

\subsection{Selection of the cubature rules and FEMs}\label{sec:Q+FE}

By the combined error of the FE approximations and the cubature errors from \RefProp{prop:bound-total-error} we can now choose the mesh diameters $h_\setu$ for the FE approximations and the number of cubature points $n_\setu$ for the cubature formulas to balance the combined error and obtain an upper bound of $\epsilon/2$ up to multiplicative constants using the method of Lagrange multipliers to minimize the associated cost.
The maximum appearing in the bound of the next result will be dealt with later in \RefThm{thm:main-theorem}.

\begin{proposition}\label{prop:FE+Q-error}
	Assume the conditions of \RefProp{prop:bound-total-error} hold with $0 < p^* \le (2 + \dd/\tau)^{-1} < \frac12$ and the cost of solving the linear systems for the FEMs are $O(h_\setu^{-\dd})$ as in~\eqref{eq:cost-of-solve} and the application of $G$ to the FE approximations converge like $O(h_\setu^\tau)$ as in~\eqref{eq:Lprho-FE-error-G}.
	Take
	\begin{align*}
	  \lambda
	  &=
	  \frac{(1-p^*)}{p^* \, (1 + \dd/\tau)}
	  \ge
	  1
	  &\text{and}&&
	  \alpha
	  &=
	  \lfloor\lambda\rfloor + 1
	  \ge
	  2
	  ,
	\end{align*}
	and take the FE mesh diameters $h_\setu$ and the number of cubature points $n_\setu$ as the solution to an optimization problem \textup{(}specified as~\eqref{eq:minimize} in the proof of this statement, with solutions~\eqref{eq:hu} and~\eqref{eq:ku} and with $n_\setu = 2^{\log_2(\lfloor k_\setu \rfloor)} \in \N_0$\textup{)}.
	Then the combined error from the FE approximations and the cubature approximations of the MDFEM algorithm, using transformed polynomial lattice rules with interlacing factor~$\alpha$, is bounded as
    \begin{align*}
        \sum_{\setu \in \setU(\epsilon)} 
        \left|
          I_\setu(G(u_\setu)) - Q_{\setu,n_\setu}(G(u^{h_\setu}_\setu))
        \right|
        &\lesssim
        \max\left\{1, \max_{\setu \in \setU(\epsilon)}
        \left(\frac{\ln(n_\setu)}{|\setu|}\right)^{\alpha_1 |\setu|}\right\}
        \;
        \frac{\epsilon}{2}
        ,
    \end{align*}
    where $\alpha_1 =  \alpha/2 + 1/4$, and the computational cost is bounded as
    \begin{align*}
      \cost(Q_\epsilon)
      &\lesssim
      \epsilon^{-1/\lambda - \dd/\tau}
      =
      \epsilon^{-a_\MDFEM}
      &
      \text{with} \quad
      a_\MDFEM
      &:=
      \frac{1}{\lambda} + \frac{\dd}{\tau}
      =
      \frac{p^* + \dd/\tau}{1 - p^*}
      .
    \end{align*}
\end{proposition}

\begin{proof}
From \RefProp{prop:bound-total-error} we obtain
\begin{align}
    \notag
    &
    \sum_{\setu \in \setU(\epsilon)}
    \left| I_\setu(G(u_\setu)) - Q_{\setu,n_\setu}(G(u^{h_\setu}_\setu)) \right|
    \\
    \label{eq:Q-error-1}
    &\qquad\lesssim
    \max
    \left\{
      1, \max_{\setu \in \setU(\epsilon)} \left(\frac{\ln(n_\setu)}{|\setu|}\right)^{\alpha_1 |\setu|}
    \right\}
    \sum_{\setu \in \setU(\epsilon)} 
    \left( 
      \frac{\gamma_\setu \, 2^\lambda \, C_{\alpha,\lambda,|\setu|} \, |\setu|^{\alpha_1 |\setu|}}{(\max\{1, 2 \, n_\setu\})^\lambda} 
      +
      2^{|\setu|} \,
      h_\setu^\tau 
    \right)
    .
\end{align}
Note that compared to \RefProp{prop:bound-total-error} we have (possibly) increased the upper bound by multiplying the numerator by $2^\lambda$ while only multiplying the denominator by $2^\lambda$ for each $\setu$ with $n_\setu \in \N_0$ when $n_\setu \ge 1$.
This is such that the optimization problem we formulate next will give us an upper bound for this error.
We will deal with the max term later in \RefThm{thm:main-theorem} and now put an upper bound of $\epsilon/2$ on the sum over~$\setu$.
We are looking for positive real numbers $k_\setu$ and $h_\setu$ which are the solutions of the following minimization problem: 
\begin{align}\label{eq:minimize}
  \begin{array}{l}
    \displaystyle
    \text{minimize }
    \sum_{\setu \in \setU(\epsilon)} k_\setu \, h_\setu^{-\dd} \pounds_\setu
    \\
    \displaystyle
    \text{subject to }
    \sum_{\setu \in \setU(\epsilon)}
    \left(
      \frac{\gamma_\setu \, 2^\lambda \, C_{\alpha,\lambda,|\setu|} \, |\setu|^{\alpha_1 |\setu|}}{k_\setu^\lambda} 
      +
      2^{|\setu|} \, h_\setu^\tau  
    \right)
    =
    \frac{\epsilon}{2}
    .
  \end{array}
\end{align}
Note that since we set $n_\setu = 2^{\log_2(\lfloor k_\setu \rfloor)} \le k_\setu$, the objective function is an upper bound on the cost~\eqref{eq:cost-MDFEM} while the constraint is an upper bound on the sum over $\setu$ in the error bound~\eqref{eq:Q-error-1} since $\max\{1, 2 \, n_\setu\} = \max\{1, 2^{1+\log_2(\lfloor k_\setu \rfloor)} \} \ge k_\setu$.
See also, e.g., \cite[Section~4.3]{KNPSW17} for a similar technique.

Define $a_\setu := \gamma_\setu \, 2^\lambda \, C_{\alpha,\lambda,|\setu|} \, |\setu|^{\alpha_1 |\setu|}$.
The Lagrangian is then given by
\begin{align*}
    \Lambda(\xi)
    &=
    \sum_{\setu \in \setU(\epsilon)} k_\setu \, h_\setu^{-\dd} \, \pounds_\setu
    +
    \xi
    \left(
    \sum_{\setu \in \setU(\epsilon)}
    \left(
      a_\setu \, k_\setu^{-\lambda}
      +
      2^{|\setu|} \, h_\setu^\tau  
    \right)
    -  \frac{\epsilon}{2}
    \right)
    ,
\end{align*}
with $\xi$ the Lagrange multiplier.
For each $\setu \in \setU(\epsilon)$ we obtain the equations
\begin{align*}
  \begin{cases}
    \displaystyle \frac{\partial \Lambda}{\partial k_\setu} 
    =
    0
    =
    h_\setu^{-\dd} \, \pounds_\setu
    -
    \xi \, \lambda \, a_\setu \, k_\setu^{-\lambda -1}
    ,
    \\[4mm]
    \displaystyle
    \frac{\partial \Lambda}{\partial h_\setu} 
    =
    0
    =
    -\dd \, k_\setu \, h_\setu^{-\dd-1} \, \pounds_\setu 
    + 
    \xi \, \tau \, 2^{|\setu|} \, h_\setu^{\tau-1} 
    .
  \end{cases}  
\end{align*}
From the second equation we obtain (after multiplying with $h_\setu$)
\begin{align}\label{eq:L1}
  h_\setu^{-\dd} \, \pounds_\setu
  &=
  \xi
  \frac{\tau \, 2^{|\setu|} \, h_\setu^\tau}{\dd \, k_\setu}
\end{align}
which we insert into the first equation to obtain
\begin{align}\label{eq:hu}
  2^{|\setu|} \, h_\setu^\tau
  &=
  \frac{\lambda\,\dd}{\tau} \, a_\setu \, k_\setu^{-\lambda}
  &\Rightarrow&&
  h_\setu
  &=
  \left( \frac{\lambda\,\dd}{\tau} \, \frac{a_\setu}{2^{|\setu|}} \right)^{1/\tau} k_\setu^{-\lambda/\tau}
\end{align}
of which we substitute the first form into the constraint to obtain
\begin{align}\label{eq:L3}
  \sum_{\setu \in \setU(\epsilon)} \left( a_\setu \, k_\setu^{-\lambda} + 2^{|\setu|} \, h_\setu^\tau \right)
  &=
  (1 + \lambda\,\dd/\tau ) \, \sum_{\setu \in \setU(\epsilon)} a_\setu \, k_\setu^{-\lambda}
  =
  \frac{\epsilon}{2}
  .
\end{align}
From~\eqref{eq:L1} and~\eqref{eq:hu} we also find
\begin{align}
  \notag
  k_\setu
  &=
  \xi \,
  \frac{\tau}{\dd} \,
  \frac{2^{|\setu|} \, h_\setu^{\tau+\dd}}{\pounds_\setu}
  =
  \xi \,
  \frac{\tau}{\dd} \,
  \frac{2^{|\setu|}}{\pounds_\setu}
  \,
  \left( \frac{\dd \, \lambda}{\tau \, 2^{|\setu|}} \, a_\setu \right)^{(\tau+\dd)/\tau} k_\setu^{-\lambda(\tau+\dd)/\tau}
  \\
  \notag
  \Rightarrow
  k_\setu
  &=
  \left(
  \xi \,
  \left(
    \frac{\dd}{\tau \, 2^{|\setu|}}
  \right)^{\dd/\tau}
  \,
  \frac{\left( \lambda \, a_\setu \right)^{(\tau+\dd)/\tau}}{\pounds_\setu}
  \right)^{\tau/(\tau+\lambda(\tau+\dd))}
  \\
  \label{eq:ku}
  &=
  B
  \,
  \xi^{\tau/(\tau+\lambda(\tau+\dd))}
  \left(\frac{a_\setu^{\tau+\dd}}{2^{|\setu| \dd} \, \pounds_\setu^\tau}\right)^{1/(\tau+\lambda(\tau+\dd))}
  ,
\end{align}
where we have set
\begin{align*}
  B
  =
  B(\dd,\lambda,\tau)
  &:=
  \left(\frac{\dd^{\dd} \, \lambda^{\tau+\dd}}{\tau^{\dd}}\right)^{1/(\tau+\lambda(\tau+\dd))}
  .
\end{align*}
Inserting this expression for $k_\setu$ into~\eqref{eq:L3} we obtain the following expression for the Lagrange multiplier
\begin{align}
  \notag
  \xi^{\tau/(\tau+\lambda(\tau+\dd))}
  &=
  \left(
  \frac{2}{\epsilon} \,
  (1 + \lambda\,\dd/\tau) \, B^{-\lambda}
  \sum_{\setu \in \setU(\epsilon)} a_\setu \,
  \left(\frac{a_\setu^{\tau+\dd}}{2^{|\setu| \dd} \, \pounds_\setu^\tau}\right)^{-\lambda/(\tau+\lambda(\tau+\dd))}
  \right)^{1/\lambda}
  \\
  \label{eq:L-multiplier}
  &=
  \left(
  \frac{2}{\epsilon} \,
  (1 + \lambda\,\dd/\tau) \, B^{-\lambda}
  \underbrace{
  \sum_{\setu \in \setU(\epsilon)}
  \left(a_\setu^\tau \, 2^{\lambda |\setu| \dd} \, \pounds_\setu^{\lambda \tau}\right)^{1/(\tau+\lambda(\tau+\dd))}
  }_{=: K_\epsilon}
  \right)^{1/\lambda}
  .
\end{align}
We require the sum $K_\epsilon$ in~\eqref{eq:L-multiplier} to be uniformly bounded while $\epsilon \to 0$, that is, we require
\begin{align}\label{eq:uniformly-bounded}
    K_\epsilon
    \le
    K
    :=
    \sum_{|\setu| < \infty}
    \left(
      \gamma_\setu^\tau \, 2^{\lambda \tau} \, C_{\alpha,\lambda,|\setu|}^\tau \, |\setu|^{\tau \alpha_1 |\setu|} \,
      2^{\lambda |\setu| \dd} \,
      \pounds_\setu^{\lambda\tau}
      \right)^{1/(\tau + \lambda(\tau+\dd))} 
    &<
    \infty
    .
\end{align}
Since $\gamma_\setu = 2^{|\setu|/2} \prod_{j\in \setu} b_j$ with $\{b_j\}_{j\ge 1} \in \ell^{p^*}(\N)$ and,
both $\pounds_\setu$ and $C_{\alpha,\lambda,|\setu|}$ are at most exponential in $|\setu|$, by applying~\cite[Lemma~1]{NN21} the sum~\eqref{eq:uniformly-bounded} is bounded if the following two conditions are satisfied:
\begin{align}\label{eq:condition:alpha1}
    \frac{\tau \, \alpha_1}{\tau + \lambda \, (\tau+\dd)} < 1
    ,
    \qquad 
    \text{or equivalently,}
    \qquad
    \alpha_1
    <
    1 + \lambda \, (1+\dd/\tau)
    ,
\end{align}
and 
\begin{align}\label{eq:condition:lambda}
    \frac{\tau}{\tau + \lambda \, (\tau+\dd)} \ge p^*
    ,
    \qquad 
    \text{or equivalently,}
    \qquad
    \lambda 
    \le 
    \frac{(1-p^*)}{p^* \, (1 + \dd/\tau)}
    .
\end{align}
Note that it is required in \RefThm{thm:HO-Q-R-truncated} that $\lambda \ge 1$, and together with~\eqref{eq:condition:lambda} this restricts us to the case when $p^*$ is sufficiently small, that is,
\begin{align}\label{eq:condition:pstar}
    1
    \le
    \frac{(1-p^*)}{p^* \, (1 + \dd/\tau)}
    ,
    \qquad
    \text{or equivalently,}
    \qquad
    p^* \le \frac1{2 + \dd /\tau}
    .
\end{align}
Making use of~\eqref{eq:L1} and the first expression in~\eqref{eq:hu}, and then~\eqref{eq:L3} and~\eqref{eq:L-multiplier}, the computational cost~\eqref{eq:cost-MDFEM} is bounded like
\begin{align}
    \notag
    \cost(Q_\epsilon)
    &\lesssim
    \sum_{\setu \in \setU(\epsilon)} n_\setu \, h_\setu^{-\dd} \,\pounds_\setu
    \le
    \sum_{\setu \in \setU(\epsilon)} k_\setu \, h_\setu^{-\dd} \,\pounds_\setu
    =
    \xi \, \lambda \sum_{\setu \in \setU(\epsilon)} a_\setu \, k_\setu^{-\lambda}
    =
    \frac{\lambda}{1 + \lambda\,\dd/\tau} \, \frac{\epsilon}{2} \, \xi
    \\
    \notag
    &=
    \frac{\lambda}{1 + \lambda\,\dd/\tau} \, \frac{\epsilon}{2}
  \left(
  \frac{2}{\epsilon} \,
  (1 + \lambda\,\dd/\tau) \, B^{-\lambda}
  \, K_\epsilon
  \right)^{1 + 1/\lambda + \dd/\tau}
    \\
    \notag
    &=
    \lambda \, \left((1 + \lambda\,\dd/\tau)^{-1} \, \frac{\epsilon}{2}\right)^{-1/\lambda-\dd/\tau} \,
  \left(
  B^{-\lambda}
  \, K_\epsilon
  \right)^{1 + 1/\lambda + \dd/\tau}
  .
\end{align}
Using the stated conditions we know that $K_\epsilon \le K < \infty$.
Hence we can write
\begin{align*}
  \cost(Q_\epsilon)
  &\lesssim
  \epsilon^{-1/\lambda-\dd/\tau}
  .
\end{align*}
To optimize the speed of convergence we have to pick $\lambda$ and $\tau$ as large as possible while satisfying~\eqref{eq:condition:lambda}.
Hence we choose
\begin{align*}
  \lambda
  &=
  \frac{(1-p^*)}{p^* \, (1 + \dd/\tau)}
  ,
\end{align*}
and we have $\lambda \ge 1$ due to~\eqref{eq:condition:pstar}.
By choosing the interlacing order of the interlaced polynomial lattice rule to be $\alpha = \lfloor\lambda\rfloor + 1$ we satisfy $\alpha \ge 2$ and $\lambda \in [1,\alpha)$, and we also satisfy~\eqref{eq:condition:alpha1}, since
\begin{align*}
  \alpha_1
  = 
  \frac{\alpha}{2} + \frac1{4}
  =
  \frac{\lfloor\lambda\rfloor}{2} + \frac{3}{4}
  <
  1 + \lambda \, (1+\dd/\tau)
  ,
\end{align*} 
for any $\lambda \ge 0$.
This finishes the proof.
\end{proof}

\begin{remark}\label{rem:randomized-FE+Q-error}
Since \RefProp{prop:FE+Q-error} is making use of interlaced polynomial lattice rules to achieve higher order convergence, i.e., $\lambda \ge 1$, we end up with the condition $0 < p^* \le (2 + \dd/\tau)^{-1}$ there.
Following the proof we see that if we make use of first order cubature rules with error bounds for
$\tfrac1{2} \le \lambda < 1$, we obtain the condition $p^* \le (3/2 + \dd/(2\tau))^{-1}$ by using~\eqref{eq:condition:lambda}.
In particular, we can use transformed \emph{randomly digitally shifted polynomial lattice rules} from \RefRem{rem:randomized-QMC} as cubature rules in the MDFEM algorithm.
A full description and analysis of using such randomized cubature rules in the context of the MDFEM algorithm for the uniform case is given in~\cite{NN21}.
We will write $\E^{\Delta(\epsilon)}[| \II(G(u)) - Q_\epsilon(G(u)) |^2]$ to denote the \emph{total mean square error} of the randomized MDFEM algorithm using randomly digitally shifted polynomial lattice rules.
The expected value is taken over a set $\Delta(\epsilon) := \{\bsDelta^{(\setu)}\}_{\setu \in \setU(\epsilon)}$ of independent random digital shifts, one for each subproblem~$\setu$, and where each such shift $\bsDelta^{(\setu)}$ is uniformly distributed over $[0,1)^{|\setu|}$.
\end{remark}

\subsection{Main result}

We are now able to analyze the complexity of the MDFEM algorithm.
The analysis is under the conditions of \RefPropRange{prop:bound-total-error}{prop:FE+Q-error} extended to the randomized setting using \RefRemTwo{rem:randomized-QMC}{rem:randomized-FE+Q-error}.

\begin{theorem}\label{thm:main-theorem}    
    Assume that for a given $\alpha \in \N$ there is a sequence $\{b_j\}_{j\ge1} \in \ell^{p^*}(\N)$ for some $p^* \in (0,1)$, with $0 < b_j \le 1$ and
  	\begin{align*}
	  \kappa
	  =
	  \left\|\sum_{j \ge 1} \frac{|\phi_j|}{b_j} \right\|_{L^\infty(D)}
	  <
	  \frac{\ln(2)}{\alpha}
	  .
	\end{align*}
    Let $G \in V^*$ be a bounded linear functional.
    Assume that $D$, $a$, $f$ and $G$ have sufficient spatial regularity such that the application of $G$ to the FE approximations of the truncated solutions converge like $O(h_\setu^\tau)$ as in~\eqref{eq:Lprho-FE-error-G} and assume that the cost of solving the linear systems for the FEMs are $O(h_\setu^{-\dd})$ as in~\eqref{eq:cost-of-solve}.
    
    For a given requested error $\epsilon > 0$ take the active set $\setU(\epsilon) = \setU(\epsilon,p^*)$ as in~\eqref{eq:active-set}. Set
    \begin{align*}
	  \lambda
	  &=
	  \frac{(1-p^*)}{p^* \, (1 + \dd/\tau)}
	  \ge
	  \frac12
	  &\text{and}
	  &&
	  \alpha
	  &=
	  \lfloor\lambda\rfloor + 1
	  \ge
	  1
	  .
	\end{align*}
    Further, take the FE mesh diameters $h_\setu$ and the number of cubature points $n_\setu$ as the solution to the optimization problem~\eqref{eq:minimize} \textup{(}see \RefProp{prop:FE+Q-error} and \RefRem{rem:randomized-FE+Q-error}, with solutions~\eqref{eq:hu} and~\eqref{eq:ku} and with $n_\setu = 2^{\log_2(\lfloor k_\setu \rfloor)}$\textup{)}, such that the convergence of the cubature rules over the unit cube can be bounded by $O(n_\setu^{-\lambda})$, see \RefThmTwo{thm:error-bound-IPLR}{thm:HO-Q-R-truncated} and \RefRem{rem:randomized-QMC}.
    
    Then, for the MDFEM algorithm $Q_\epsilon$, given in~\eqref{eq:MDFEM-algorithm}, the following hold.
    \begin{itemize}
    \item If $0 < p^* \le (2 + \dd/\tau)^{-1}$, i.e., $\lambda \ge 1$ and $\alpha \ge 2$, then, using transformed interlaced polynomial lattice rules with interlacing factor $\alpha$, we have
      \begin{align*}
        |\II(G(u)) - Q_\epsilon(G(u))|
        &\lesssim
        \epsilon^{1 - o(1)}
        .
      \end{align*}
    \item If $(2 + \dd/\tau)^{-1} < p^* \le (3/2 + \dd/(2\tau))^{-1}$, i.e., $\frac12 \le \lambda < 1$ and $\alpha = 1$, then, using transformed randomly digitally shifted polynomial lattice rules, we have
      \begin{align*}
        \sqrt{\E^{\Delta(\epsilon)} \left[|\II(G(u)) - Q_\epsilon(G(u))|^2\right]}
        &\lesssim
        \epsilon^{1 - o(1)}
        .
      \end{align*}	 	
	\end{itemize} 
    In both cases the computational cost is bounded as
    \begin{align}\label{eq:a-MDFEM}
      \cost(Q_\epsilon)
      &\lesssim
      \epsilon^{-1/\lambda - \dd/\tau}
      =
      \epsilon^{-a_\MDFEM}
      &
      \text{with} \quad
      a_\MDFEM
      &:=
      \frac{1}{\lambda} + \frac{\dd}{\tau}
      =
      \frac{p^* + \dd/\tau}{1 - p^*}
      .
    \end{align}	
\end{theorem}

\begin{proof}
We first show the statement for the case $0 < p^* \le (2 + \dd/\tau)^{-1}$.
It follows from Propositions~\ref{prop:bound-total-error},~\ref{prop:active-set-truncation-error} and~\ref{prop:FE+Q-error} that the error of the MDFEM is then bounded by
\begin{align*}
    \left| \II(G(u)) - Q_\epsilon(G(u)) \right|
    &\lesssim
    \max\left\{1, \max_{\setu \in \setU(\epsilon)} \left(\frac{\ln(n_\setu)}{|\setu|}\right)^{\alpha_1 |\setu|}\right\}
    \;
    \epsilon
\end{align*}
with computational cost
\begin{align*}
    \cost(Q_\epsilon)
    &\lesssim
    \epsilon^{-1/\lambda - \dd/\tau}
    =
    \epsilon^{(p^* + \dd/\tau)/(1 - p^*)}
    .
\end{align*}
Using~\eqref{eq:U-max-size} from \RefProp{prop:cardinality-active-set} and $\ln(n_\setu) \lesssim \ln(\epsilon^{-1/\lambda}) \lesssim \ln(\epsilon^{-1})$, see~\eqref{eq:ku} in combination with~\eqref{eq:L-multiplier}, we can use the same argument as in~\cite[Theorem~1]{NN21}, see also~\cite[Lemma~1]{PW11}, and we have
\begin{align*}
    \max\left\{1, \max_{\setu \in \setU(\epsilon)} \left(\frac{\ln(n_\setu)}{|\setu|}\right)^{\alpha_1 |\setu|}\right\}
    = 
    \epsilon^{-\delta(\epsilon)}
    ,
\end{align*}
where $\delta(\epsilon)= O\left(\ln(\ln(\ln(\epsilon^{-1})))/\ln(\ln(\epsilon^{-1}))\right) = o(1)$ as $\epsilon \to 0$.
Hence, we can write 
\begin{align*}
  \left|\II(G(u)) -Q_\epsilon(G(u))\right|
  &\lesssim
  \epsilon^{1 - o(1)}
  .
\end{align*}
For the case $(2+\dd/\tau)^{-1} < p^* \le (3/2 + \dd/(2\tau))^{-1}$ the statement follows using similar arguments, a full proof is provided in~\cite{NN21}. 
\end{proof}

Finally we would like to compare the complexity of the MDFEM with the quasi-Monte Carlo finite element method (QMCFEM)~\cite{HS19} and with the multilevel quasi-Monte Carlo finite element method (MLQMCFEM)~\cite{HS19ML}, both in the setting of using product weights and a wavelet expansion for the lognormal field.

The QMCFEM truncates the parameter vector $\bsy$ to some dimension~$s$ and then approximates the $s$-truncated PDE for different samples $\bsy^{(i)} \in \R^s$, $i=0,\ldots,N-1$, obtained by a QMC method.
Because $s$ might be arbitrarily large, the QMCFEM requires QMC rules with convergence independent of the dimension of the integrand.
Such QMC rules over the Euclidean space $\R^s$ with the Gaussian distribution were developed in~\cite{KSWWa10} by mapping \emph{randomly shifted lattice rules} over the unit cube $[0,1]^s$ to $\R^s$ by the inverse of the normal cumulative distribution.
However, since this mapping might damage the smoothness of the integrand, the convergence rate of these QMC rules was limited to first order (with respect to the number of QMC points~$N$).
Particularly, the QMCFEM was analysed in~\cite[Section~10]{HS19} under the same conditions as those of \RefThm{thm:main-theorem} (marked with the subscripts ``PROD'' for product weights and ``wav'' for wavelet expansion, and making use of ``Gaussian weight functions'' for the QMC rules) to achieve a root-mean-square error bound of the form
\begin{align}\label{eq:RMSE-QMCFEM}
  \sqrt{\E^{\Delta} \left[|\II(G(u)) - Q^\QMCFEM(G(u))|^2\right]}
  &\lesssim
  s^{-\left( \frac{2}{p^*} - \frac12 \right) + \delta}
  +
  N^{-\left(\frac12 \min\left\{ \frac32, \frac1{p^*} \right\} + \frac14 \right) + \delta}
  +
  h^\tau
  ,
\end{align}
where $\delta > 0$ is a parameter that might be chosen arbitrarily small (but then increases the hidden constant towards infinity), $N$ is the number of cubature points, $h$ is the FE mesh diameter and $s$ is the truncation dimension.
(Note that \cite[Section~10]{HS19} considers the specific choice of a Gaussian random field with Mat{\'e}rn covariance for which we (optimistically) set $p^* = d / \nu$, with $\nu$ the smoothness parameter of the Mat{\'e}rn covariance function. Note additionally that in \cite{BCM2018} it is proven that $p^*$ can be chosen arbitrarily close to $d/\nu$ in the case of a wavelet-type expansion.)
The cost for the QMCFEM can be estimated by $\cost(Q^\QMCFEM) \lesssim N(s + h^{-\dd})$.
Since we are interested in the asymptotic best possible rates, we will ignore the hidden constants and focus on the rate.
This means we formally set $\delta=0$ in~\eqref{eq:RMSE-QMCFEM} as a proxy of being arbitrarily close to the optimal rate and thereby obtaining a slightly optimistic bound on the work.
By balancing all three contributions to be $\epsilon/3$ we find
\begin{align*}
  \cost(Q^{\mathrm{QMCFEM}})
  &\lesssim
  \epsilon^{-a_\QMCFEM}
  ,
  \\
  \text{with\ \ }
  a_\QMCFEM
  &:=
  \begin{cases}
    1 + \max\{ 2p^*/(4-p^*), \dd/\tau \}, & \text{if } 0 < p^* \le 2/3 , \\
    4p^*/(2+p^*) + \max\{ 2p^*/(4-p^*), \dd/\tau \}, & \text{if } p^* \ge 2/3 .
  \end{cases}
\end{align*}
A smaller exponent means less work.
For both exponents, $a_\QMCFEM$ as just defined, and $a_\MDFEM$ as given in~\eqref{eq:a-MDFEM}, we can recognize two parts: the first part is the convergence order of the QMC cubature w.r.t.\ its cost and the second part is the convergence order of the linear functional applied to the FE approximation w.r.t.\ the cost of the FE approximation.

We now look at the case $0 < p^* \le (2 + \dd/\tau)^{-1} < 1/2 < 2/3$ when the MDFEM can employ a higher-order QMC cubature rule.
We then have $\lambda \ge 1$ and hence $a_\MDFEM$ has the factor $1/\lambda$ for the QMC cubature part.
For the QMCFEM the convergence of the QMC cubature is limited to $1$ and hence, for the case $0 < p^* \le 2/3$ we obtain
\begin{align*}
  a_\QMCFEM
  &=
  1 + \max\left\{\frac{2 p^*}{4-p^*}, \frac{\dd}{\tau}\right\}
  =
  \begin{cases}
  1 + \dd/\tau , & \text{if } 0 < p^* \le (4 \dd/\tau) (2 + \dd/\tau)^{-1} , \\
  1 + 2 p^* / (4-p^*) , & \text{if } p^* \ge (4 \dd/\tau) (2 + \dd/\tau)^{-1} . \\
  \end{cases}
\end{align*}
It is easy to see that if $\dd/\tau \ge 1/4$ then we have $0 < p^* \le (2 + \dd/\tau)^{-1} \le (4 \dd/\tau) (2 + \dd/\tau)^{-1}$.
Hence $a_\MDFEM = 1/\lambda + \dd/\tau \le a_\QMCFEM = 1 + \dd/\tau$ and the MDFEM and the QMCFEM perform the same when $\lambda = 1$ and the MDFEM will perform better when $\lambda > 1$.
On the other hand, if $\dd/\tau < 1/4$ then there are two cases to be considered.
When $0 < p^* \le (4 \dd/\tau) (2 + \dd/\tau)^{-1} < (2 + \dd/\tau)^{-1}$ we have exactly the same conclusion as for $\dd/\tau \ge 1/4$.
Lastly, when $0 < (4 \dd/\tau) (2 + \dd/\tau)^{-1} <  p^* \le (2 + \dd/\tau)^{-1}$ the QMCFEM is in the unfortunate case that the term for the FEM convergence is larger than $\dd/\tau$ and so it will always lose.

Hence, for any $p^*$ such that $0 < p^* \le (2 + \dd/\tau)^{-1}$ the MDFEM performs better or similar compared to the QMCFEM, with the possibility of higher-order convergence for the MDFEM.

\label{pp:compare-MLQMCFEM}As a second comparison we look at the multilevel variant of the QMCFEM which is the MLQMCFEM algorithm given in~\cite{HS19ML}.
Since the QMC rules for the MLQMCFEM are the same kind as those for the QMCFEM, they can achieve at most order~$1$ convergence and this requires $0 < p^* < 2/3$.
The error bounds in \cite[Theorem~6.3 and~6.5]{HS19ML} take more complicated forms in which $\bar{\xi}$ there is the convergence of the QMC rule on each level, comparable with our $\lambda$, and $\eta$ there is our $\ddelta$.
They are written in terms of the dominating cost, being either $\epsilon^{-1/\lambda}$, with $\lambda < 1$ due to the QMC rules there, or $\epsilon^{-\dd/\tau}$ with extra log factors depending on the situation.
So also in this case it is clear that in the case $0 < p^* \le (2 + \dd/\tau)^{-1} < 1/2 < 2/3$ if we take $\lambda \ge 1$ for the MDFEM and we can take $\dd/\tau \le 1/\lambda$ then the MDFEM performs better or similar compared to the MLQMCFEM, with the possibility of higher-order convergence for the MDFEM.

\section*{Acknowledgements}

The authors would like to thank the three anonymous referees for very careful and detailed commenting on the manuscript.
The authors would also like to thank Frances Y.\ Kuo and Kien V.\ Nguyen for discussion on parts of the manuscript.
We gratefully acknowledge the financial support from the Research Foundation Flanders (FWO) under grant G091920N.

\appendix

\section{Appendix}

We collect some results here which would otherwise disturb the flow of the paper.

\subsection{Proof of \RefLem{lem:anchored-decomposition}}\label{app:proof-anchored-props}

\begin{proof}[Proof of \RefLem{lem:anchored-decomposition}]
  The first property, \eqref{eq:prop1}, is a well known property of the anchored decomposition and we omit the proof.
  We now prove the second property, \eqref{eq:prop2}, which is, 
  \begin{align*}
		(\partial^{\bsomega_\setu}_{\bsy_\setu} F_\setu)(\bsy_\setu) 
		&=
		(\partial^{\bsomega_\setu}_{\bsy_\setu} \FF(\cdot_\setu))(\bsy_\setu)
		\quad \text{when } \forall j \in \setu : \omega_j \ge 1
		,
  \end{align*}
  i.e., for $\bsomega_\setu \in \N^{|\setu|}$.
  First we move the derivative operator inside the explicit form for $F_\setu$, see~\eqref{eq:anchored-decomposition-F},
  \begin{align*}
    (\partial^{\bsomega_\setu}_{\bsy_\setu} F_\setu)(\bsy_\setu)
    &=
    \sum_{\setv \subseteq \setu} (-1)^{|\setu|-|\setv|} \, (\partial^{\bsomega_\setu}_{\bsy_\setu} \FF(\cdot_\setv))(\bsy_\setv)
  \end{align*}
  and realize that $\FF(\cdot_\setv)$ is a function which is constant in all variables not in $\setv$, and as such if we have an $\omega_j \ge 1$ for which $j \notin \setv$ then the derivative $(\partial^{\bsomega_\setu}_{\bsy_\setu} \FF(\cdot_\setv))$ is zero.
  Since $\bsomega_\setu \in \N^{|\setu|}$ we have $\omega_j \ge 1$ for all $j \in \setu$ and thus the only remaining term for $\setv \subseteq \setu$ is the one for which~$\setv = \setu$.
  
  To show the third property, \eqref{eq:prop3}, which is
  \begin{align*}
	  (\partial^{\bsomega_\setu}_{\bsy_\setu} F_\setu)(\bsy^*_\setu)
	  &=
	  0
	  \quad \text{when } \exists j \in \setu : y^*_j = 0 \text{ and } \omega_j = 0
	  ,
  \end{align*}
  we first note that here $\bsomega_\setu \in \N_0^{|\setu|}$, i.e., $\omega_j$ is allowed to be zero for $j \in \setu$.
  Therefore, suppose there is a $j \in \setu$ such that $\omega_j = 0$, then we are actually not taking the partial derivative w.r.t.\ the $j$th variable, and by the definition of the partial derivative this means we keep the $j$th variable fixed while we take the derivatives w.r.t.\ the variables in $\supp(\bsomega_\setu)$.
  Then, in the case that $\omega_j = 0$ we can do the evaluation at $y^*_j$ before taking the other partial derivatives, i.e., $(\partial^{\bsomega_\setu}_{\bsy_\setu} F_\setu)(\bsy^*_\setu) = (\partial^{\bsomega_\setu}_{\bsy_\setu} (F_\setu|_{y_j=y^*_j}))(\bsy^*_{\setu \setminus \{j\}})$.
  Since $j \in \setu$ we can use the first property~\eqref{eq:prop1} to deduce that $F_\setu|_{y_j=y^*_j} = 0$ when~$y^*_j = 0$.
\end{proof}

\subsection{Derivation of anchored Sobolev kernels and Taylor representation}\label{app:kernel}

We provide the derivation of the reproducing kernel of a univariate anchored Sobolev space w.r.t.\ a positive weight function $\rho$ for anchored functions.
In particular this shows how to obtain the kernel given in~\eqref{eq:kernel-a-R} for the anchored Gaussian Sobolev space for anchored functions.
For $\alpha \in \N$ the space consists of functions that have absolutely continuous derivatives up to order $\alpha-1$ for any bounded interval and have square integrable derivative of order $\alpha$ w.r.t.\ the weight function $\rho$ and are anchored at~$0$.
The domain is the support of~$\rho$.
These spaces are meant to be used for functions $F_\setu$ obtained by the anchored decomposition~\eqref{eq:anchored-decomposition-F}.

In accordance with \cite[Example~4.4]{KSWW10b} and \cite{GMR14}, the only constant function in our univariate anchored space is the zero function, and we can use the tensor product of this univariate kernel to create a multivariate kernel which represents multivariate anchored functions.
A reproducing kernel $K$ for a reproducing kernel Hilbert space $\HH(K)$ has the property that $K(x,y)$ as a function of $y$ is a function in $\HH(K)$ for any $x$ in the domain.
To satisfy this property, the kernel for the anchored Sobolev space on the unit cube for anchored functions in \cite[Example~4.2]{KSWW10b} and \cite[Section~5.2]{DG14}, should be amended to have the sum over the derivatives from $1$ to $\alpha-1$ where the ``$0$~otherwise'' appears.
Then it agrees with the kernel given here.

Our proof uses similar arguments as in~\cite[Section~1.2]{Wah90} where the kernel of the \emph{anchored Sobolev space} over the unit cube is derived, but we modify the techniques for the case when the inner product contains a positive weight function $\rho$ in the $L^2$ inner product of the $\alpha$th derivatives, and when the functions are anchored.

\begin{proposition}\label{prop:kernel}
  The reproducing kernel of the anchored Sobolev space for anchored functions $H_{\alpha,0,\rho}(\R)$, with $\rho(t) = \rho(-t)$, e.g., the anchored Gaussian Sobolev space, with inner product
  \begin{align*}
  \langle F, G \rangle_{H_{\alpha,0,\rho}(\R)}
  &:=
  \sum_{\tau=1}^{\alpha-1} F^{(\tau)}(0) \, G^{(\tau)}(0)
  +
  \int_\R F^{(\alpha)}(y) \, G^{(\alpha)}(y) \, \rho(y) \rd y
  ,
  \end{align*}
  is given by
  \begin{align*}
  K_{\alpha,0,\rho}(x,y)
  &:=
  \sum_{\tau=1}^{\alpha-1} \frac{x^\tau}{\tau!} \frac{y^\tau}{\tau!}
  +
  \mathds{1}\{xy > 0\}
  \int_0^{\min\{|x|,|y|\}} 
    \frac{(|x|-t)^{\alpha-1}}{(\alpha-1)!}
    \frac{(|y|-t)^{\alpha-1}}{(\alpha-1)!}
    \frac1{\rho(t)} 
    \rd t
  ,
  \end{align*}
  where $\mathds{1}\{X\}$ is the indicator function on~$X$.
\end{proposition}
\begin{proof}
By Taylor's theorem, and using $F(0) = 0$, we have
\begin{align*}
    F(y)
    =
    \sum_{\tau=1}^{\alpha-1} \frac{y^\tau}{\tau!} \, F^{(\tau)}(0)
    +
    \int_0^y \frac{(y-t)^{\alpha-1}}{(\alpha-1)!} \, F^{(\alpha)}(t) \rd t
    =
    \sum_{\tau=1}^{\alpha-1} \frac{y^\tau}{\tau!} \, F^{(\tau)}(0)
    -
    \int_{y}^0 \frac{(y-t)^{\alpha-1}}{(\alpha-1)!} \, F^{(\alpha)}(t) \rd t
    ,
\end{align*}
where we have written both forms as this helps to interpret how our kernel will operate in connection to the integral over $\R$ in the inner product.
By the reproducing property we require
\begin{align*}
    &F(y)
    =
    \langle F, K_{\alpha,0,\rho}(\cdot,y) \rangle_{H_{\alpha,0,\rho}(\R)}
    =
    \sum_{\tau=1}^{\alpha-1} F^{(\tau)}(0) \, K_{\alpha,0,\rho}^{(\tau)}(0,y)
    +
    \int_\R F^{(\alpha)}(t) \, K_{\alpha,0,\rho}^{(\alpha)}(t,y) \, \rho(t) \rd t
    \\
    &=
    \sum_{\tau=1}^{\alpha-1} F^{(\tau)}(0) \, K_{\alpha,0,\rho}^{(\tau)}(0,y)
    +
    \int_0^{+\infty} F^{(\alpha)}(t) \, K_{\alpha,0,\rho}^{(\alpha)}(t,y) \, \rho(t) \rd t
    +
    \int_{-\infty}^0 F^{(\alpha)}(t) \, K_{\alpha,0,\rho}^{(\alpha)}(t,y) \, \rho(t) \rd t
    ,
\end{align*}
where the derivatives of $K_{\alpha,0,\rho}$ are taken with respect to the first variable.

Comparing the two representations of $F$ leads us to choose the kernel $K_{\alpha,0,\rho}$ such that
\begin{align}\label{eq:devK1}
    K_{\alpha,0,\rho}^{(\tau)}(0,y) 
    &=
    \frac{y^\tau}{\tau!}
    \qquad \text{ for } \tau = 1, \ldots, \alpha-1
\end{align}
and
\begin{align}\label{eq:devK2}
    K_{\alpha,0,\rho}^{(\alpha)}(t,y) 
    &=
    \begin{cases*}
        \displaystyle
        \frac{(y-t)^{\alpha-1}}{(\alpha-1)!} \, \frac1{\rho(t)} \, \mathds{1}\{t \in [0,y]\}
        , & if $y > 0$, \\[1mm]
        \displaystyle
        \frac{-(y-t)^{\alpha-1}}{(\alpha-1)!} \, \frac1{\rho(t)} \, \mathds{1}\{t \in [y,0]\}
        , & if $y < 0$, \\[1mm]
        0
        , & if $y=0$.
    \end{cases*}
\end{align}
Since $K_{\alpha,0,\rho}(\cdot,y)$ itself needs to be a function from the space $H_{\alpha,0,\rho}(\R)$ we require that $K_{\alpha,0,\rho}(0,y)=0$ for any $y$ and its Taylor expansion with respect to the first variable is given by
\begin{align}\label{eq:TaylorK}
    K_{\alpha,0,\rho}(x,y)
    =
    \sum_{\tau=1}^{\alpha-1} \frac{x^\tau}{\tau!} \, K_{\alpha,0,\rho}^{(\tau)} (0,y) 
    +
    \int_0^x \frac{(x-t)^{\alpha-1}}{(\alpha-1)!} \, K_{\alpha,0,\rho}^{(\alpha)}(t,y) \rd t
    .
\end{align}
Therefore, inserting \eqref{eq:devK1} and~\eqref{eq:devK2} into~\eqref{eq:TaylorK} we find for $y \ge 0$ that
\begin{align*}
    K_{\alpha,0,\rho}(x,y)
    &=
    \sum_{\tau=1}^{\alpha-1} \frac{x^\tau}{\tau!} \frac{y^\tau}{\tau!}
    +
    \int_0^x \frac{(x-t)^{\alpha-1}}{(\alpha-1)!} \frac{(y-t)^{\alpha-1}}{(\alpha-1)!}
    \, \frac1{\rho(t)} 
    \, \mathds{1}\{t \in [0,y]\} \rd t
    ,
\end{align*}
which can be written as 
\begin{align*}
    K_{\alpha,0,\rho}(x,y)
    &=
    \begin{cases*}
        \displaystyle
        \sum_{\tau=1}^{\alpha-1} \frac{x^\tau}{\tau!} \frac{y^\tau}{\tau!}
        +
        \int_0^{\min\{x,y\}} \frac{(x-t)^{\alpha-1}}{(\alpha-1)!} \frac{(y-t)^{\alpha-1}}{(\alpha-1)!}
        \frac1{\rho(t)}  \rd t
        ,
        & if $x,y \ge 0$,
        \\[1mm]
        \displaystyle
        \sum_{\tau=1}^{\alpha-1} \frac{x^\tau}{\tau!} \frac{y^\tau}{\tau!}
        ,
        & if $x \le 0$ and $y \ge 0$.
    \end{cases*} 
\end{align*}
Similarly, for $y \le 0$ we have 
\begin{align*}
    K_{\alpha,0,\rho}(x,y)
    =
    \begin{cases*}
        \displaystyle
        \sum_{\tau=1}^{\alpha-1} \frac{x^\tau}{\tau!} \frac{y^\tau}{\tau!}
        +
        \int_{\max\{x,y\}}^0 \frac{(x-t)^{\alpha-1}}{(\alpha-1)!} \frac{(y-t)^{\alpha-1}}{(\alpha-1)!}
        \frac1{\rho(t)}  \rd t
        ,
        & if $x,y \le 0$,
        \\[1mm]
        \displaystyle
        \sum_{\tau=1}^{\alpha-1} \frac{x^\tau}{\tau!} \frac{y^\tau}{\tau!}
        ,
        & if $x \ge 0$ and $y \le 0$.
    \end{cases*} 
\end{align*}
Hence, under the assumption that $\rho(t) = \rho(-t)$, the explicit formula for $K_{\alpha,0,\rho}$ can be written as
\begin{align*}
    K_{\alpha,0,\rho}(x,y)
    &=
    \sum_{\tau=1}^{\alpha-1} \frac{x^\tau}{\tau!} \frac{y^\tau}{\tau!}
    +
    \mathds{1}\{xy > 0\}
    \int_0^{\min\{|x|,|y|\}} 
    \frac{(|x|-t)^{\alpha-1}}{(\alpha-1)!}
    \frac{(|y|-t)^{\alpha-1}}{(\alpha-1)!}
    \frac1{\rho(t)} 
    \rd t
    .
    \qedhere
\end{align*}
\end{proof}

We have stated \RefProp{prop:kernel} in a form which is relevant to the paper, but in fact the proof is more general.
As stated, we do not require $\rho$ to be the Gaussian density. The condition $\rho(t) = \rho(-t)$ in the proposition is only stated to obtain an easy final expression and could be removed.
The resulting kernel is valid for other positive weight functions $\rho$ and we assume the domain of the integrals are then truncated to the support of~$\rho$.
E.g., $\rho$ could be the uniform density on the unit cube and we then recover the anchored Sobolev space for anchored functions on $[0,1]$.
If we want to drop the requirement that $F(0) = 0$ then we add the constant one to the kernel and include the term for $\tau = 0$ in the inner product.
Such a kernel for $\rho \equiv 1$ over $\R$ is given in~\cite[Section~11.5.1]{NW10}.
The statements are also easily generalizable for an arbitrary anchor point, similar as the kernels in \cite[Example~4.2]{KSWW10b} and~\cite[Section~5.2]{DG14}.

We take the multivariate space to be the tensor product of the univariate spaces and therefore the kernel is obtained as the product, see, e.g., \cite[Example~4.4]{KSWW10b}.
For anchored functions in such a multivariate anchored Sobolev space (not necessarily just the Gaussian case as we use in this paper) we have the following representation as a Taylor series with integral remainder term.

\begin{lemma}\label{lem:Taylor}
  For $\setu \subset \N$ with $1 \le |\setu| < \infty$, assume $F_\setu \in H_{\alpha,0,\rho,|\setu|}(\Omega^{|\setu|})$, with $\Omega$ the support of the positive weight function~$\rho$, is obtained from an anchored decomposition~\eqref{eq:anchored-decomposition-F} of a function~$\FF$.
  Then we have the representation
\begin{align}\label{eq:Taylor-Fu}
    F_\setu(\bsy_\setu)
    &=
    \sum_{\substack{\bsnu_\setu \in \{1:\alpha\}^{|\setu|} \\ \setv := \{ j : \nu_j = \alpha\} }}
    \; \left[ \prod_{j \in \setu \setminus \setv} \frac{y_j^{\nu_j}}{\nu_j!} \right] \;
    \int_{\bszero_\setv}^{\bsy_\setv}
    F_\setu^{(\bsnu_\setu)}(\bst_\setv) 
    \prod_{j \in \setv} \frac{(y_j-t_j)^{\alpha-1}}{(\alpha-1)!}
    \rd \bst_\setv
    ,
\end{align}
where $F_\setu^{(\bsnu_\setu)}(\bst_\setv)$ means to evaluate the function $F_\setu^{(\bsnu_\setu)}$ in the point $\bst_\setv \in \R^\N_\setv$, i.e., setting all other arguments to zero, or, if viewing $F_\setu^{(\bsnu_\setu)}$ as a $|\setu|$-variate function then $F_\setu^{(\bsnu_\setu)}(\bst_\setv) = F_\setu^{(\bsnu_\setu)}(\bst_\setv, \bszero_{\setu \setminus \setv})$.
For $\bsomega_\setu \in \{0:\alpha\}^{|\setu|}$ we have
\begin{align}\label{eq:Taylor-DFu}
    F_\setu^{(\bsomega_\setu)}(\bsy_\setu)
    &=
    \sum_{\substack{\bsomega_\setu \le \bsnu_\setu \in \{1:\alpha\}^{|\setu|} \\ \setw := \{ j : \omega_j = \alpha \} \\ \setv := \{ j \notin \setw : \nu_j = \alpha \} \\ \setz := \setu \setminus (\setv \cup \setw)}}
    \; \left[ \prod_{j \in \setz} \frac{y_j^{\nu_j - \omega_j}}{(\nu_j - \omega_j)!} \right] \;
    \int_{\bszero_\setv}^{\bsy_\setv}
    F_\setu^{(\bsnu_\setu)}(\bst_\setv, \bsy_\setw)
    \prod_{\substack{j \in \setv}} \frac{(y_j-t_j)^{\alpha-\omega_j-1}}{(\alpha-\omega_j-1)!}
    \rd \bst_\setv
    .
\end{align}
The functions $F_\setu^{(\bsnu_\setu)}$ under the integrals could be replaced with $\FF^{(\bsnu_\setu)}$ in both~\eqref{eq:Taylor-Fu} and~\eqref{eq:Taylor-DFu}.
\end{lemma}
\begin{proof}
  Functions in $H_{\alpha,0,\rho,|\setu|}(\Omega^{|\setu|})$ satisfy all requirements to use the Taylor theorem with integral remainder up to order $\alpha$ in each direction successively, i.e., they have absolutely continuous derivatives up to order $\alpha-1$ for any bounded interval.
  Hence
\begin{align*}
    F_\setu(\bsy_\setu)
    &=
    \sum_{\substack{\bsnu_\setu \in \{0:\alpha\}^{|\setu|} \\ \setv := \{ j : \nu_j = \alpha\} }}
    \; \left[ \prod_{j \in \setu \setminus \setv} \frac{y_j^{\nu_j}}{\nu_j!} \right] \;
    \int_{\bszero_\setv}^{\bsy_\setv}
    F_\setu^{(\bsnu_\setu)}(\bst_\setv) 
    \prod_{j \in \setv} \frac{(y_j-t_j)^{\alpha-1}}{(\alpha-1)!}
    \rd \bst_\setv
    ,
\end{align*}
  where we remark that we wrote $\bsnu_\setu \in \{0:\alpha\}^{|\setu|}$, i.e., $\nu_j$ is allowed to be zero also for $j \in \setu$.
  We can now make use of \RefLem{lem:anchored-decomposition}.
  In particular, property~\eqref{eq:prop3} implies that any term for which there is at least one $j \in \setu$ for which $\nu_j = 0$ will vanish.
  This proves~\eqref{eq:Taylor-Fu}.
    
  To obtain the expression for $F_\setu^{(\bsomega_\setu)}(\bsy_\setu)$ for $\bsomega_\setu \in \{0:\alpha\}^{|\setu|}$ we work in a similar way by applying the Taylor theorem with integral remainder term to $F_\setu^{(\bsomega_\setu)}$ for each $\omega_j < \alpha$.
  For $\omega_j = \alpha$ we cannot apply the Taylor theorem anymore, so we need to take care that in our expression we can just recover the $\alpha$th derivatives, i.e., we should not integrate those components for which $\omega_j = \alpha$.
  Therefore we introduce the set $\setw := \{ j \in \setu : \omega_j = \alpha \}$.
  As in the expression for $F_\setu$ we gather the indices for which we need the integral in the set $\setv$ which is now modified to exclude those indices in the set~$\setw$.
  This proves~\eqref{eq:Taylor-DFu}.

  By property~\eqref{eq:prop2} we also know that $F_\setu^{(\bsnu_\setu)}(\bsy_\setu) = \FF^{(\bsnu_\setu)}(\bsy_\setu)$ for $\bsnu_\setu \in \{1:\alpha\}^{|\setu|}$ and any $\bsy_\setu$ so also for $\bsy_\setu = (\bst_\setv, \bszero_{\setu \setminus \setv}) \in \R^\N_\setu$ or $\bsy_\setu = (\bst_\setv, \bsy_\setw, \bszero_\setz) \in \R^\N_\setu$ as in~\eqref{eq:Taylor-Fu} and~\eqref{eq:Taylor-DFu} respectively.  
\end{proof}

\subsection{Proof of \RefProp{prop:embedding-F-rho-T}: norm embedding after mapping}\label{app:proof-embedding}

We deferred the proof of \RefProp{prop:embedding-F-rho-T} due to its length.
We first show two lemmas that we need for the proof.
The following result is taken from~\cite[Lemma~3]{DILP18}.
 
\begin{lemma}\label{lem:derivative-F-rho}
 	For $F : \R^s \to \R$ having mixed partial derivatives for all $\bstau \in \{0:\alpha\}^s$ we have for any $\bstau \in \{0:\alpha\}^s$ that
 	\begin{align*}
 	(F \rho)^{(\bstau)}(\bsy)
 	=
 	\rho(\bsy)
 	\sum_{\bsomega \le \bstau}
 	(-1)^{|\bstau-\bsomega|} \,
 	c(\bstau,\bsomega)
 	\,
 	H_{\bstau-\bsomega}(\bsy)
 	\,
 	F^{(\bsomega)}(\bsy)
 	,
 	\end{align*}
 	where the sum is over all $\bsomega \in \N_0^s$ for which $0 \le \omega_j \le \tau_j$ for all $j \in \{1,2,\ldots,s\}$, $c(\bstau,\bsomega) := {\bstau \choose \bstau-\bsomega}\sqrt{(\bstau-\bsomega)!} := \prod_{j =1}^s {\tau_j \choose \tau_j -\omega_j} \sqrt{(\tau_j -\omega_j)!}$ and $H_{\bstau}(\bsy) := \prod_{j=1}^s H_{\tau_j}(y_j)$, with $H_\tau$, for $\tau \in \N_0$, the $\tau$th normalized probabilistic Hermite polynomial given by
 	\begin{align}\label{eq:Hermite-pol}
 	H_\tau(y)
 	:= 
 	\sqrt{\tau!} \sum_{k=0}^{\lfloor \tau/2 \rfloor} \frac{(-1)^k}{k! \, (\tau-2k)!} \frac{y^{\tau-2k}}{2^k}
 	.
 	\end{align}
\end{lemma}

\begin{lemma}\label{lem:polynomial-hermite-bound}
	For $\alpha \in \N$, $\tau \in \{0:\alpha\}$ and $\eta \in [0,2\alpha-1]$ we have
	\begin{align}\label{eq:def:C-diamond}
	C_{\diamond,\alpha}
	&:=
	\int_\R
	|H_\tau(y)|^2 \,
	|y|^{\eta} \,
	\rho(y)
	\rd y
	\le
		\alpha! \,
		(1 + \alpha/2) \,
		\frac{2^{2\alpha}}{\sqrt{2\pi}} \,
		\Gamma(2\alpha) \,
		I_0(1/2)
	,
	\end{align}
	with $I_0$ the modified Bessel function of the first kind of order~$0$.
\end{lemma}
\begin{proof}
  Using~\eqref{eq:Hermite-pol}, the Cauchy--Schwarz inequality and $\int_\R |y|^\eta \rho(y) \rd y = 2^{\eta/2} \, \Gamma((\eta+1)/2) / \sqrt{\pi}$ for any $\eta \ge 0$, we have
	\begin{align*}
	&\int_\R |H_\tau(y)|^2 \, |y|^\eta \, \rho(y) \rd y
	=
	\int_\R
	\tau! 
	\left[
	\sum_{k=0}^{\lfloor \tau/2 \rfloor} \frac{(-1)^k}{k! \, (\tau-2k)!}\frac{y^{\tau-2k}}{2^k}
	\right]^2
	|y|^\eta
	\, \rho(y) \rd y
	\\
	&\qquad\le
	\tau! \,
	\left( \sum_{k=0}^{\lfloor \tau/2 \rfloor} (-1)^{2k} \right)
	\left(
	\sum_{k=0}^{\lfloor \tau/2 \rfloor} \frac{2^{-2k}}{(k! \, (\tau-2k)!)^2} 
	\int_\R
	|y|^{2(\tau-2k)+\eta}
	\, \rho(y) \rd y
	\right)
	\\
	&\qquad=
	\tau! \,
	(1 + \lfloor \tau/2 \rfloor)
	\sum_{k=0}^{\lfloor \tau/2 \rfloor} \frac{2^{-2k}}{(k! \, (\tau-2k)!)^2} 
	\frac1{\sqrt{\pi}} \,
	2^{\tau - 2k + \eta/2} \,
	\Gamma(\tau - 2k + (\eta+1)/2)
	\\
	&\qquad\le
	\alpha! \,
	(1 + \alpha/2) \,
	\frac{2^{2\alpha}}{\sqrt{2\pi}} \,
	\Gamma(2\alpha)
	\sum_{k=0}^{\infty} \frac{2^{-4k}}{(k!)^2}
	\\
	&\qquad=
	\alpha! \,
	(1 + \alpha/2) \,
	\frac{2^{2\alpha}}{\sqrt{2\pi}} \,
	\Gamma(2\alpha) \,
    I_0(1/2)
	,
	\end{align*}
where in the second inequality we use $\tau \le \alpha$, $\eta \le 2\alpha-1$,  $\tau - 2k \ge 0$, $1/2 \le \tau - 2k + (\eta+1)/2 \le 2\alpha$ for any $k \in \{0:\lfloor \tau/2 \rfloor\}$ and $\Gamma(1/2) \le \Gamma(2\alpha)$ for $\alpha \ge 1.43258$. For $\alpha = 1$ the result follows by directly comparing the stated upper bound with the closed form solution of the integral for $\tau=0$ (given above) and $\tau=1$ (which is the expression above with $\eta$ replaced by~$\eta+2$).
\end{proof} 

We are now ready to give the proof of \RefProp{prop:embedding-F-rho-T} which states that for $F \in H_{\alpha,0,\rho,s}(\R^s)$ with $\alpha \in \N$ and $T \ge 1/(2\sqrt{2})$, the function $(F \rho) \circ \bsT : [0,1]^s \to \R^s$ belongs to $H_{\alpha,s}([0,1]^s)$ and
  \begin{align}\label{eq:embedding:repeat}
    \|(F \rho) \circ \bsT\|_{H_{\alpha,s}([0,1]^s)}
    &\le 
    C_{1,\alpha}^s \, T^{(\alpha-1/2)s} \, \|F\|_{H_{\alpha,0,\rho,s}(\R^s)}
    ,
  \end{align}	
with $C_{1,\alpha}$ defined in~\eqref{eq:def:C1}.
The proof starts along the lines of \cite[proof of Lemma~4]{DILP18} but we need some extra work to arrive at the norm in the space $H_{\alpha,0,\rho,s}(\R^s)$ for which we will make use of \RefLem{lem:polynomial-hermite-bound}.
We remind the reader that $F(y_1,\ldots,y_s)$ is actually some relabelling of $F_\setu(\bsy_\setu)$ with $|\setu| = s$ coming from an anchored decomposition and that $H_{\alpha,0,\rho,s}(\R^s)$ is the anchored Gaussian Sobolev space specifically for such functions having the properties listed in \RefLem{lem:anchored-decomposition}.

\begin{proof}[Proof of \RefProp{prop:embedding-F-rho-T}]
Using the inner product~\eqref{eq:ip-u-unitcube-s}, the Cauchy--Schwarz inequality and the chain rule $((F\rho)\circ T)^{(\bstau)}(\bsy) = (2T)^{|\bstau|} (F \rho)^{(\bstau)}(\bsy)$ we have
\begin{align*}
    &\|(F \rho) \circ \bsT\|_{H_{\alpha,s}([0,1]^s)}^2
    =
    \sum_{\substack{\bstau \in \{0:\alpha\}^s \\ \setv := \{ j : \tau_j = \alpha \} }}
    \int_{[0,1]^{|\setv|}} 
    \left[
    \int_{[0,1]^{s-|\setv|}}
    ((F \rho) \circ \bsT)^{(\bstau)}(\bsy) \rd \bsy_{-\setv}
    \right]^2
    \rd \bsy_\setv
    \notag
    \\
    &\qquad\le
    \sum_{\substack{\bstau \in \{0:\alpha\}^s \\ \setv := \{ j : \tau_j = \alpha \} }}
    \int_{[0,1]^{|\setv|}} 
    \left(\int_{[0,1]^{s-|\setv|}} 1^2 \rd\bsy_{-\setv}\right)
    \left( \int_{[0,1]^{s-|\setv|}} \left[ ((F\rho)\circ \bsT)^{(\bstau)}(\bsy) \right]^2 \rd\bsy_{-\setv} \right)
    \rd\bsy_\setv
    \notag 
    \\
    &\qquad=
    \sum_{\bstau \in \{0:\alpha\}^s}
    \int_{[0,1]^s} 
    \left[
    ((F \rho) \circ \bsT)^{(\bstau)}(\bsy)
    \right]^2
     \rd \bsy
    \notag
    \\
    &\qquad=
    \frac1{(2T)^s}
    \sum_{\bstau \in \{0:\alpha\}^s}
    (2T)^{2|\bstau|}
    \int_{[-T,T]^s}
    \left[
    (F \rho)^{(\bstau)}(\bsy)
    \right]^2
    \rd \bsy 
    .
\end{align*}
Applying \RefLem{lem:derivative-F-rho} to the last inequality we then write
\begin{align*}
    &\|(F \rho) \circ \bsT\|_{H_{\alpha,s}([0,1]^s)}^2
    \notag
    \\
    &\qquad\le
    \frac1{(2T)^s}
    \sum_{\bstau \in \{0:\alpha\}^s}
    (2T)^{2|\bstau|}
    \int_{[-T,T]^s}
    \left[
    \rho(\bsy)
    \sum_{\bsomega \le \bstau}
    c(\bstau,\bsomega)
    \,
    |H_{\bstau-\bsomega}(\bsy)|
    \,
    | F^{(\bsomega)}(\bsy)|
    \right]^2
    \rd \bsy
    \notag
    \\
    &\qquad=
    \frac1{(2T)^s}
    \sum_{\bstau \in \{0:\alpha\}^s}
    (2T)^{2|\bstau|}
    \sum_{\bsomega \le \bstau} c(\bstau,\bsomega)
    \sum_{\bsomega' \le \bstau} c(\bstau,\bsomega')
    \notag
    \\
    &\qquad \qquad \qquad \qquad 
    \times
    \int_{[-T,T]^s}
    |H_{\bstau-\bsomega}(\bsy)|
    \;
    |H_{\bstau-\bsomega'}(\bsy)|
    \;
    |F^{(\bsomega)}(\bsy)|
    \;
    |F^{(\bsomega')}(\bsy)| 
    \;
    \rho^2(\bsy)
    \rd \bsy
    .
\end{align*}
Applying the Cauchy--Schwarz inequality to the last integral leads to
\begin{align}\label{eq:ap7}
    &\|(F \rho) \circ \bsT\|_{H_{\alpha,s}([0,1]^s)}^2
    \notag
    \\
    &\qquad\le
    \frac1{(2T)^s}
    \sum_{\bstau \in \{0:\alpha\}^s}
    (2T)^{2|\bstau|}
    \sum_{\bsomega \le \bstau} c(\bstau,\bsomega)
    \underbrace{
    \left[\int_{[-T,T]^s}
    |H_{\bstau-\bsomega}(\bsy)|^2
    \,
    |F^{(\bsomega)}(\bsy)|^2
    \,
    \rho^2(\bsy)
    \rd \bsy
    \right]^{1/2}}_{=: \calY(\bstau,\bsomega)}
    \notag
    \\
    &
    \qquad \qquad \qquad \qquad \qquad 
    \times
    \sum_{\bsomega' \le \bstau} 
    c(\bstau,\bsomega')
    \underbrace{\left[\int_{[-T,T]^s}
    |H_{\bstau-\bsomega'}(\bsy)|^2
    \,
    |F^{(\bsomega')}(\bsy)|^2
    \,
    \rho^2(\bsy)
    \rd \bsy
    \right]^{1/2}}_{=\calY(\bstau,\bsomega')}
    .
\end{align}
We will show below that for any $\bstau \in \{0:\alpha\}^s$ and $\bsomega \in \N_0^s$ such that $\bsomega \le \bstau$ we have the uniform bound
\begin{align}\label{eq:ap4}
    \calY(\bstau,\bsomega)
    :=
    \left[\int_{[-T,T]^s}
    |H_{\bstau-\bsomega}(\bsy)|^2
    \,
    |F^{(\bsomega)}(\bsy)|^2
    \,
    \rho^2(\bsy)
    \rd \bsy
    \right]^{1/2}
    \le
    C_{*,\alpha}^{s/2} \,
    \|F\|_{H_{\alpha,0,\rho,s}(\R^s)}
    ,
\end{align}
where $C_{*,\alpha} :=\alpha\, C_{\diamond,\alpha}$ with $C_{\diamond,\alpha}$ defined in~\eqref{eq:def:C-diamond}.
Inserting~\eqref{eq:ap4} into~\eqref{eq:ap7} then leads to
\begin{align}\label{eq:ap8}
    \|(F \rho) \circ \bsT\|_{H_{\alpha,s}([0,1]^s)}^2
	&\le
    C_{*,\alpha}^s \,
    \|F\|_{H_{\alpha,0,\rho,s}(\R^s)}^2 \,
    \frac1{(2T)^s}
    \sum_{\bstau \in \{0:\alpha\}^s}
    (2T)^{2|\bstau|}
    \left[
    \sum_{\bsomega \le \bstau} c(\bstau,\bsomega)
    \right]^2
    .
\end{align}
Moreover, we have, for $\bstau \in \{0:\alpha\}^s$ and $\bsomega \le \bstau$,
\begin{align*}
    \left[
    \sum_{\bsomega \le \bstau} c(\bstau,\bsomega)
    \right]^2
    &=
    \prod_{j=1}^s
    \left[
    \sum_{\omega_j=0}^{\tau_j}
    {\tau_j \choose \tau_j-\omega_j}
    \sqrt{(\tau_j-\omega_j)!}
    \right]^2
    \\
    &\le
    \prod_{j=1}^s
    \left[ \sqrt{\alpha!} \sum_{\omega_j =0} ^{\tau_j}
    {\tau_j \choose \tau_j-\omega} \right]^2
    =
    (\alpha!)^s
    \prod_{j=1}^s 
    2^{2\tau_j}
    .
\end{align*}
Inserting this into~\eqref{eq:ap8} implies
\begin{align*}
    \|(F \rho) \circ \bsT\|_{H_{\alpha,s}([0,1]^s)}^2
    &\le
    C_{*,\alpha}^s \,
    \|F\|_{H_{\alpha,0,\rho,s}(\R^s)}^2 \,
    (\alpha!)^s \,
    \frac1{(2T)^s}
    \sum_{\bstau \in \{0:\alpha\}^s}
    (2T)^{2|\bstau|} \, 2^{2|\bstau|}
    \\
    &=
    C_{*,\alpha}^s \,
    \|F\|_{H_{\alpha,0,\rho,s}(\R^s)}^2 \,
     (\alpha!)^s \,
    \frac1{(2T)^s}
    \left(
    \sum_{\tau=0}^\alpha
    (4T)^{2\tau}
    \right)^s
    \\
    &=
    C_{*,\alpha}^s \,
    \|F\|_{H_{\alpha,0,\rho,s}(\R^s)}^2 \,
    (\alpha!)^s \,
    \frac1{(2T)^s}
    \left(
      \frac{(4T)^{2\alpha+2}-1}{(4T)^2-1}
    \right)^s
    \\
    &\le
    C_{*,\alpha}^s \, (\alpha!)^s \,
    16^{\alpha s} \,
    T^{(2\alpha-1)s} \,
    \|F\|_{H_{\alpha,0,\rho,s}(\R^s)}^2
    \\
    &
    =
    C_{1,\alpha}^{2s} \, T^{(2\alpha-1)s} \, \|F\|_{H_{\alpha,0,\rho,s}(\R^s)}^2
    ,
\end{align*}
where we used $(4T)^{2\alpha+2} / ((4T)^2-1) \le 2 \, (4 T)^{2\alpha}$ for $T \ge 1/(2\sqrt{2}) > 1/4$, and where
\begin{align*}
    C_{1,\alpha}
    :=
    \left(
      C_{*,\alpha} \,
      \alpha! \,
      16^\alpha
    \right)^{1/2}
    &=
    \left(
      \alpha \,
      \alpha! \,
      (1+\alpha/2) \,
      \frac{2^{2\alpha}}{\sqrt{2\pi}} \,
      \Gamma(2\alpha) \,
      I_0(1/4) \,
      \alpha! \,
      2^{4\alpha}
    \right)^{1/2}
    \\
    &=
    \alpha! \,
    2^{3\alpha}
    \left(
      \alpha\,
      (1+\alpha/2) \,
      \frac{1}{\sqrt{2\pi}} \,
      \Gamma(2\alpha) \,
      I_0(1/4)
    \right)^{1/2}
    .
\end{align*}
We have now arrived at the claim of our statemement~\eqref{eq:embedding:repeat}.

To complete the proof we still need to show~\eqref{eq:ap4}.
We are going to use the Taylor representation from \RefLem{lem:Taylor} for the derivatives of $F = F_\setu$, where in our current exposition $\{1:s\}$ is a relabeling of~$\setu$.
Thus, using~\eqref{eq:Taylor-DFu}, and with the understanding that $F = F_\setu$, $\bsomega = \bsomega_\setu \in \{0:\alpha\}^{|\setu|}$ and $\bsy = \bsy_\setu \in \R^{|\setu|}$, we have, with some slight abuse of notation,
\begin{align*}
    F^{(\bsomega)}(\bsy)
    &=
    F_\setu^{(\bsomega_\setu)}(\bsy_\setu)
    \\
    &=
    \sum_{\substack{\bsomega_\setu \le \bsnu_\setu \in \{1:\alpha\}^{|\setu|} \\ \setw := \{ j : \omega_j = \alpha \} \\ \setv := \{ j \notin \setw : \nu_j = \alpha \} \\ \setz := \setu \setminus (\setv \cup \setw) }}
    \; \left[ \prod_{j \in \setz} \frac{y_j^{\nu_j - \omega_j}}{(\nu_j - \omega_j)!} \right] \;
    \int_{\bszero_\setv}^{\bsy_\setv}
    F_\setu^{(\bsnu_\setu)}(\bst_\setv, \bsy_\setw)
    \prod_{\substack{j \in \setv}} \frac{(y_j-t_j)^{\alpha-\omega_j-1}}{(\alpha-\omega_j-1)!}
    \rd \bst_\setv
    \\
    &=
    \sum_{\substack{\bsomega \le \bsnu \in \{1:\alpha\}^s \\ \setw := \{ j : \omega_j = \alpha \} \\ \setv := \{ j \notin \setw : \nu_j = \alpha \} \\ \setz := \{1:s\} \setminus (\setv \cup \setw) }}
    \; \left[ \prod_{j \in \setz} \frac{y_j^{\nu_j - \omega_j}}{(\nu_j - \omega_j)!} \right] \;
    \int_{\bszero_\setv}^{\bsy_\setv}
    F^{(\bsnu)}(\bst_\setv, \bsy_\setw)
    \prod_{\substack{j \in \setv}} \frac{(y_j-t_j)^{\alpha-\omega_j-1}}{(\alpha-\omega_j-1)!}
    \rd \bst_\setv
    .
\end{align*}
Note that inside the integral we evaluate the function $F = F_\setu$ at $(\bst_\setv, \bsy_\setw) = (\bst_\setv, \bsy_\setw, \bszero_\setz) \in \R^{|\setu|}$ with $\setz = \setu \setminus (\setv \cup \setw)$ and $\setv, \setw, \setz$ are pairwise disjoint with $\setu = \setv \cup \setw \cup \setz$ since the above definitions are equivalent to
\begin{align*}
  \setw &= \{ j \in \setu : \nu_j = \alpha \text{ and } \omega_j = \alpha \} , \\
  \setv &= \{ j \in \setu : \nu_j = \alpha \text{ and } \omega_j \ne \alpha \} , \\
  \setz &= \{ j \in \setu : \nu_j \ne \alpha \} ,
\end{align*}
where we used that $\omega_j \le \nu_j \le \alpha$ for the set~$\setw$.

In what follows, for any $y \in \R$ we write $[0,y]^* = [0,y]$ if $y \ge 0$ and $[0,y]^* = [y,0]$ if $y < 0$.
For $\bsy \in \R^s$ we write $[\bszero, \bsy]^* = [0, y_1]^* \times \cdots \times [0,y_s]^*$.
So, squaring the above expression for $F^{(\bsomega)}$ and applying the Cauchy--Schwarz inequality twice we obtain
\begin{align}
  \notag
  &|F^{(\bsomega)}(\bsy)|^2
  \\
  \notag
  &\le
    \alpha^s \hspace*{-4mm} %% FIXME manual spacing correction
    \sum_{\substack{\bsomega \le \bsnu \in \{1:\alpha\}^s \\ \setw := \{ j : \omega_j = \alpha \} \\ \setv := \{ j \notin \setw : \nu_j = \alpha \} \\ \setz := \{1:s\} \setminus (\setv \cup \setw) }}
    \hspace*{-1.4mm} %% FIXME manual spacing correction
    \left[ \prod_{j \in \setz} \frac{y_j^{2(\nu_j - \omega_j)}}{((\nu_j - \omega_j)!)^2} \right]
    %\;
    \int_{[\bszero_\setv,\bsy_\setv]^*} \hspace*{-1.4mm} %% FIXME manual spacing correction
    |F^{(\bsnu)}(\bst_\setv, \bsy_\setw)|^2
    \rd \bst_\setv
    \;
    \prod_{\substack{j \in \setv}}
    \int_{[0,y_j]^*} \hspace*{-1.4mm} %% FIXME manual spacing correction
    \frac{(y_j-t_j)^{2(\alpha-\omega_j-1)}}{((\alpha-\omega_j-1)!)^2}
    \rd t_j
  \\
  \notag
  &=
    \alpha^s \hspace*{-4mm} %% FIXME manual spacing correction
    \sum_{\substack{\bsomega \le \bsnu \in \{1:\alpha\}^s \\ \setw := \{ j : \omega_j = \alpha \} \\ \setv := \{ j \notin \setw : \nu_j = \alpha \} \\ \setz := \{1:s\} \setminus (\setv \cup \setw) }}
    \hspace*{-1.4mm} %% FIXME manual spacing correction
    \left[ \prod_{j \in \setz} \frac{|y_j|^{2\nu_j - 2\omega_j}}{((\nu_j - \omega_j)!)^2} \right]
    %\;
    \int_{[\bszero_\setv,\bsy_\setv]^*} \hspace*{-1.4mm} %% FIXME manual spacing correction
    |F^{(\bsnu)}(\bst_\setv, \bsy_\setw)|^2
    \rd \bst_\setv
    \;
    \prod_{\substack{j \in \setv}}
    \frac{|y_j|^{2\alpha-2\omega_j-1}}{(2\alpha-2\omega_j-1)((\alpha-\omega_j-1)!)^2}
  \\
  \notag
  &\le
    \alpha^s \hspace*{-4mm} %% FIXME manual spacing correction
    \sum_{\substack{\bsomega \le \bsnu \in \{1:\alpha\}^s \\ \setw := \{ j : \omega_j = \alpha \} \\ \setv := \{ j \notin \setw : \nu_j = \alpha \} \\ \setz := \{1:s\} \setminus (\setv \cup \setw) }}
    \hspace*{-1.4mm} %% FIXME manual spacing correction
    \left[ \prod_{j \in \setz} |y_j|^{2\nu_j - 2\omega_j} \right]
    %\;
    \int_{[\bszero_\setv,\bsy_\setv]^*} \hspace*{-1.4mm} %% FIXME manual spacing correction
    |F^{(\bsnu)}(\bst_\setv, \bsy_\setw)|^2
    \rd \bst_\setv
    \;
    \prod_{\substack{j \in \setv}}
    |y_j|^{2\alpha-2\omega_j-1}
  \\
  \label{eq:apB}
  &\le
    \alpha^s \hspace*{-4mm} %% FIXME manual spacing correction
    \sum_{\substack{\bsomega \le \bsnu \in \{1:\alpha\}^s \\ \setw := \{ j : \omega_j = \alpha \} \\ \setv := \{ j \notin \setw : \nu_j = \alpha \} \\ \setz := \{1:s\} \setminus (\setv \cup \setw) }}
    \underbrace{
    \left[ \prod_{j \in \setz \cup \setv} |y_j|^{2\alpha - 1} \right]
    %\;
    \int_{[\bszero_\setv,\bsy_\setv]^*} \hspace*{-1.4mm} %% FIXME manual spacing correction
    |F^{(\bsnu)}(\bst_\setv, \bsy_\setw)|^2
    \rd \bst_\setv
    }_{=: B(\bsnu,\bsomega,\bsy)}
  .
\end{align}
Note that in the second and third step we used that for $j \in \setv$ we know that $\omega_j \le \alpha - 1$, and in the last step we used that for $j \in \setz$ we have $\nu_j \le \alpha - 1$.
We want to use this bound for $|F^{(\bsomega)}(\bsy)|^2$ in $\calY^2(\bstau,\bsomega)$, cf.~\eqref{eq:ap4}, which is multiplying with $|H_{\bstau-\bsomega}(\bsy)|^2 \, \rho^2(\bsy)$ and taking the integral over $[-T,T]^s$.
We can move the integral inside of the sum over $\bsnu$ in the above expression to obtain
\begin{align}\label{eq:apY}
    \calY^2(\bstau,\bsomega)
    &\le
    \alpha^s
    \sum_{\substack{\bsomega \le \bsnu \in \{1:\alpha\}^s \\ \setw := \{ j : \omega_j = \alpha \} \\ \setv := \{ j \notin \setw : \nu_j = \alpha \} \\ \setz := \{1:s\} \setminus (\setv \cup \setw) }}
    \underbrace{
    \int_{[-T,T]^s} 
    |H_{\bstau-\bsomega}(\bsy)|^2 \, B(\bsnu,\bsomega,\bsy)
    \, \rho^2(\bsy)
    \rd\bsy
    }_{=: A(\bsnu,\bsomega,\bstau)}
    ,
\end{align}
with $B(\bsnu,\bsomega,\bsy)$ defined in~\eqref{eq:apB}.
Then
\begin{align*}
  &A(\bsnu,\bsomega,\bstau)
  \\
  &\quad=
  \int_{[-T,T]^s}
    |H_{\bstau-\bsomega}(\bsy)|^2 \,
    \left[ \prod_{j \in \setz \cup \setv} |y_j|^{2\alpha - 1} \right]
    \;
    \int_{[\bszero_\setv,\bsy_\setv]^*} %\hspace*{-1.4mm} %% FIXME manual spacing correction
    |F^{(\bsnu)}(\bst_\setv, \bsy_\setw)|^2
    \rd \bst_\setv
    \; \rho^2(\bsy)
    \rd\bsy
  \\
  &\quad\le
  \int_{[-T,T]^s}
    |H_{\bstau-\bsomega}(\bsy)|^2 \,
    \left[ \prod_{j \in \setz \cup \setv} |y_j|^{2\alpha - 1} \right]
    \;
    \frac{1}{\rho_\setv(\bsy_\setv)}
    \int_{[\bszero_\setv,\bsy_\setv]^*} %\hspace*{-1.4mm} %% FIXME manual spacing correction
    |F^{(\bsnu)}(\bst_\setv, \bsy_\setw)|^2
    \, \rho_\setv(\bst_\setv)
    \rd \bst_\setv
    \; \rho^2(\bsy)
    \rd\bsy
  \\
  &\quad\le
  \int_{\R^s}
    |H_{\bstau-\bsomega}(\bsy)|^2 \,
    \left[ \prod_{j \in \setz \cup \setv} |y_j|^{2\alpha - 1} \right]
    \;
    \frac{1}{\rho_\setv(\bsy_\setv)}
    \int_{\R^{|\setv|}}
    |F^{(\bsnu)}(\bst_\setv, \bsy_\setw)|^2
    \, \rho_\setv(\bst_\setv)
    \rd \bst_\setv
    \; \rho^2(\bsy)
    \rd\bsy
  .
\end{align*}
We now split the integral over $\bsy$ into a product of three integrals over the pairwise disjoint sets $\setv \cup \setw \cup \setz = \{1:s\}$.
We obtain
\begin{align}\notag
    &A(\bsnu,\bsomega,\bstau)
    \\
    &\qquad=
    \left[
    \prod_{j \in \setz}
    \int_\R
    |H_{\tau_j-\omega_j}(y_j)|^2 \,
    |y_j|^{2\alpha-1}
    \smash[b]{\underbrace{\rho^2(y_j)}_{\le \rho(y_j)}}
    \rd y_j
    \right]
    \times
    \left[
    \prod_{j \in \setv}
    \int_\R
    |H_{\tau_j-\omega_j}(y_j)|^2 \,
    |y_j|^{2\alpha-1}
    \, \rho(y_j)
    \rd y_j
    \right]
    \notag
    \\
    &
    \qquad\qquad \times \left[
    \int_{\R^{|\setw|}}
    \underbrace{|H_{\bstau_\setw-\bsomega_\setw}(\bsy_\setw)|^2 \, \rho_\setw(\bsy_\setw)}_{\le 1}
    \;
    \int_{\R^{|\setv|}}
    |F^{(\bsnu)}(\bst_\setv, \bsy_\setw)|^2
    \, \rho_\setv(\bst_\setv)
    \rd \bst_\setv
    \;
    \rho_\setw(\bsy_\setw)
    \rd \bsy_\setw
    \right]
    \notag
    \\
    &\qquad\le
    C_{\diamond,\alpha}^{|\setz|} \, C_{\diamond,\alpha}^{|\setv|} \,
    \int_{\R^{|\setv\cup\setw|}}
    |F^{(\bsnu)}(\bsy_{\setv\cup\setw})|^2
    \, \rho_{\setv\cup\setw}(\bsy_{\setv\cup\setw})
    \rd \bsy_{\setv\cup\setw}
    \label{eq:apA}
    ,
\end{align}
where we used \RefLem{lem:polynomial-hermite-bound} and the constant $C_{\diamond,\alpha}$ defined there in~\eqref{eq:def:C-diamond}, and $H_{\bstau}(\bsy)\sqrt{\rho(\bsy)} \le 1$ for any $\bstau \in \N_0^s$ and any $\bsy \in \R^s$, see~\cite[Lemma~1]{DILP18}. 
Inserting~\eqref{eq:apA} into~\eqref{eq:apY} leads to 
\begin{align*}
 \calY^2(\bstau,\bsomega)
 &\le
     \alpha^s
    \sum_{\substack{\bsomega \le \bsnu \in \{1:\alpha\}^s \\ \setw := \{ j : \omega_j = \alpha \} \\ \setv := \{ j \notin \setw : \nu_j = \alpha \} \\ \setz := \{1:s\} \setminus (\setv \cup \setw) }}
    C_{\diamond,\alpha}^{|\setz|} \, C_{\diamond,\alpha}^{|\setv|} \,
    \int_{\R^{|\setv\cup\setw|}}
    |F^{(\bsnu)}(\bsy_{\setv\cup\setw})|^2
    \, \rho_{\setv\cup\setw}(\bsy_{\setv\cup\setw})
    \rd \bsy_{\setv\cup\setw}
 \\
 &\le
     \alpha^s C_{\diamond,\alpha}^s
    \sum_{\substack{\bsomega \le \bsnu \in \{1:\alpha\}^s \\ \setv := \{ j : \nu_j = \alpha \} }}
    \int_{\R^{|\setv|}}
    |F^{(\bsnu)}(\bsy_\setv)|^2
    \, \rho_\setv(\bsy_\setv)
    \rd \bsy_\setv
 \\
 &\le
     \alpha^s C_{\diamond,\alpha}^s
    \sum_{\substack{\bsnu \in \{1:\alpha\}^s \\ \setv := \{ j : \nu_j = \alpha \} }}
    \int_{\R^{|\setv|}}
    |F^{(\bsnu)}(\bsy_\setv)|^2
    \, \rho_\setv(\bsy_\setv)
    \rd \bsy_\setv
 =
 \alpha^s C_{\diamond,\alpha}^s \, \|F\|^2_{H_{\alpha,0,\rho,s}(\R^s)}
 ,
\end{align*}
where in the last line we obtain the norm based on the inner product~\eqref{eq:ip-a-Rs} of the space $H_{\alpha,0,\rho,s}(\R^s)$.
This shows~\eqref{eq:ap4}. The proof is now complete.
\end{proof}

\subsection{Interlaced polynomial lattice rules for $H_{\alpha,s}([0,1]^s)$}\label{app:IPLR}

The aim of this section is to show that interlaced polynomial lattice rules can achieve the almost optimal order of convergence for integration in the space $H_{\alpha,s}([0,1]^s)$ defined in \RefSec{sec:uSob-unitcube}.
Interlaced polynomial lattice rules were also used in the setting of PDEs with random diffusion coefficient, but for the uniform case, i.e., with integrals directly expressible over the unit cube, in \cite{DKLNS14}.
Here we map our integrals over the full space into the unit cube by the strategy described in \RefSec{sec:HOQMC}. But the unanchored Sobolev space here is different from the unanchored Sobolev space considered in \cite{DKLNS14}.
We adjust the analysis of \cite{DKLNS14} and \cite{God15} to show that the fast component-by-component construction algorithm as in \cite{DKLNS14} can also construct optimal interlaced polynomial lattice rules for our space $H_{\alpha,s}([0,1]^s)$.
We remind the reader that the inner product of our space was already given in~\eqref{eq:ip-u-unitcube} and~\eqref{eq:ip-u-unitcube-s}.
As explained in the introduction, we do not consider weighted function spaces, since the MDM already takes care to limit the number of dimensions for each subproblem.

\emph{Interlaced polynomial lattice rules} are a modification of \emph{polynomial lattice rules} to achieve higher-order convergence for integration over the unit cube in classes of \emph{Walsh spaces} and \emph{weighted unanchored Sobolev spaces}, see, e.g., \cite{DKLNS14,God15}.
The aim is to approximate multivariate integrals over the $s$-dimensional unit cube
\begin{align*}
  I_{[0,1]^s} (F)
  &:= 
  \int_{[0,1]^s} F(\bsy) \rd \bsy
\end{align*}
by a quasi-Monte Carlo rule of the form
\begin{align*}
  Q_{[0,1]^s,P_n}(F)
  &:=
  \frac1{n} \sum_{i=0}^{n-1} F(\bsy^{(i)})
  ,
\end{align*}
where $P_n:=\{\bsy^{(i)}\}_{i=0}^{n-1}$ is the cubature point set.
The \emph{worst-case error} of the QMC rule $Q_{[0,1]^s,P_n}$ in the normed space $H_{\alpha,s}([0,1]^s)$ is defined by
\begin{align*}
  e_\wor(P_n; H_{\alpha,s}([0,1]^s))
  :=
  \sup_{\|F\|_{H_{\alpha,s}([0,1]^s)} \le 1}
  \left|
  I_{[0,1]^s}(F)
  -
  Q_{[0,1]^s,P_n}(F) 
  \right|
  .
\end{align*}
Hence, for any $F \in H_{\alpha,s}([0,1]^s)$
\begin{align*}
	\left| I_{[0,1]^s}(F) - Q_{[0,1]^s,P_n}(F) \right|
	&\le 
	e_\wor(P_n; H_{\alpha,s}([0,1]^s))
	\,
	\|F\|_{H_{\alpha,s}([0,1]^s)}
	.
\end{align*}

We need to introduce some necessary definitions.
For simplicity we restrict ourselves to polynomial lattice rules over the finite field $\Z_2$.
Let $\Z_2[\x]$ denote the set of all polynomials over $\Z_2$ and $\Z_2[\x^{-1}]$ denote the set of all formal Laurent series over $\Z_2$.
For any $m \in \N$ let us define a mapping $\vartheta_m: \Z_2[\x^{-1}] \to [0,1)$ by
\begin{align*}
  \vartheta_m\Big( \sum_{i=\ell}^\infty w_i \, \x^{-i} \big)
  &:=
  \sum_{i=\max(1,\ell)}^m w_i \, 2^{-i}
  .
\end{align*} 
In the following we will identify any integer $k \in \{ 0,\ldots,2^m-1 \}$, having binary expansion $k = \kappa_0 + \kappa_1 2 + \cdots +\kappa_{m-1} 2^{m-1}$, with the polynomial $k(\x) = \kappa_0 + \kappa_1 \x + \cdots + \kappa_{m-1} \x^{m-1} \in \Z_2[\x]$ and vice versa.

\begin{definition}[polynomial lattice rule]
	For $m, s \in \N$ let $p \in \Z_2[\x]$ be an irreducible polynomial such that $\deg(p)=m$ and let $\bsq=(q_1, \ldots, q_s) \in \mathscr{G}_m^s$ with 
	\begin{align*}
	  \mathscr{G}_m := \{ q \in \Z_2[\x]: \deg(q) < m \}
	  .
	\end{align*}
	A \emph{polynomial lattice point set} $P_{p,m, s}(\bsq)$ is a set of $n = 2^m$ points $\bsy^{(0)}, \ldots, \bsy^{(2^m-1)} \in [0,1)^s$ where
	\begin{align*}
	  \bsy^{(k)}
	  = 
	  \left(
	  \vartheta_m\Big(\frac{k(\x) q_1(\x)}{p(\x)}\Big), \ldots, \vartheta_m\Big(\frac{k(\x) q_s(\x)}{p(\x)}\Big)
	  \right)
	  .
	\end{align*}
	A QMC rule using this point set is called a \emph{polynomial lattice rule} with generating vector $\bsq$ and modulus $p$.
\end{definition}

The convergence for a polynomial lattice rule is typically that of a normal QMC rule, i.e., $O(n^{-1+\delta})$, $\delta > 0$, under appropriate conditions and modulo $\log$-factors.
By making use of \emph{interlacing} we can obtain higher order convergence $O(n^{-\alpha+\delta})$, $\delta > 0$, for $\alpha > 1$, again, under appropriate conditions and modulo $\log$-factors.
Interlacing is the process of combining $\alpha \in \N$, $\alpha \ge 2$, base-$b$ elements into one. I.e., we can interlace a tuple of $\alpha$ base-$2$ numbers in the interval $[0,1)$ and combine them into one base-$2$ number in the interval $[0,1)$.
This is the aim of the \emph{digit interlacing function} which will take a polynomial lattice point set in $\alpha s$ dimensions and interlace the points with a factor $\alpha$ to obtain an interlaced polynomial lattice point set in $s$ dimensions.

\begin{definition}[interlaced polynomial lattice rule]\label{def:IPLR}
	Define the \emph{digit interlacing function} $\mathscr{D}_\alpha: [0,1)^\alpha \to [0,1)$ with interlacing factor $\alpha \in \N$ by 
	\begin{align*}
	\mathscr{D}_\alpha(y_1, \ldots, y_\alpha)
	:=
	\sum_{i=1}^{\infty} \sum_{j=1}^\alpha \frac{\xi_{i,j}}{2^{\alpha(i-1)+j}}
	,
	\end{align*}
	where $y_j = \xi_{1,j} 2^{-1}+ \xi_{2,j} 2^{-2}+ \cdots$ for $j = 1,\ldots, \alpha$ and, in case the number of arguments is a multiple of $\alpha$, define $\mathscr{D}_\alpha : [0,1)^{\alpha s} \to [0,1)^s$ by
	\begin{align*}
	\mathscr{D}_\alpha(y_1, \ldots, y_{\alpha s}):= (\mathscr{D}_\alpha(y_1, \ldots, y_\alpha), \ldots, \mathscr{D}_\alpha(y_{(s-1)\alpha+1}, \ldots, y_{s\alpha}))
	.
	\end{align*} 
	For $m, s\in \N$ 
	let $p \in \Z_2[\x]$ be an irreducible polynomial such that $\deg(p)=m$ and let $\bsq= (q_1,\ldots,q_{\alpha s})\in \mathscr{G}_m^{\alpha s}$. An \emph{interlaced polynomial lattice point set} (of order $\alpha$) $\mathscr{D}_\alpha(P_{p,m,\alpha s}(\bsq))$ is a set of $n = 2^m$  points $\bsy^{(0)}, \ldots, \bsy^{(2^m-1)} \in [0,1)^s$ such that 
	\begin{align*}
	\bsy^{(k)} = \mathscr{D}_\alpha(\bsx^{(k)})
	,
	\end{align*}
	where $\{\bsx^{(k)}\}_{k=0}^{2^m-1} \in [0,1)^{\alpha s}$ are the points of a polynomial lattice point set $P_{p,m, \alpha s}(\bsq)$.
	A QMC rule using this point set is called an \emph{interlaced polynomial lattice rule} (of order $\alpha$) with generating vector $\bsq$ and modulus~$p$.
\end{definition}	

To analyse the error we will make use of the dual of the point set.
First we need to define vectors which have a specified support.
Therefore, define for an integer vector $\bsk$ the function $\supp(\bsk) := \{ j : k_j \ne 0 \}$ where the index $j$ ranges over the dimensions of $\bsk$.
To range over all $s$-dimensional vectors with support on the set $\setu$ we write $\bsk_\setu \in \N^s_\setu$.

\begin{definition}[dual of polynomial lattice point set]
	\label{def:dual-PLR}
	Given $k \in \N_0$ with binary expansion $k = \kappa_0 + \kappa_1 2 + \cdots+ \kappa_{a-1} 2^{a-1}$ define the associated truncated polynomial
	\begin{align*}
	  (\tr_m(k))(\x)
	  &:=
	  \kappa_0+ k_1 \x+ \cdots + \kappa_{m-1} \x^{m-1}
	\end{align*}
	where $\kappa_a = \cdots = \kappa_{m-1} =0$ if $a < m$.
	For $\bsk \in \N^s_0$ define $\tr_m(\bsk):= (\tr_m(k_1), \ldots, \tr_m(k_s))$. 
	The \emph{dual} of the polynomial lattice point set $P$ with modulus $p$ and $\deg(p)=m$ and generating vector $\bsq \in \mathscr{G}_m^s$ is defined by
	\begin{align*}
	  P^\perp
	  &:=
	  \big\{ \bsk \in \N_0^s: \tr_m(\bsk) \cdot \bsq \equiv 0 \pmod{p} \big\}
	  \subseteq
	  \N_0^s
	  ,
	\end{align*}
	and the \emph{dual with support $\setu$} by
	\begin{align*}
	  P^\perp_\setu
	  &:=
	  \big\{ \bsk_\setu \in \N^s_\setu \subset \N_0^s : \tr_m(\bsk_\setu) \cdot \bsq_\setu \equiv 0 \pmod{p} \big\}
	  \subset
	  P^\perp
	  .
	\end{align*}
\end{definition}

Since interlacing reduces $\alpha$ dimensions to a single dimension, the definition of the dual with support $\setu$ for an interlaced polynomial lattice point set needs to keep track of its source dimensions.
Therefore, for $\setv \subseteq \{1:\alpha s\}$, we define, see also \cite[Equation~(3.26)]{DKLNS14},
\begin{align*}
  \setu_\alpha(\setv)
  &:=
  \big\{ \lceil j / \alpha \rceil : j \in \setv \big\} \subseteq \{1:s\}
  ,
\end{align*}
which tells us where the source dimensions end up in the interlaced point set.

\begin{definition}[dual of interlaced polynomial lattice point set]
	\label{def:dual-IPLR}
	Define the \emph{digit interlacing function for non-negative integers} $\mathscr{E}_\alpha: \N_0^\alpha \to \N_0$ with interlacing factor $\alpha \in \N$ by
	\begin{align*}
	  \mathscr{E}_\alpha(k_1, \ldots, k_\alpha) 
	  &:= 
	  \sum_{i=0}^{\infty} \sum_{j=1}^\alpha
	  \kappa_{i,j} \, 2^{i \alpha+j-1}
	  ,
	\end{align*}
	where $k_j = \kappa_{0,j} + \kappa_{1,j} 2 +  \kappa_{2,j} 2^2+\cdots$ for $j = 1,\ldots, \alpha$ and, in case the number of arguments is a multiple of $\alpha$, define $\mathscr{E}_\alpha: \N_0^{\alpha s} \to \N_0^s$ by
	\begin{align*}
	  \mathscr{E}_\alpha(k_1, \ldots, k_{\alpha s})
	  &:=
	  (\mathscr{E}_\alpha(k_1, \ldots, k_\alpha), \ldots, \mathscr{E}_\alpha(k_{(s-1)\alpha+1}, \ldots, k_{s\alpha}))
	  .
	\end{align*}
	The \emph{dual} of the interlaced polynomial lattice point set $\mathscr{D}_\alpha(P_{p,m,\alpha s}(\bsq))$ is defined by
	\begin{align*}
	  (\mathscr{D}_\alpha(P_{p,m,\alpha s}(\bsq)))^\perp
	  &:=
	  \big\{
	    \mathscr{E}_\alpha(\bsk) \in \N_0^s : \bsk = (k_1,\ldots, k_{\alpha s}) \in (P_{p,m,\alpha s}(\bsq))^\perp \subseteq \N_0^{\alpha s}
	  \big\}
	  \subseteq
	  \N_0^s
	  ,
	\end{align*}
	where $(P_{p,m,\alpha s}(\bsq))^\perp$ is the dual of $P_{p,m,\alpha s}(\bsq)$ as given in \RefDef{def:dual-PLR}.
	The \emph{dual with support $\setu$} is defined by
	\begin{multline*}
	  (\mathscr{D}_\alpha(P_{p,m,\alpha s}(\bsq)))^\perp_\setu
	  :=
	  \Big\{
	    \bsh_\setu \in \N^s_\setu \subset \N_0^s
	    :
	    \bsh_\setu = \mathscr{E}_\alpha(\bsk)
	    \\ \text{ for which }
	    \bsk = (k_1,\ldots, k_{\alpha s}) \in 
	    \bigcup_{\substack{\setv \subseteq \{1:\alpha s\} \\ \text{s.t.\ } \setu_\alpha(\setv) = \setu}} (P_{p,m,\alpha s}(\bsq))^\perp_\setv
	  \Big\}
      \subset
      (\mathscr{D}_\alpha(P_{p,m,\alpha s}(\bsq)))^\perp
	  .
	\end{multline*}
\end{definition}

In order to state a bound on the worst-case error we still need to introduce a weight function which measures the importance of the $k$th Walsh basis functions and which will provide a link to the space $H_{\alpha,s}([0,1]^s)$.
For $k \in \N$ with binary expansion $k = \kappa_1 2^{m_1-1} + \kappa_2 2^{m_2-1}+ \cdots +\kappa_{v} 2^{m_v-1}$ such that $m_1 > m_2 > \cdots > m_v > 0$ define 
\begin{align*}
	\mu_\alpha(k)
	:=
	\sum_{i =1}^{\min(\alpha, v)} m_i
	,
\end{align*}  
and $\mu_\alpha(0)=0$. For $\bsk \in \N_0^s$ we define $\mu_\alpha(\bsk) := \sum_{j=1}^s \mu_\alpha(k_j)$.
We can now state a first bound on the worst-case error for interlaced polynomial lattice rules in the space $H_{\alpha,s}([0,1]^s)$.

\begin{proposition}\label{prop:wce-uSob-unit-cube}
	For any $\alpha \in \N$, with $\alpha \ge 2$, we have
\begin{align*}
  e_\wor(\mathscr{D}_\alpha(P_{p,m,\alpha s}(\bsq)); H_{\alpha,s}([0,1]^s))
  &\le
  \sum_{\emptyset \ne \setv \subseteq \{1:\alpha s\}}
  (2^{\alpha (\alpha-1)} \widehat{C}_\alpha)^{|\setu_\alpha(\setv)|/2}
  \sum_{\bsk_\setv \in (P_{p,m,\alpha s}(\bsq))^\perp_\setv}
  2^{-\alpha \mu_1(\bsk_\setv)}
  ,
\end{align*}
	with
	\begin{align}\label{eq:def:Chat-alpha}
		\widehat{C}_\alpha
		:= 
		\max_{1 \le \nu \le \alpha}
		\left\{ \sum_{\tau = \nu}^\alpha\frac{C_\tau^2}{2^{2(\tau - \nu)}} + \frac{2 \, C_{2\alpha}}{2^{2(\alpha-\nu)}}  \right\} 
		,
	\end{align}
	where $C_1 := 2^{-1}$ and $C_\tau := (5/3)^{\tau-2} \, 2^{-\tau}$ for $\tau \ge 2$.	
\end{proposition}
\begin{proof}
The reproducing kernel for the space $H_{\alpha,s}([0,1]^s)$ is well known and can be found, e.g., in \cite{BD09}. For our unweighted tensor product space it is as follows
\begin{align*}
  K_{\alpha,s}(\bsy,\bsy')
  &:=
  \prod_{j=1}^s \left( 1 + \sum_{\tau=1}^\alpha \frac{B_\tau(y_j)}{\tau!}\frac{B_\tau(y'_j)}{\tau!} + (-1)^{\alpha+1} \frac{B_{2\alpha}(|y_j-y'_j|)}{(2\alpha)!} \right)
  ,
\end{align*}
where $B_\tau$ is the Bernoulli polynomial of degree $\tau \in \N$.
We already gave the inner product for this space in~\eqref{eq:ip-u-unitcube} and~\eqref{eq:ip-u-unitcube-s}.
We will expand the kernel in a double Walsh series, see, e.g., \cite{BD09,God15}.
For $\bsk, \bsell \in \N_0^s$ the $(\bsk, \bsell)$th Walsh coefficient is defined by
\begin{align*}
  \widehat{K}_{\alpha,s}(\bsk, \bsell)
  := 
  \int_{[0,1)^s} \int_{[0,1)^s}
  K_{\alpha,s}(\bsy, \bsy') \,
  \overline{\wal_\bsk(\bsy)} \,
  \wal_\bsell(\bsy')
  \rd \bsy
  \rd \bsy'
  .
\end{align*}
From, e.g., \cite[Theorem~13]{BD09}, we have, with $\{\bsy^{(i)}\}_{i=0}^{2^m-1} \in [0,1)^s$ the points of the interlaced polynomial lattice rule,
\begin{multline*}
  (e_\wor(\mathscr{D}_\alpha(P_{p,m,\alpha s}(\bsq)); H_{\alpha,s}([0,1]^s)))^2
  =
  -1 + \frac{1}{2^{2m}} \sum_{i, i' = 0}^{2^m-1} K_{\alpha,s}(\bsy^{(i)}, \bsy^{(i')})
  \\
  =
  -1 + \sum_{\bsk, \bsell \in \N_0^s} \widehat{K}_{\alpha,s}(\bsk, \bsell) \,
     \frac{1}{2^m} \sum_{i=0}^{2^m-1} \wal_\bsk(\bsy^{(i)}) \,
     \frac{1}{2^m} \sum_{i'=0}^{2^m-1} \overline{\wal_\bsell(\bsy^{(i')})}
  .
\end{multline*}
Using the ``character property'', see, e.g., \cite[Lemma~1]{God15}, we have
\begin{align*}
  \frac1{2^m} 
  \sum_{i=0}^{2^m-1} \wal_\bsk(\bsy^{(i)})
  &=
  \begin{cases*}
    1,  & if $\bsk \in \mathscr{D}_\alpha(P_{p,m,\alpha s}(\bsq)))^\perp$, \\
    0,  & otherwise.
  \end{cases*}
\end{align*}
Hence,
\begin{align*}
  (e_\wor(\mathscr{D}_\alpha(P_{p,m,\alpha s}(\bsq)); H_{\alpha,s}([0,1]^s)))^2
  &=
  -1
  +
  \sum_{\bsk,\bsell \in (\mathscr{D}_\alpha(P_{p,m,\alpha s}(\bsq)))^\perp}
  \widehat{K}_{\alpha,s}(\bsk, \bsell)
  .
\end{align*}
Using~\cite[Lemma~14 and Equation~(13) together with Proposition~20]{BD09} we have $\widehat{K}_{\alpha,s}(\bsk, \bsell) = 0$ if $\supp(\bsk) \ne \supp(\bsell)$.
When both $\bsk$ and $\bsell$ are $\bszero$ we have $\widehat{K}_{\alpha,s}(\bszero, \bszero) = 1$.
Otherwise, for $\setu \ne \emptyset$ and $\bsk_\setu, \bsell_\setu \in \N^s_\setu = \{ \bsk \in \N_0^s : \supp(\bsk) = \setu \}$ we have
\begin{align*}
  \left|\widehat{K}_{\alpha,s}(\bsk_\setu, \bsell_\setu)\right|
  \le
  \widehat{C}_\alpha^{|\setu|} \, 2^{-\mu_\alpha(\bsk_\setu) - \mu_\alpha(\bsell_\setu)}
  ,
\end{align*}
with $\widehat{C}_\alpha$ defined in~\eqref{eq:def:Chat-alpha}.
Hence
\begin{align*}
  (e_\wor(\mathscr{D}_\alpha(P_{p,m,\alpha s}(\bsq)); H_{\alpha,s}([0,1]^s)))^2
  &\le
  \sum_{\emptyset \ne \setu \subseteq \{1:s\}}
  \widehat{C}_\alpha^{|\setu|}
  \sum_{\bsk_\setu, \bsell_\setu \in (\mathscr{D}_\alpha(P_{p,m,\alpha s}(\bsq)))^\perp_\setu}
  2^{-\mu_\alpha(\bsk_\setu) - \mu_\alpha(\bsell_\setu)}
  \\
  &=
  \sum_{\emptyset \ne \setu \subseteq \{1:s\}}
  \widehat{C}_\alpha^{|\setu|}
  \left(
  \sum_{\bsk_\setu \in (\mathscr{D}_\alpha(P_{p,m,\alpha s}(\bsq)))^\perp_\setu}
  2^{-\mu_\alpha(\bsk_\setu)}
  \right)^2
  \\
  &=
  \sum_{\emptyset \ne \setu \subseteq \{1:s\}}
  \widehat{C}_\alpha^{|\setu|}
  \left(
  \sum_{\substack{\setv \subseteq \{1:\alpha s\} \\ \setu_\alpha(\setv) = \setu}} \;
  \sum_{\bsk_\setv \in (P_{p,m,\alpha s}(\bsq))^\perp_\setv}
  2^{-\mu_\alpha(\mathscr{E}_\alpha(\bsk_\setv))}
  \right)^2
  .
\end{align*}
Thus, taking the square root on both sides and using $\left(\sum_{j} a_j\right)^{1/2} \le \sum_{j} |a_j|^{1/2}$ on the right hand side, we obtain
\begin{align*}
  e_\wor(\mathscr{D}_\alpha(P_{p,m,\alpha s}(\bsq)); H_{\alpha,s}([0,1]^s))
  &\le
  \sum_{\emptyset \ne \setu \subseteq \{1:s\}}
  \widehat{C}_\alpha^{|\setu|/2}
  \sum_{\substack{\setv \subseteq \{1:\alpha s\} \\ \setu_\alpha(\setv) = \setu}} \;
  \sum_{\bsk_\setv \in (P_{p,m,\alpha s}(\bsq))^\perp_\setv}
  2^{-\mu_\alpha(\mathscr{E}_\alpha(\bsk_\setv))}
  \\
  &=
  \sum_{\emptyset \ne \setv \subseteq \{1:\alpha s\}}
  \widehat{C}_\alpha^{|\setu_\alpha(\setv)|/2}
  \sum_{\bsk_\setv \in (P_{p,m,\alpha s}(\bsq))^\perp_\setv}
  2^{-\mu_\alpha(\mathscr{E}_\alpha(\bsk_\setv))}
  .
\end{align*}
Using \cite[Lemma~3.8 and the subsequent equation]{DKLNS14} we have
\begin{align*}
  \mu_\alpha(\mathscr{E}_\alpha(\bsk_\setv))
  &\ge
  \alpha \mu_1(\bsk_\setv) - \frac{\alpha (\alpha-1)}{2} \, |\setu_\alpha(\setv)|
  ,
\end{align*}
from which the result follows.
\end{proof}

We are now in a similar situation as \cite[Equation~(3.30)]{DKLNS14} where a function $E_d(\bsq)$ is defined which is equal to the upper bound in \RefProp{prop:wce-uSob-unit-cube} with ``modified weights'', which in our case would be $\widetilde{\gamma}_\setv := (2^{\alpha (\alpha-1)} \widehat{C}_\alpha)^{|\setu_\alpha(\setv)|/2}$, and after which a fast component-by-component construction algorithm is presented.
In \cite{DKLNS14} the weights of the function space $\gamma_{\setu_\alpha(\setv)}$ are also present in $\widetilde{\gamma}_\setv$, but in the unweighted setting here they are all~$1$.
Hence we can pull out the modified weights using $\widetilde{\gamma}_\setv \le 2^{\alpha(\alpha-1)s/2}$ which holds for all $\setv \subseteq \{1:\alpha s\}$ since $\widehat{C}_\alpha < 1$, see \cite[Table~1 for $q=2$]{BD09}.
The following proposition now follows immediately from~\cite[Theorem~3.9]{DKLNS14} by using $\widetilde{\gamma}_\setv = 1$ for all~$\setv$.
We note that using the actual weights $\widetilde{\gamma}_\setv$ would improve the result, but would not change the complexity for the MDFEM so we prefer this simpler result.

\begin{proposition}\label{prop:bound-E}
	For any $\alpha \in \N$, with $\alpha \ge 2$, let $p$ be an irreducible polynomial with $\deg(p) = m$.
	For $d \in \N$ and $\bsq \in \mathscr{G}_m^d$ define
	\begin{align*}
	  E_d(\bsq)
	  &:=
	  \sum_{\emptyset \ne \setv \subseteq \{1:d\}} \;
	  \sum_{\bsk_\setv \in (P_{p,m,d}(\bsq))^\perp_\setv}
      2^{-\alpha \mu_1(\bsk_\setv)}
      =
	  \sum_{\bszero \ne \bsk \in (P_{p,m,d}(\bsq))^\perp}
      2^{-\alpha \mu_1(\bsk)}
      .
	\end{align*}
	A generating vector $\bsq^* = (q_1^*, q_2^*, \ldots,q_d^*) \in \mathscr{G}_m^d$ can be constructed using a CBC approach for $d=1,2,\ldots$, minimizing $E_d(\bsq)$ in each step, such that 
	\begin{align*}
	E_d(\bsq^*)
	&\le 
	\left(
	\frac{2}{2^m-1}
	\right)^\lambda
	\left(
	\sum_{\emptyset \neq \setv \subseteq \{1:d\}}
	\frac{1}
	{(2^{\alpha/\lambda} -2)^{|\setv|}}
	\right)^\lambda
	=
	\left(
	\frac{2}{2^m-1}
	\right)^\lambda
	\left[
	\left(1+\frac{1}{2^{\alpha/\lambda}-2}\right)^d - 1
	\right]^\lambda
	,
	\end{align*}
	for all  $\lambda \in [1,\alpha)$.
\end{proposition}

Combining \RefPropTwo{prop:wce-uSob-unit-cube}{prop:bound-E} for $d=\alpha s$ we obtain \RefThm{thm:error-bound-IPLR} in the main text.

\bibliographystyle{plain}
\bibliography{MDFEM-lognormal}

\end{document}